\documentclass[preprint]{imsart}
\usepackage[english]{babel}
\usepackage[T1]{fontenc}
\usepackage[utf8]{inputenc}
\usepackage{lmodern}
\usepackage{calc}
\usepackage[a4paper,portrait,left=3cm,right=3cm,top=3cm,bottom=3cm]{geometry}

\usepackage{amsmath}
\usepackage{amsfonts}
\usepackage{amssymb}
\usepackage{amsthm}
\usepackage{stmaryrd}
\usepackage{euscript}
\usepackage{xspace}
\usepackage{bm}
\usepackage{enumerate}
\usepackage{hyperref}
\usepackage{graphicx}
\usepackage{subcaption}
\usepackage[normalem]{ulem}
\usepackage[numbers]{natbib}

\startlocaldefs

\newtheorem{definition}{Definition}%[section]

\newtheorem{lemma}{Lemma}[section]
\newtheorem{theorem}{Theorem}
\newtheorem{proposition}[theorem]{Proposition}

\newtheorem*{lemma*}{Lemma}

\theoremstyle{definition}

\newtheorem{remark}{Remark}

\numberwithin{equation}{section}

\DeclareMathOperator{\supp}{supp}

\DeclareMathOperator{\tr}{tr}
\DeclareMathOperator{\Id}{Id}

\newcommand{\e}{\ensuremath{\mathrm{e}}\xspace}

\newcommand{\R}{\ensuremath{\mathbb{R}}\xspace}

\newcommand{\N}{\ensuremath{\mathbb{N}}\xspace}
\newcommand{\Z}{\ensuremath{\mathbb{Z}}\xspace}

\newcommand{\Hbb}{\ensuremath{\mathbb{H}}\xspace}

\newcommand{\Tbb}{\ensuremath{\mathbb{T}}\xspace}

\newcommand{\Vbb}{\ensuremath{\mathbb{V}}\xspace}

\newcommand{\Nb}{\ensuremath{\mathbf{N}}\xspace}

\newcommand{\Bi}{\ensuremath{\mathcal B}\xspace}

\newcommand{\Di}{\ensuremath{\mathcal D}\xspace}
\newcommand{\Ei}{\ensuremath{\mathcal E}\xspace}

\newcommand{\Gi}{\ensuremath{\mathcal G}\xspace}
\newcommand{\Hi}{\ensuremath{\mathcal H}\xspace}
\newcommand{\Ii}{\ensuremath{\mathcal I}\xspace}
\newcommand{\Ji}{\ensuremath{\mathcal J}\xspace}

\newcommand{\Li}{\ensuremath{\mathcal L}\xspace}

\newcommand{\Ni}{\ensuremath{\mathcal N}\xspace}

\newcommand{\Ti}{\ensuremath{\mathcal T}\xspace}
\newcommand{\Vi}{\ensuremath{\mathcal V}\xspace}
\newcommand{\Wi}{\ensuremath{\mathcal W}\xspace}
\newcommand{\Xii}{\ensuremath{\mathcal X}\xspace}

\newcommand{\dimH}{\ensuremath{{\dim}_{\text{\normalfont\tiny H}}}\xspace}
\newcommand{\dimP}{\ensuremath{{\dim}_{\text{\normalfont\tiny P}}}\xspace}

\newcommand{\dimBu}{\ensuremath{\overline{\dim}_{\text{\normalfont\tiny B}}}\xspace}

\newcommand{\vset}{\ensuremath{{\emptyset}}\xspace}

\newcommand{\pth}[1]{(#1)}
\newcommand{\pthb}[1]{\bigl(#1\bigr)}
\newcommand{\pthB}[1]{\Bigl(#1\Bigr)}
\newcommand{\pthbb}[1]{\biggl(#1\biggr)}

\newcommand{\bkt}[1]{[#1]}
\newcommand{\bktb}[1]{\bigl[#1\bigr]}

\newcommand{\brc}[1]{\{#1\}}
\newcommand{\brcb}[1]{\bigl\{#1\bigr\}}
\newcommand{\brcB}[1]{\Bigl\{#1\Bigr\}}
\newcommand{\brcbb}[1]{\biggl\{#1\biggr\}}

\newcommand{\bk}[1]{\langle#1\rangle}

\newcommand{\dt}{\ensuremath{\mathrm d}\xspace} % dt integrale
\newcommand{\eqdef}{:=}

\newcommand{\ivoo}[1]{\ensuremath{(#1)}}
\newcommand{\ivoob}[1]{\ensuremath{\bigl(#1\bigr)}}

\newcommand{\ivof}[1]{\ensuremath{(#1]}}
\newcommand{\ivofb}[1]{\ensuremath{\bigl(#1\bigr]}}

\newcommand{\ivfo}[1]{\ensuremath{[#1)}}
\newcommand{\ivfob}[1]{\ensuremath{\bigl[#1\bigr)}}

\newcommand{\ivff}[1]{\ensuremath{[#1]}}
\newcommand{\ivffb}[1]{\ensuremath{\bigl[#1\bigr]}}

\newcommand{\iivff}[1]{\ensuremath{\llbracket#1\rrbracket}}

\newcommand{\abs}[1]{\lvert#1\rvert}

\newcommand{\ceil}[1]{\lceil#1\rceil}

\renewcommand{\Pr}{\ensuremath{\mathbb P}\xspace}

\newcommand{\prB}[2][]{\mathbb{P}#1\pthB{#2}}

\newcommand{\esp}[2][]{\mathbb{E}#1\bkt{#2}}

\newcommand{\varb}[2][]{\mathrm{Var}#1\pthb{\hspace{1pt}#2\hspace{1pt}}}

\newcommand{\varbb}[2][]{\mathrm{Var}#1\pthbb{#2}}

\newcommand{\varcb}[3][]{\mathrm{Var}#1\pthb{\hspace{1pt}#2\bigm|#3\hspace{1pt}}}

\newcommand{\pthc}[2]{\pth{\hspace{1pt}#1\hspace{1.5pt}|\hspace{1.5pt}#2\hspace{1pt}}}
\newcommand{\pthcb}[2]{\pthb{\hspace{1pt}#1\bigm|#2\hspace{1pt}}}

\newcommand{\pthcbb}[2]{\pthbb{#1\biggm|#2}}

\newcommand{\indi}{\ensuremath{\mathbf{1}}\xspace}
\newcommand{\eps}{\varepsilon}

\DeclareMathOperator{\cov}{Cov}

\newcommand{\vsp}{\vspace{.15cm}}

\endlocaldefs

\begin{document}

\begin{frontmatter}

%% Title of the paper %%
\title{Singularities of stable super-Brownian motion}
\runtitle{Singularities of stable super-Brownian motion}

%% Authors & address %%
\begin{aug}
\author{\fnms{Paul} \snm{Balança}\thanksref{t2}\ead[label=e1]{paul.balanca@gmail.com}}\and%
\author{\fnms{Leonid} \snm{Mytnik}\thanksref{t2}\ead[label=e2]{leonid@ie.technion.ac.il}}%
% \ead[label=u1,url]{http://www.foo.com}
\thankstext{t2}{Research supported by the Israel Science Foundation grant 1325/14.}%
\runauthor{P. Balança \& L. Mytnik}%
\affiliation{Technion}%
\address{Faculty of Industrial Engineering and Management\\%
Technion Israel Institute of Technology\\%
Haifa 32000, Israël\\%
\printead{e1,e2}}
\end{aug}

\begin{abstract}
We investigate in this work the spectrum of singularities of super-Brownian motion with stable branching. The main purpose is to provide a uniform description of the latter in high dimension $d\geq\tfrac{2}{\gamma-1}$, presenting the singularities existing at every time $t$ and characterising as well the set of random times at which singularities of higher order appear. In lower dimensions, we give a partial description of the singularities which complement the recent study of the density by \citet{Mytnik.Wachtel-2015}.
\end{abstract}



\end{frontmatter}

%!TEX root = article.tex
% mainfile: article.tex
%%%%%%%%%%%%%%%%%%%%%%%%%%%%%%%%%%%%%%%%%%%%%%%%%%%%%%%%%%%%%%%%%%%%%%%%
%% Introduction
%%%%%%%%%%%%%%%%%%%%%%%%%%%%%%%%%%%%%%%%%%%%%%%%%%%%%%%%%%%%%%%%%%%%%%%%

\section{Introduction and main results}

A so-called stable super-Brownian motion  $X_t(\dt x)$ in dimension $d\geq 1$ with branching index $\gamma\in\ivof{1,2}$ is a finite measure-valued Markov process related to the log-Laplace equation
\begin{align}  \label{eq:log_laplace}
  \frac{\partial }{\partial t} u(t,x) = \Delta u(t,x) + a u(t,x) - bu(t,x)^\gamma,
\end{align}
where $a\in\R$ and $b>0$ are any fixed constants and $\Delta$ is a $d$-dimensional Laplacian. We will sometimes call this process $\gamma$-stable super-Brownian motion
or $\gamma$-SBM.
The underlying motion of $X$ is characterised by the operator $\Delta$ appearing in \eqref{eq:log_laplace}, and therefore, corresponds to Brownian motion. Note that the use of a fractional Laplacian $\Delta_\alpha$
instead of $\Delta$  leads to a more general superprocess with symmetric stable motion.
Some regularity properties of such superprocessess in dimension $d=1$ were investigated in \cite{Fleischmann-1988,Fleischmann.Mytnik.ea-2010}). Moreover, the continuous state branching mechanism of the superprocess is described by the function
$$\psi(u)=-au+bu^\gamma,$$
and such branching mechanism belongs to the domain of attraction of a stable law. In the rest of this work, we will assume that $a=0$, $b=1$ i.e. the branching is critical, and denote by $\zeta$ the extinction time of the super-process $X_t(\dt x)$, i.e. $\zeta=\inf\brc{t:X_t(\R^d) = 0}$. In the case of critical branching that we consider, $\zeta$ is known to be almost surely finite.
% whenever the branching is (sub-)critical).
%Moreover,
Let  $M_f(\R^d)$ be the space of finite measures on $\R^d$ equipped with
weak topoogy.
In what follows, we will also assume that $X_t(\dt x)$ starts from a deterministic measure $\mu$ with finite mass, i.e. $\mu\in M_f(\R^d)$. \vsp

A large literature has investigated the fractal geometry of superprocesses, and in particular super-Brownian motion, in high dimension. Whenever $d\geq\tfrac{2}{\gamma-1}$, the measure-valued process $X_t(\dt x)$, at fixed time, is known to be a singular with respect to the Lebesgue measure. Moreover, in the case of stable branching mechanism, the Hausdorff and packing dimensions of its support are well known: for any fixed $t>0$,
\begin{align}  \label{eq:sbm_support_fixt}
  \dimH \supp X_t = \dimP \supp X_t = d\wedge \frac{2}{\gamma-1}, \quad\text{a.s. on the event }\quad \brc{X_t(\R^d)>0},
\end{align}
where we use $\dimH$ (resp. $\dimP$) to denote Hausdorff (resp. packing) dimensions of the sets. We refer to \cite{Dawson.Hochberg-1979,Dawson.Perkins-1991} for the quadratic branching, and \cite{Duquesne.LeGall-2005} for the more general stable case (even though not directly stated for the packing dimension, the result follows easily from the arguments presented in \cite{Duquesne.LeGall-2005}).

This result has later been extended by \citet{Delmas-1999} and \citet{Duquesne.LeGall-2005} to more general branching mechanisms, and to more general set of times.
In particular, for $\gamma$-stable super-Brownian motion $X$, it was proved in \citet{Delmas-1999}, that for any non-empty closed set $F\subset\ivoo{0,\infty}$,
\begin{align}  \label{eq:sbm_support_fixt_ext}
  \dimH \overline{\bigcup_{t\in F} \supp X_t } = d\wedge \pthB{ \frac{2}{\gamma-1} + 2\dimH F } \quad\text{on the event}\quad\brcb{F\subset\ivoo{0,\zeta}}.
\end{align}

%In the case of a quadratic branching, i.e. the super-Brownian motion (SBM), the geometry of the former has been thoroughly studied.
The question of existing of Hausdorff gauge function of the measure $X_t(dx)$
was also thoroughly studied in the literature. For example, in the case of quadratic  branching  ($\gamma=2$),
\citet{Perkins-1989} have determined the exact Hausdorff gauge function $\phi$ of the measure $X_t(\dt x)$ whenever $d\geq 3$: he showed that $\phi(r)=r^2\log\log 1/r$. The (non)-existence of a packing gauge function has also been investigated by \citet{LeGall.Perkins.ea-1995}. Several of these results have recently been extended to more general branching mechanisms by \citet{Duquesne-2009,Duquesne.Duhalde-2014}. In dimension $d\geq 4$ and the case of quadratic  branching, \citet{Tribe-1989} has also obtained a uniform extension of the properties \eqref{eq:sbm_support_fixt} and \eqref{eq:sbm_support_fixt_ext}.\vsp

In this work, we are interested in studying the fine structure of the measure $X_t(\dt x)$ in terms of exceptional mass distribution. More specifically, we aim to characterise the asymptotic behaviour of $X_t(B(x,r))$ as $r\rightarrow 0$. It is known that at a typical point, the stable super-Brownian motion behaves ``well''. In particular, one can show that for any fixed $t>0$
% \footnote{LM I did not find ref; did you? if  not we should formulate it
% differently\\
% Or maybe at least lower bound follows from  Frostman lemma?\\
% PB: Could not find it as well. The lower bound is a corollary of the usual proof of the Hausdorff dimension. Should we just give this bound?\\
% PB2: The proof of the packing dimension completely follows from the paper of
% Duquesne and Le Gall. And in addition, we provide an additional proof of it for
% Theorem 3.\\
% For the local dimension, it seemed to me that for the readability it is better to just give the short proof in Appendix. Thanks for finding the trick to get the upper bound part!}
\begin{align}  \label{eq:weak_asymp}
  \lim_{r\rightarrow 0} \frac{\log X_t(B(x,r))}{\log r} = \frac{2}{\gamma-1}\quad X_t(\dt x)\text{-a.e.} \quad \Pr_\mu\text{-a.s.}
\end{align}
where the above limit at $x$, if exists, is sometimes called the \emph{local dimension} of the measure $X_t$ at $x$ (see for instance \cite[Chap. 10]{Falconer-1997}). In fact, surprisingly, we could not find statement \eqref{eq:weak_asymp} in the literature on superprocesses. However, it easily follows by putting together a set of well-known results in fractal geoemetry and standard proofs on stable super-Brownian motion (we refer to Appendix \ref{sec:appendix2} for a short proof).

Nevertheless, there may exist points at which the local mass $X_t(B(x,r))$ can be exceptionally large (or thin). Such a behaviour has been investigated by \citet{Perkins.Taylor-1998} on quadratic super-Brownian motion, proving the existence of points with exceptional thin mass. On the other hand, the former also proved that SBM has no points of large masses: for every $t>0$,
\begin{align}  \label{eq:unif_sbm_holder}
  \Pr_\mu\text{-a.s.}\quad\forall x\in\supp(X_t);\quad\liminf_{r\rightarrow 0} \frac{\log X_t(B(x,r))}{\log r} = 2.
\end{align}

The purpose of this work is to investigate the existence of such points with exceptional large masses on the super-Brownian motion with stable branching. For that purpose, we recall the notion of pointwise Hölder exponent of a measure: at every $x\in\R^d$, one defines
\begin{align}  \label{eq:def_pointwise_exponent}
  \alpha_{X_t}(x) \eqdef \liminf_{r\rightarrow 0} \frac{\log X_t\pthb{B(x,r)}}{\log r}.
\end{align}
Informally, the latter describes the asymptotic order of the largest masses appearing around $x$. Note that previous definition is only licit whenever $\alpha_{X_t}(x)\in\ivfo{0,d}$. In order to characterise higher orders of regularity, one needs to study the variations of the density of $X_t(\dt x)$, as recently done by \citet{Mytnik.Wachtel-2015} on the density of one dimensional superprocesses.

The study of the pointwise Hölder structure of measures (or stochastic processes) is usually known as \emph{multifractal analysis}, and focuses on determining the so-called \emph{multifractal spectrum}, or \emph{spectrum of singularities}, defined by:
\begin{align}  \label{eq:def_spectrum}
  \forall h\geq 0;\quad d_{X_t}(h,V)\eqdef \dimH E(h,X_t)\cap V\quad\text{where } E(h,X_t) \eqdef \brcb{x\in\R^d : \alpha_{X_t}(x) = h},
\end{align}
and $V$ is any non-empty open set in $\R^d$.
Initially introduced by \citet{Frisch.Parisi-1985} in their study of turbulence, this formalism has proved to be relevant to investigate a much larger class of random (or deterministic) measures. In the recent probability literature, \citet{Dembo.Peres.ea-2000} and \citet{Shieh.Taylor-1998} have for instance characterised the multifractal structure of the occupation measure of Brownian motion (and stable processes). This formalism has also been extended to the study of the regularity of function and sample paths, leading to several important recent works in probability: Lévy processes \cite{Jaffard-1999,Durand-2009,Balanca-2014}, super-critical Galton--Watson trees \cite{Moerters.Shieh-2002,Moerters.Shieh-2008}, density of super-processes \cite{Mytnik.Wachtel-2015}, fragmentation processes and continuous random trees \cite{Berestycki-2003,Berestycki.Berestycki.ea-2007,Balanca-2015b} to name but a few.\vsp

In our main following results, we present a uniform description of the multifractal structure of stable super-Brownian motion, the latter combining a component of the spectrum of singularities existing at every time, and the existence of exceptional times where exceptional large masses appear.
\begin{theorem}  \label{th:sbm_spectrum1}
  Suppose $\gamma\in\ivoo{1,2}$, $\mu\in M_f(\R^d)$ and under $\Pr_\mu$, $X_t(\dt x)$ is a stable super-Brownian motion starting from $\mu$. The following statements are then satisfied $\Pr_\mu$-a.s.
  \begin{enumerate}[(a)]
    \item Assuming $d\geq 2$, the spectrum of singularities of the measure $X_t(\dt x)$ is given by:
    \begin{align}  \label{eq:sbm_spectrum}
      \quad \dimH \,E(h,X_t)\cap V = \gamma h - 2,\quad\forall h\in\ivffb{\tfrac{2}{\gamma}, \tfrac{2}{\gamma-1}}\cap\ivfob{0,d},
    \end{align}
    for any $t>0$ and open set $V\subset\R^d$ such that $X_t(V)>0$.
    \item For any dimension $d\geq 1$, the spectrum of times where "small exponents" appear, is given by:
    \begin{align}  \label{eq:sbm_spectrum-times}
      \dimH \brcb{t>0 : E(h,X_t) \neq \vset } = \frac{\gamma h}{2}, \quad \forall h\in\ivfob{0,\tfrac{2}{\gamma}}\cap\ivfob{0,d},
    \end{align}
    Moreover, for every $t>0$ and any $h\in\ivfob{0,\tfrac{2}{\gamma}}\cap\ivfob{0,d}$, $E(h,X_t)$ is either empty or has zero Hausdorff dimension.
  \end{enumerate}
\end{theorem}

\begin{remark}
  Theorem~\ref{th:sbm_spectrum1} provides a full characterisation of the large masses of stable super-Brownian motion whenever $d\geq\tfrac{2}{\gamma-1}$.
It is derived in~(a) that at every level $t\in\ivoo{0,\zeta}$, a full of spectrum of singularities exists between the "typical"
% behaviour \eqref{eq:weak_asymp}
$h=\tfrac{2}{\gamma-1}$ (see~\eqref{eq:weak_asymp}) and the smaller exponent $h=\tfrac{2}{\gamma}$. Moreover in~(b) it is shown that there are also points with H\"older exponenet less than $2/\gamma$, and the spectrum of times at which such points appear is derived there.

  The dimension $d\geq\tfrac{2}{\gamma-1}$ corresponds to the transient case of stable
super-Brownian motion. In this case the results of  Theorem~\ref{th:sbm_spectrum1}  are clearly consistent with the multifractal structure of stable trees presented in \cite{Balanca-2015b}; note on this ocasion that the proof of Theorem~\ref{th:sbm_spectrum1}  relies heavily on the tools introduced in \cite{Balanca-2015b}.
\end{remark}
\begin{remark}
 % Interestingly, the results presented in
As for results of Theorem~\ref{th:sbm_spectrum1}
%are not restricted to the purely transient case, but also apply to
for dimensions $d<\tfrac{2}{\gamma-1}$, they are also very interesting.
 As pointed out above, the limitations of the pointwise exponent \eqref{eq:def_pointwise_exponent} imply that the results are not as
 complete as in dimension $d\geq\tfrac{2}{\gamma-1}$. Nevertheless, these results also bring information on the behaviour of stable super-Brownian motion and its singularities of low orders. More precisely, it is known (see \citet{Fleischmann-1988}) that if $d<\tfrac{2}{\gamma-1}$, at every fixed time $t$, $X_t(\dt x)$ admits a density $X_t(x)$ (we conjecture the latter exists for all times $t>0$ at the exception of times of jumps of the process). Moreover, according to the work of \citet{Mytnik.Perkins-2003}, whenever $2\leq d<\tfrac{2}{\gamma-1}$, $X_t(x)$ is everywhere unbounded. The first part \eqref{eq:sbm_spectrum} of the spectrum provides a clear explanation of this fact: $2\leq d<\tfrac{2}{\gamma-1}$, then  at every $t>0$, there exists a set of singularities of $X_t(\dt x)$ with exponent $h\in\ivfob{\tfrac{2}{\gamma},d}$. Clearly, this set is dense on the set of points of positive density $X_t(x)$ and, at these points, the density can not be bounded (otherwise, one would clearly have $\alpha_{X_t}(x)\geq d$), and therefore explodes at a rate at least equal to $h-d$.

  In dimension $d=1$, Theorem~\ref{th:sbm_spectrum1} also provides interesting information on the behaviour of stable super-Brownian motion, complementing the recent work of \citet{Mytnik.Wachtel-2015}. In this case, it is known (see \cite{Mytnik.Perkins-2003,Fleischmann.Mytnik.ea-2010}) that at fixed time $t>0$, the density is Hölder continuous. Moreover,
%and surprisingly, it is also known
  it follows from results in~\cite{Mytnik.Perkins-2003}, that the density is not continuous in time-space: in fact, it explodes in open neighbourhood of every time-space point $(t,x)$. Theorem~\ref{th:sbm_spectrum1} provides an informal explanation to this counter-intuitive behaviour: according to \eqref{eq:sbm_spectrum-times}, singularities of order stricly smaller than $1$ only appear at exceptional (dense, owing to the self-similarity of stable SBM)
%   \footnote{LM how do we know it is dense?\\PB: Self-similarity gives local version of Theorem 1. Should I add it?\\
% LM2: yes I guess we should say something here about "dense" :))   I also added below:\ldots  at these "exceptional" \\PB2: I guess it should be fine like. Self-similarity is quite standard argument for proving density.}
  times , leading to an explosion of the density at these exceptional points. The previous theorem in fact implies that the Hausdorff dimension of this dense set of times is at least $\tfrac{\gamma}{2}$. On the other hand, typical times do not present any such singularity of small order, hence informally inplying the existence of a continuous density. The more complete understanding of the singularities of the one-dimensional density is the subject of on-going work.
\end{remark}
\vsp

One may observe that Theorem $2$ in \cite{Balanca-2015b} presents a stronger characterisation of the spectrum of singularities by studying the structure on any sub-tree $\Ti(F)=\cup_{a\in F}\Ti(a)$. Assuming $d\geq \tfrac{2\gamma}{\gamma-1}$, we may obtained as well an analogue uniform result on stable super-Brownian motion. For that purpose, we need to introduce the following notion (previously defined in \cite{Balanca-2015b}): a Borel set $F$ is said to satisfy a \emph{strong Frostman's lemma} if there exists a probability measure $\mu_F$ on $F$ (i.e. $\supp F\subseteq F$) such that for every $\eps>0$,
\begin{align}  \label{eq:strong_frostman}
  \exists r_0>0,\quad  \forall x\in F,\ \forall r\in\ivoo{r,r_0};\quad \mu_F\pthb{B(x,r)} \leq r^{\dimH F - \eps}.
\end{align}
Note that even though the previous assumption is stronger that the celebrated Frostman's lemma, it remains a mild assumption satisfied by a large class of fractal sets (and in particular sets with a finite Hausdorff measure for a given gauge function). Under this condition on the fractal sets considered, we can then obtained a strong uniform statement.
\begin{theorem}  \label{th:sbm_spectrum2}
  Suppose $\gamma\in\ivoo{1,2}$, $\mu\in M_f(\R^d)$ and under $\Pr_\mu$, $X_t(\dt x)$ is a stable super-Brownian motion starting from $\mu$.
  \begin{enumerate}[(a)]
    \item Assuming $d\geq \tfrac{2\gamma}{\gamma-1}$, $\Pr_\mu$-a.s., for any Borel set $F\subset\ivoo{0,\zeta}$ satisfying the \emph{strong Frostman's condition} \eqref{eq:strong_frostman} and such that $\dimH F=\dimP F$, we get:
    \begin{align}
      \forall h\in\ivffb{\tfrac{2}{\gamma}, \tfrac{2}{\gamma-1}};\quad \dimH \bigcup_{s\in F} E(h,X_s) = \gamma h - 2 + 2\dimH F;
    \end{align}

    \item Assuming $d\geq 2$, for any Borel set $F\subset\ivoo{0,\infty}$, $\Pr_\mu$-a.s., the singularities of smallest order satisfy:
    \begin{align}
      \inf_{x\in\R^d} \inf_{s\in F} \alpha_{X_s}(x) = \frac{2-2\dimP F\cap\ivoo{0,\zeta}}{\gamma}\quad\text{on the event }F\cap\ivoo{0,\zeta}\neq\vset.
    \end{align}
  \end{enumerate}
\end{theorem}
% \footnote{LM $\xi$ above is replaced by $\zeta$}
% \footnote{LM I would comment more on this result. Is there a counter examole such that in dimensions
% $d\in (\frac{2\gamma}{\gamma -1}-2+2dim(F), \frac{2\gamma}{\gamma -1})$ (1.11) does not hold (e.g. in the case of $\gamma=2$ people should have thought about it?) ? For example in the case of $F=\{t\}$ we do not have such limitation on dimensions...\\
% PB: The same bound appear in Theorem 3 and exists when $\gamma=2$ ($d\geq 4$) in the work of Tribe and Serlet. I can't think right now of a counter-example but I think random sets $F$ could lead to different lhs and rhs (at least, that is what happens on Gaussian processes). \\
% LM2 Ok let us leave it as it is}
We may note a major difference between the two previous results \emph{(a)} and \emph{(b)}: the first one is uniform in the set $F$ (a.s. for all sets) whereas the second one is not. The weaker form of the latter is natural since the largest masses only appear at exceptional levels. This typical behaviour exists as well in the case of continuous stable trees \cite[Th. 5]{Balanca-2015b}.
\citet{Moerters-2001} has also obtained a similar result by studying the speed of the fastest particles on super-Brownian motion.\vsp

In addition,
%we also observe that
as a corollary of Theorems~\ref{th:sbm_spectrum1} and \ref{th:sbm_spectrum2}, we obtain a uniform characterisation of the fractal dimension of the support of stable super-Brownian motion, extending classic results of \citet{Delmas-1999} and \citet{Tribe-1989,Serlet-1995} on the subject.
\begin{theorem}  \label{th:support_dimension}
  Suppose $\gamma\in\ivoo{1,2}$, $d\geq \tfrac{2}{\gamma-1}$ and $\mu\in M_f(\R^d)$. Then, $\Pr_\mu$-a.s.,
  \begin{align}
    \forall t\in\ivoo{0,\zeta};\quad \dimH \pth{\supp X_t} = \dimP \pth{\supp X_t} = \frac{2}{\gamma-1}.
  \end{align}
  Moreover, whenever $d\geq\tfrac{2\gamma}{\gamma-1}$, $\Pr_\mu$-a.s.
  \begin{align}
    \text{for any closed set } F\subset\ivoo{0,\zeta};\quad \dimH \overline{\bigcup_{t\in F} \supp X_t } = \frac{2}{\gamma-1} + 2\dimH F.
  \end{align}
\end{theorem}
\vsp

The rest of the paper is organised as follows: in Section~\ref{sec:notations}, we introduce
the main notations and tools, presenting in particular the \emph{Lévy snake} approach to the construction of
superprocesses and a \emph{local nondeterminism property} adapted to this setting. The proofs of the main
result restricted to excursion measures is presented in Section~\ref{sec:stable_sbm_high_dim}, dividing the former
in upper and lower bound estimates. Section~\ref{sec:spectrum_sbm} gathers the proofs of Theorems \ref{th:sbm_spectrum1}, \ref{th:sbm_spectrum2} and \ref{th:support_dimension} on stable super-Brownian motion. Finally, in Appendix~\ref{sec:appendix} is proved a technical lemma on stable trees which is a consequence of fine properties obtained in \cite{Balanca-2015b}.

\section{Notations and preliminary properties} \label{sec:notations}

We aim to study the multifractal structure of stable super-Brownian motion using an approach different from the work of \citet{Mytnik.Wachtel-2015}. The latter is based on the martingale representation of superprocesses whereas this article uses an alternative construction of the former based on the \emph{Brownian snake} (or \emph{Lévy snake} in its most general setting). Briefly, this approach has been initially introduced in the work of \citet{LeGall.LeJan-1998a} and consists in first constructing the continuous trees encoding the genealogy of a CSBP, and then, conditionally on a realization of a tree, introducing
the motion of particles (usually Brownian motion) as a stochastic process indexed by that continuous tree. As observed by \citet{Duquesne.LeGall-2005}, the latter construction allows to lift more easily the fractal geometry of continuous Lévy trees to super-processes. In our specific case, we therefore rely on our previous study \cite{Balanca-2015b} of the multifractal structure of stable trees to investigate the singularities of stable super-Brownian motion.

\subsection{Introduction to continuous stable trees and the Brownian snake}

We will start by recalling several notations and important results concerning continuous stable trees and the Brownian snake. Most of the material presented in the beginning of this section is available in much more details in the seminal articles \cite{LeGall.LeJan-1998a,LeGall-1999,Duquesne.LeGall-2002,Duquesne.LeGall-2005}.\vsp

\noindent\textbf{Continuous random stable trees.} As presented by \citet{Duquesne.LeGall-2002,Duquesne.LeGall-2005}, Lévy trees
are encoded by excursions of a continuous non-negative process $(H_u)_{u\geq 0}$ named the \emph{height process}, $H$ being simply
a reflected Brownian motion whenever $\gamma=2$. In the case that is considered in this paper $H$ arises as a certain
functional of spectrally positive stable process with exponent $\gamma\in (1,2)$. In what follows we consider just this particular case.

We denote by $\mathbf{N}^H(\dt H)$ the \emph{excursion measure}
constructed by \citet{Duquesne.LeGall-2002}, i.e. meaning that $\mathbf{N}^H$ is a measure on $C(\R_+,\R_+)$ (non-negative continuous functions)
such that $H(0)=0$ and $H(u)=0$ for every $u\geq \Li_H$ where $\Li_H\eqdef\sup\brc{u\geq 0 : H(u)>0}<\infty$ denotes the lifetime of an excursion.
We denote by $(\Lambda^a_t\,, t\geq 0)$ the local time of $H$ at level $a\geq 0$, It is well known (see e.g. Section 3.2 in  \citet{Duquesne.LeGall-2005}) that by
under excursion measure $\mathbf{N}^H$, the law of the total local time $\Lambda^a_{\Li_H}$ of excursion $H$ is characterized as follows:
\begin{align}
\label{eq:4_21_1}
   \mathbf{N}^H\pthb{1 - \e^{-\lambda\Lambda^a_{\Li_H}}} = \pthb{(\gamma-1)a + \lambda^{1-\gamma} }^{-\tfrac{1}{\gamma-1}}=: u_a(\lambda),\quad \forall \lambda\ge 0.
\end{align}
% \footnote{LM I guess it is more reasonable to put Ray-Knight discussion here?}
Originally described as a Ray--Knight theorem, this result connects the law of local time of
%Lévy trees
the height process $H$ under $\mathbf{N}^H$,
to the Laplace transform of continuous state branching processes (CSBPs) and informally states that the total mass of the local time idexed by levels
%$t\mapsto\bk{\ell^t}$
$a\mapsto \Lambda^a_{\Li_H}$
is a CSBP starting from a single individual.

Under $\mathbf{N}^H(\dt H)$, the excursion $(H_u)_{0\leq u\leq\Li_H}$ is the depth-first exploration process of a rooted \R-tree which is defined as a quotient metric space. To define it rigorously  we introduce the equivalence relation $u\sim_H v$ if and only if $d_H(u,v) = 0$, where $d_H$ denotes the following pseudo-distance:
\begin{align*}
  d_H\pthb{ u,v } = H_u + H_v - 2\min_{u\wedge v\leq w\leq u\vee v} H_w.
\end{align*}
Stable trees are then defined as a quotient metric space: $(\Ti,d) \eqdef \pthb{\ivff{0,\Li_H}/\sim_H, d_H}$. As presented in \cite[Th. 2.1]{Duquesne.LeGall-2005}, $(\Ti,d)$ is an $\R$-tree, i.e. a metric space such that for every $\sigma,\sigma'\in\Ti$
\begin{enumerate}[\it(i)]
  \item There is a unique isometry $f_{\sigma,\sigma'}$ from $\ivffb{0,d(\sigma,\sigma')}$ into $\Ti$ such $f_{\sigma,\sigma'}(0)=\sigma$ and $f_{\sigma,\sigma'}(d(\sigma,\sigma'))=\sigma'$. We set $\iivff{\sigma,\sigma'}=f_{\sigma,\sigma'}\pthb{\ivffb{0,d(\sigma,\sigma')}}$, that is the geodesic joining $\sigma$ to $\sigma'$;\vsp
  \item If $g:\ivff{0,1}\rightarrow\Ti$ is continuous injective, then $g([0,1]) = \iivff{g(0),g(1)}$.
\end{enumerate}
We refer to \cite{Dress.Moulton.ea-1996,Evans.Pitman.ea-2006,Evans-2010} for a more detailed overview on the topic of (random) \R-trees.

To summarize,  the stable tree is the tree $(\Ti,d)$ coded by function $H$ under excursion measure $\mathbf{N}^H$.
% \footnote{LM I guess we did not give the definition of $\Nb$ before? I hope the next sentence is enough?\\
% PB: It seems fine to me, as we define \Ti clearly\\
% LM2 Fine}
We denote by $\Nb(d\Ti)$  the measure on the set of equivalence classes of rooted compact \R-trees, such that $\Nb$
gives the law of the stable tree $(\Ti,d)$ under the measure $\mathbf{N}^H$.
Finally, for every time $t>0$, we denote by $\Ti(t)=\brc{\sigma\in\Ti : d(\rho,\sigma)=t}$ the level set of generation $t$. Note that for the sake of readability, we will denote by $\Bi(\sigma,r)$ the open balls in an $\R$-tree metric space $(\Ti,d)$, and use on the contrary the usual notation $B(x,r)$ when referring to balls in $\R^d$.\vsp

As presented in \cite{Duquesne.LeGall-2005}, one can construct a local time $\ell^t(\dt\sigma)$ carried by the level set $\Ti(t)$ and which, informally, represents the mass distribution of the population at time $t$.
As it is shown in the proof of Theorem~4.2 in~\cite{Duquesne.LeGall-2005}, the measure $\ell^t(\dt\sigma)$ is the
image of the measure $\Lambda^t(\dt s)$ under certain mapping. In particular, one can immediately see from that definition, that the total masses of $\ell^t$ and $\Lambda^t$ are equal, that is,
\begin{align}
\label{eq:4_21_2}  \bk{\ell^t,\indi}=\Lambda^t_{\Li_H}.
\end{align}
In what follows we use the following notation for the total mass of the local time $\ell^t$:
$$ \bk{\ell^t}\eqdef\bk{\ell^t,\indi}.$$
\eqref{eq:4_21_2} and \eqref{eq:4_21_1} immediately
% \footnote{LM I removed reference to \citet[Th 1.4.1]{Duquesne LeGall-2002} }
 %\citet[Th 1.4.1]{Duquesne LeGall-2002}  have determined
give the law of the total mass of the local time $\bk{\ell^t}$. Namely,
\begin{align}
 \Nb\pthb{1 - \e^{-\lambda \bk{\ell^t}}} =  N\pthb{1 - \e^{-\lambda\Lambda^t_{\Li_H}}} =
 u_t(\lambda)= \pthb{(\gamma-1)t + \lambda^{1-\gamma} }^{-\tfrac{1}{\gamma-1}},\quad\lambda\ge 0.
\end{align}

$\Nb(\dt\Ti)$-a.e., the local time $t\mapsto\ell^t(\dt\sigma)$ is càdlàg for the weak topology and $\bk{\ell^t} > 0$ if and only if $h(\Ti)>t$, where $h(\Ti)$ denotes the total height of the tree: $$h(\Ti) = \sup\brc{d(\rho(\Ti),\sigma) : \sigma\in\Ti}.$$ The measure $N$ of the event
 $\{h(\Ti)>t\}$ is explicitly given by
\begin{align}
  \Nb\pth{ \bk{\ell^t} > 0 } = \pthb{(\gamma-1)t}^{-\tfrac{1}{\gamma-1}} \eqdef v(t),\quad\forall t\in\ivoo{0,\infty}.
\end{align}
We then denote by $\Nb_t(\dt\Ti)$ the conditional probability measure $\Nb\pthc{\dt\Ti}{\bk{\ell^t} > 0}$. In addition, we will also designate by $\sigma_\zeta$ the extinction node of $\Ti$: $d(\rho,\sigma_\zeta) = h(\Ti)$. We refer to \cite[Th. 4.4]{Duquesne.LeGall-2005} for the proof of its uniqueness.

For any $\sigma,\sigma'\in\Ti$, $\iivff{\sigma,\sigma'}$ stands for the unique geodesic between $\sigma$ and $\sigma'$. The subtree $\Ti_\sigma$ stemming from $\sigma\in\Ti$ is then defined as following:
\begin{align*}
  \forall \sigma\in\Ti;\quad \Ti_\sigma = \brcb{ \sigma'\in\Ti : \sigma\in\iivff{\rho(\Ti),\sigma'} }.
\end{align*}
For all $t,\delta\in\ivoo{0,\infty}$, we also introduce the subset $\Ti(t,\delta) = \brcb{\sigma\in\Ti(t) : h(\Ti_\sigma) > \delta}\subset\Ti(t)$. Since $\Ti$ is a compact space, $\Ti(t,\delta)$ is a finite subset of $\Ti(t)$. Moreover, we denote by $Z(t,\delta) \eqdef \# \Ti(t,\delta)$ its cardinal and  by $\Tbb(t,\delta)$ the collection of subtrees rooted at level $t$ and higher than $\delta$:
\begin{align*}
  \Tbb(t,\delta) = \brc{ \Ti_\sigma : \sigma \in \Ti(t,\delta)}\subset \Tbb
\end{align*}
where $\Tbb$ stands for the set of all equivalence classes of rooted compact
\R-trees (two rooted \R-trees are called equivalent if there is a root-preserving isometry mapping the two). We also designate by $\tr(t)$ the truncated tree above $t$: $\tr(t) = \brcb{\sigma\in\Ti : d(\rho(\Ti),\sigma) \leq t}$. Note that the limiting case of $\delta=0$ is simply defined by: $\Ti(t,0)=\cup_{\delta>0} \Ti(t,\delta)$ and $\Tbb(t,0)=\cup_{\delta>0} \Tbb(t,\delta)$.\vsp

\noindent\emph{Branching property.} One important feature of Lévy trees is the branching property presented by \citet{Duquesne.LeGall-2005}. For any $t\in\ivoo{0,\infty}$,
let $\Gi_t$ be the $\sigma$-field generated by $\tr(t)$ and $\Ni_t$ be the following point measure
\begin{align}
  \Ni_t(\dt\sigma'\dt\Ti') = \sum_{\sigma\in\Ti(t,0)} \delta_{(\sigma,\Ti_\sigma)}.
\end{align}
The branching property then states that under $\Nb_t\pthc{\dt\Ti}{\Gi_t}$, $\Ni_t$ is a Poisson point process on $\Ti(t)\times\Tbb$ with intensity $\ell^t(\dt\sigma')\Nb(\dt\Ti')$. Note that \citet{Weill-2007} has conversely proved that the branching properly entirely characterised the law of Lévy trees.\vsp

\noindent\textbf{Brownian snake.} Given a compact rooted $\R$-tree $\Ti$ representing the genealogy of individuals, we may now introduce the motion of the \emph{Brownian snake}. For that purpose, for any $x\in\R^d$, we define an $\R^d$-valued Gaussian process $\Wi\eqdef(W_\sigma, \sigma\in\Ti)$ indexed by the continuous tree $\Ti$ and whose distribution is characterised by
\begin{align}  \label{eq:cov_W}
  \esp{W_\sigma} = x\quad\text{and}\quad \cov\pthb{W_\sigma,W_{\sigma'}} = d(\rho(\Ti),\sigma \land \sigma') \Id_d
\end{align}
where $\Id_d$ designates the $d$-dimensional identity matrix. We denote by $\Tbb_{\text{sp}}$ the space of spatial trees corresponding to couples of the form $(\Ti,\Wi)$. Given $\Ti$, we denote by $Q_\Ti^x(\dt\Wi)$ the law of the tree-indexed process $(W_\sigma)_{\sigma\in\Ti}$ starting from $x$ and introduce the measure $\N_x(\dt\Ti\dt\Wi)$ on the set of spatial trees as following
\begin{align}
  \N_x(\dt\Ti\dt\Wi) \eqdef \Nb(\dt\Ti)\,Q_\Ti^x(\dt\Wi).
\end{align}
Note that the previous definition is licit since the map $\Ti\mapsto Q_\Ti^x$ is measurable. As presented in \cite{Duquesne.LeGall-2005}, if $\Ti$ follows the law of a stable tree, the process $\Wi$ has $\N_x$-a.e. a continuous modification which is Hölder continuous with exponent $\tfrac{1}{2}-\eps$ for any $\eps>0$. In what follows if a function is H\"older continuous with exponent $\eta$, we for simplicity call it $\eta$ continuous function. \vsp

Nowthe super-Brownian motion with stable branching can
%then
 be constructed from the above spatial tree. First
%Starting with the excursion measures, we define,
under $\N_x$ and for every $t>0$, we define the measure $\Xii_t=\Xii_t(\Ti,W)$ on $\R^d$ as follows: for
any bounded measurable non-negative function $\varphi$ on $\R^d$,
\begin{align}  \label{eq:sbm_excursion_measures}
  \bk{\Xii_t,\varphi} = \int_{\Ti(t)} \varphi(W_\sigma) \, \ell^t(\dt\sigma).
\end{align}
%where $\varphi$ is a positive measurable function on $\R^d$.
Note that the above equality
%directly
immediately implies that $\supp \Xii_t \subset W(\Ti(t))$.
 \citet{Duquesne.LeGall-2002} have
%then
proved that certain Poisson sum of such
% \footnote{LM I removed "excursion" all the way down ---are you sure it is called this way?\\
% PB: My bad, they are not. I checked the rest of the article too. LM2 good}  %excursion
measures gives a (stable) super-Brownian motion:
\begin{proposition}  \label{prop:sbm_construction}
  Suppose $\mu\in M_f(\R^d)$ and
  \begin{align*}
    \sum_{i\in\Ii} \delta_{\pth{\Ti^i,\Wi^i}}
  \end{align*}
  is a Poisson measure with intensity $\int_{\R^d}\mu(\dt x)\N_x(\dt\Ti\dt\Wi)$. Then, the measure-valued process $(X_t)_{t\geq 0}$ defined by
  \begin{align}  \label{eq:sbm_poisson_construction}
    \forall t>0;\quad X_t(\dt x) = \sum_{i\in\Ii} \Xii_t(\Ti^i,\Wi^i) \quad\text{and}\quad X_0=\mu
  \end{align}
  is a super-Brownian motion starting from $\mu$ with stable branching mechanism.
\end{proposition}
As pointed out by \citet{Duquesne.LeGall-2005}, the interesting aspect of the previous construction is to directly obtain a proper version of a superprocess, i.e. which is càdlàg with respect to the weak topology on $\ivfo{0,\infty}$. In addition, for every $t>0$, the sum \eqref{eq:sbm_poisson_construction} only presents a finite number of terms, as only finitely many trees satisfy $h(\Ti^i)>t$.

Consequently, in the rest of the article, we will mainly focus on the study of the fine fractal geometry of the
%excursion
measure  $\Xii_t$ under $\N_x$, and then deduce the multifractal structure of stable super-Brownian motion itself. More specifically, we will prove the following main proposition about properties of
% excursion
measure $\Xii_t$ under $\N_x$.
\begin{proposition}  \label{prop:spectrum_excursions}
  Suppose $x\in\R^d$. Then, the following statements  hold $\N_x$-a.e.
  \begin{enumerate}[(a)]
    \item Assuming $d\geq 2$, the spectrum of singularities of the
%excursion
measure $\Xii_t(\dt x)$ is equal
% \footnote{LM I suggest to change from $h$ to $h$ to be consistent with
% the previous theorems\\ PB: Which modification?  \\LM2 I made a typo: what I meant that in the previous version we had $h\in (1/\gamma, 1/(\gamma-1))$ and I changed it to $h\in (2/\gamma, 2/(\gamma-1))$ and dimensions respectively\\
% PB2: I guess it's better and more consistent to use this convention in the main results. I kept $h\in (1/\gamma, 1/(\gamma-1))$ in the technical lemmas and their proofs as it helps for the readability of the calculus and is more consistent with the tree notations. Is it fine for you?}
to:
    \begin{align*}
      \forall h\in\ivffb{\tfrac{2}{\gamma},\tfrac{2}{\gamma-1}}\cap\ivfob{0,d};\quad \dimH \,E(h,\Xii_t)\cap V = \gamma h - 2,
    \end{align*}
    for any $t>0$ and open set $V\subset\R^d$ such that $\Xii_t(V)>0$.
    \item Supposing $d\geq \tfrac{2\gamma}{\gamma-1}$, for any Borel set $F\subset\ivoo{0,h(\Ti)}$ satisfying the \emph{strong Frostman's lemma} \eqref{eq:strong_frostman} and such that $\dimH F=\dimP F$, we have:
    \begin{align*}
      \forall h\in\ivffb{\tfrac{2}{\gamma}, \tfrac{2}{\gamma-1}};\quad \dimH \bigcup_{s\in F} E(h,\Xii_s) = \gamma h - 2 + 2\dimH F.
    \end{align*}
     \item For any dimension $d\geq 1$,
    \begin{align*}
      \forall h\in\ivfob{0,\tfrac{2}{\gamma}}\cap\ivfob{0,d};\quad \dimH \brcb{t>0 : E(h,\Xii_t) \neq \vset } = \gamma h,
    \end{align*}
    Moreover, for every $t>0$ and any $h\in\ivfob{0,\tfrac{2}{\gamma}}\cap\ivfob{0,d}$, $E(h,\Xii_t)$ is either empty or has zero Hausdorff dimension.
  \end{enumerate}
\end{proposition}
One key tool to study the multifractal structure of
%excursion measures
the  measure $\Xii_t$ under $\N_x$ is to introduce the notion of \emph{local nondeterminism} for the Brownian snake.

\subsection{Local nondeterminism on Brownian motion indexed by trees}

\citet{Balanca-2015b} provides a fine description of the multifractal structure of the indexing stable trees. Consequently, to obtain a characterisation of the singularities of stable super-Brownian motion, one has to precisely analyse the behaviour of the tree-indexed Gaussian process $\Wi$. The study of Gaussian processes and, more generally, Gaussian random fields, has been a long existing field of research in probability theory. In particular, much work has been done in the last 20 years to understand the fine geometry of multiparameter multidimensional Gaussian processes such as \emph{fractional Lévy fields} , \emph{fractional Brownian sheets} or the solutions of SPDEs with Gaussian noise \cite{Xiao.Zhang-2002,Ayache.Xiao-2005,Xiao-2006,Xiao-2009a}. The study of these Gaussian fields share some common ground, and more specifically, it has appeared that one key element in the characterisation of the fractal geometry of Gaussian processes lays in the property commonly called \emph{local nondeterminism}. Initially introduced by \citet{Berman-1970} to investigate the local time of Gaussian processes, it has since been successfully used to understand multiple geometric aspects of a large class of Gaussian fields (we refer particularly to the surveys of \citet{Xiao-2006,Xiao-2013} for a deeper review on the subject). As a consequence, it seems quite reasonable in our context to present an analogue of the local non-determinism property on the tree-indexed process $\Wi$.
\begin{lemma}  \label{lemma:sbm_lnd}
  Suppose $\Ti$ is a compact $\R$-tree, $d=1$, $x=0$, $N\geq 1$
  % \footnote{I do not like notation $N$ since it is used to denote excursion measure of $H$ but I still did not change it \\PB: Prefer $M$ or $K$?\\ LM2 Maybe it is just easier to change notation for the excursion of $H$ from $N$ to $\mathbf{N}^H$ or similar. We almost do not use it anyway.\\ PB2: Sounds good to me, I changed to $\mathbf{N}^H$ in the rest.}
  and $\Wi=(W_\sigma)_{\sigma\in\Ti}$ is the Gaussian process defined by \eqref{eq:cov_W}. Then, there exists a constant $c_{1,N}>0$ such that for every $\sigma\in\Ti$ and all $\sigma_1,\dotsc,\sigma_N\in\Ti$,
  \begin{align*}
    \varcb{W_\sigma}{W_{\sigma_0},W_{\sigma_1},\cdots,W_{\sigma_N}} \geq c_{1,N} \min_{0\leq i\leq N} d(\sigma,\sigma_i),
  \end{align*}
  where by convention $\sigma_0$ denotes the root $\rho$ of $\Ti$. In addition, $c_{1,N} \geq (2N!)^{-2}$.
\end{lemma}
\begin{proof}
  Recall that by the definition of the conditional
%expectancy
expectation  in $L^2$,
  \begin{align*}
     \varcb{W_\sigma}{W_{\sigma_0},W_{\sigma_1},\cdots,W_{\sigma_N}} = \inf_{a\in\R^{N+1}} \varbb{{W_\sigma - \sum_{i=0}^N a_i W_{\sigma_i} }}.
  \end{align*}
  We denote by $\Ti_\sigma$ the subtree rooted in $\sigma$, and define the following two subsets of $\brc{\sigma_0,\sigma_1,\cdots,\sigma_n}$: $\Ei_+ = \brc{\sigma_i : \sigma_i\in\Ti_\sigma}$ and $\Ei_- = \brc{\sigma_i : \sigma_i\notin\Ti_\sigma}$. In addition, for any $\sigma'\preceq\sigma''\in\Ti$, we denote by $W_{\iivff{\sigma',\sigma''}}$ the increment $W_{\sigma''}-W_{\sigma'}$. Let us set $a\in\R^{N+1}$ and observe that
  \begin{align*}
    \varbb{{W_\sigma - \sum_{i=0}^N a_i W_{\sigma_i} }}
    &=\varbb{{W_\sigma - \sum_{\sigma_i\in\Ei_+} a_{i} W_{\sigma_{i}} - \sum_{\sigma_j\in\Ei_-} a_j W_{\sigma_{j}} }} \\
    &= \varbb{{W_\sigma \pthbb{1-\sum_{\sigma_i\in\Ei_+} a_i } -\sum_{\sigma_i\in\Ei_+} a_i \pthb{W_{\sigma_{i}}-W_{\sigma}} - \sum_{\sigma_j\in\Ei_-} a_j W_{\sigma_{j}} }}.
  \end{align*}
  For any $\sigma_i\in\Ei_+$ and $\sigma_j\in\Ei_-$, $\cov\pth{W_{\sigma_{j}},W_{\sigma_{i}}-W_{\sigma}} = 0$ and $\cov\pth{W_{\sigma},W_{\sigma_{i}}-W_{\sigma}} = 0$. Hence, the independence of the Gaussian vectors entails
  \begin{align*}
    \varbb{W_\sigma - \sum_{i=0}^N a_i W_{\sigma_i} }
    = \varbb{W_\sigma \pthbb{1-\sum_{\sigma_i\in\Ei_+} a_i} - \sum_{\sigma_j\in\Ei_-} a_j W_{\sigma_{j}} } + \varbb{ \sum_{\sigma_i\in\Ei_+} a_i \pthb{W_{\sigma_{i}}-W_{\sigma}} }.
  \end{align*}
  Let us first suppose that $\sum_{\sigma_i\in\Ei_+} a_i\geq \tfrac{1}{2}$ and investigate the second term. We denote by $\sigma_1'$ the common ancestor of the nodes in $\Ei_+$. We also introduce a partition $\Ei_1\cup\cdots\cup\Ei_k$ of $\Ei_+$ such that $(\Ei_j)_{1\leq j\leq k}$ are subsets of separate sub-tree rooted at $\sigma_1'$. Note that by definition of $\Ei_+$, $\Ei_j\subset\Ti_{\sigma}$ for every $j\in\brc{1,\dotsc,k}$. Then,
  \begin{align*}
    \varbb{\sum_{\sigma_i\in\Ei_+} a_i \pthb{W_{\sigma_{i}}-W_{\sigma}} }
    &= \varbb{ W_{\iivff{\sigma,\sigma_1'}} \sum_{\sigma_i\in\Ei_+} a_i } + \varbb{\sum_{\sigma_i\in\Ei_+} a_i \pthb{W_{\sigma_{i}}-W_{\sigma_1'}} } \\
    &\geq \frac{1}{4} \,d(\sigma_1',\sigma_0') + \varbb{ \sum_{\sigma_i\in\Ei_+} a_i  \pthb{W_{\sigma_{i}}-W_{\sigma_1'}} },
  \end{align*}
  where by convention $\sigma'_0=\sigma$.
  Due to the independence of the components of $\Wi$ indexed by the distinguished subtrees rooted at $\sigma_1'$, the right hand term satisfies
  \begin{align*}
    \varbb{ \sum_{\sigma_i\in\Ei_+} a_i \pthb{W_{\sigma_{i}}-W_{\sigma_1'}} } = \sum_{j=1}^k \varbb{ \sum_{\sigma_i\in\Ei_j} a_i \pthb{W_{\sigma_{i}}-W_{\sigma_1'}} }.
  \end{align*}
  Since $\sum_{\sigma_i\in\Ei_+} a_i \geq \tfrac{1}{2}$ and $k\leq N$, there exists $j\in\brc{1,\dotsc,k}$ such that $\sum_{\sigma_i\in\Ei_j} a_i \geq \tfrac{1}{2N}$. Then, we may iterate the previous procedure on the collection $\Ei_j$: let us define $\sigma_2'$ as the common ancestor of the nodes $\sigma_i\in\Ei_j$ and observe that
  \begin{align*}
    \varbb{ \sum_{\sigma_i\in\Ei_j} a_{i} \pthb{W_{\sigma_{i}}-W_{\sigma_1'}} }
    &= \varbb{ W_{\iivff{\sigma_1',\sigma_2'}} \sum_{\sigma_i\in\Ei_j} a_{i} }  + \varbb{ \sum_{\sigma_i\in\Ei_j} a_{i} \pthb{W_{\sigma_{i}}-W_{\sigma_2'}} } \\
    &\geq \frac{1}{4N^2} d(\sigma'_1,\sigma'_2) + \varbb{ \sum_{\sigma_i\in\Ei_j} a_{i} \pthb{W_{\sigma_{i}}-W_{\sigma_2'}} }.
  \end{align*}
  Iterating the previous construction, we observe that the latter stop after at most $N$ steps when a node $\sigma_i\in\Ei_+$ is reached. As a consequence, we construct by induction a family of nodes $\sigma_0',\dots,\sigma'_p$ such that $\sigma'_0=\sigma$, $\sigma'_p=\sigma_i$ for some $i$ and
  \begin{align*}
    \varbb{W_\sigma - \sum_{i=0}^N a_i W_{\sigma_i} }
    \geq c_N \sum_{j=0}^{p-1} d(\sigma'_j,\sigma'_{j+1}) = c_N\, d(\sigma,\sigma_i) \geq c_N \min_{0\leq i\leq N} d(\sigma,\sigma_i),
  \end{align*}
  for a constant $c_N>0$. In addition, according to the induction procedure, we get $c_N\geq (2N!)^{-2}$.

  Let us now suppose that $\sum_{\sigma_i\in\Ei_+} a_i < \tfrac{1}{2}$ and study the term $\varb{W_\sigma \pthb{1-\sum_{\sigma_i\in\Ei_+} a_{i}} - \sum_{\sigma_i\in\Ei_-} a_i W_{\sigma_{i}} }$. We denote by $b$ the constant $b\eqdef 1-\sum_{\sigma_i\in\Ei_+} a_i\geq\tfrac{1}{2}$ and proceed in a similar fashion: let $\sigma'_1$ be the highest ancestor of the type $\sigma\land\sigma_i$, where $\sigma_i\in\Ei_-$. The previous definition is licit as $\sigma\land\sigma_i\in\iivff{\rho,\sigma}$ for any $i$. We still denote by $\Ei_1,\dotsc,\Ei_k$ the partition of $\Ei_-$ corresponding to separated subtrees stemming from $\sigma_1'$, and define $\Ei_0 = \Ei_-\setminus\cup_{j=1}^k\Ei_j$ Then,
  \begin{align*}
    &\varbb{bW_\sigma - \sum_{\sigma_i\in\Ei_-} a_i W_{\sigma_{i}} }
    = \varbb{b W_\sigma  - \sum_{j=0}^k \sum_{\sigma_i\in\Ei_j} a_i W_{\sigma_{i}} }.
  \end{align*}
  Using the independence of the different components of $W$, the right hand component is equal to
  \begin{align*}
    \varb{b (W_\sigma-W_{\sigma'_1})}
    &+ \varbb{ W_{\sigma'_1}\pthbb{ b-\sum_{j=1}^k \sum_{\sigma_i\in\Ei_j} a_i }  - \sum_{\sigma_i\in\Ei_0} a_{i} W_{\sigma_{i}} } \\
    &+ \sum_{j=1}^k \varbb{ \sum_{\sigma_i\in\Ei_j} a_{i} (W_{\sigma_{i}}-W_{\sigma'_1}) },
  \end{align*}
  where the first term $\varb{b (W_\sigma-W_{\sigma'_1})}$ is lower bounded by $\tfrac{1}{4}d(\sigma'_0,\sigma'_1)$. Let us now distinguish two different cases. Suppose first that $b_1 \eqdef b-\sum_{j=1}^k\sum_{\sigma_i\in\Ei_j} a_i \geq \tfrac{1}{4}$. Observing that we then obtain the exact same configuration, we simply iterate the previous procedure on the component
  \begin{align*}
    \varbb{ b_1 W_{\sigma'_1}  - \sum_{\sigma_i\in\Ei_0} a_{i} W_{\sigma_{i}} }.
  \end{align*}
  Otherwise, as $b\geq \tfrac{1}{2}$, there exists $j\in\brc{1,\dots,k}$ such that $\sum_{\sigma_i\in\Ei_j} a_{i} \geq \tfrac{1}{4N}$, and we therefore need to lower bound the following term
  \begin{align*}
    \varbb{ W_{\sigma'_1}\sum_{\sigma_i\in\Ei_j} a_{i} - \sum_{\sigma_i\in\Ei_j} a_{i} W_{\sigma_{i}} }.
  \end{align*}
  We then observe that this question is strictly equivalent to the first case studied in this proof, since for any $\Ei_j\subset\Ti_{\sigma'_{1}}$. As a consequence, we can iterate the procedure and obtain as well in both situations a collection of nodes $\sigma_0',\dots,\sigma'_p$ such that $\sigma'_0=\sigma$, $\sigma'_p=\sigma_i$ for some $i$ and
  \begin{align*}
    \varbb{W_\sigma - \sum_{i=0}^N a_i W_{\sigma_i} }
    \geq c_N \sum_{j=0}^{p-1} d(\sigma'_j,\sigma'_{j+1}) = c_N\, d(\sigma,\sigma_i) \geq c_N \min_{0\leq i\leq N} d(\sigma,\sigma_i),
  \end{align*}
  where the constant still satisfies $c_N\geq (2N!)^{-2}$. The latter inequality hence concludes the proof of the lemma.
\end{proof}
\begin{remark}
  We note that in Lemma~\ref{lemma:sbm_lnd}, the constant $c_{1,N}$ appearing in the lower bound may depend on the parameter $N$. In the literature, a Gaussian process is usually said to satisfy a \emph{strong local nondeterminism} property (see \cite{Xiao-2013} on this topic) if the former constant is independent of $N$. Such a property then allows to have stronger estimates on the law of the Gaussian process (small balls, local time, \dots).

  In the setting of this work, one can observe that if $\Ti$ is a continuous stable tree, then $\Wi$ is not strongly locally nondeterministic. Indeed, let $\sigma_0\in\Ti$ be a vertex with infinite multiplicity (see \cite{Duquesne.LeGall-2005} for their existence whenever $\gamma\in\ivoo{1,2}$) and for every $r>0$, $\Ei(r)\subset\Ti$ be a collection of nodes such that for every $\sigma\neq\sigma'\in\Ei(r)$, $d(\sigma_0,\sigma)=d(\sigma_0,\sigma')=r$ and the nodes $\sigma$ and $\sigma'$ belong to separate connected components in $\Ti\setminus\brc{\sigma_0}$. It then follows from the independence properties of $\Wi$ that
  \begin{align*}
    \forall r>0;\quad \varcb{W_{\sigma_0}}{W_{\sigma}:\sigma\in\Ei(r)} \leq  r\cdot\bktb{\#\Ei(r)}^{-1}
  \end{align*}
  Since $\sigma_0$ has infinite multiplicity, $\Ei(r)\rightarrow\infty$ and as a consequence, $\Wi$ can not be strongly locally nondeterministic.

  It remains nevertheless an open question on which class of compact $\R$-trees the constant $c_{1,N}$ is independent of $N$, and in particular if $\Wi$ is strongly locally nondeterministic when $\Ti$ is the \emph{Continuous Random Tree} ($\gamma=2$).
\end{remark}

\subsection{Some tree-collections and a few technical results}

As we aim to adapt Gaussian techniques to our continuous tree formalism, we may observe that the main difference with the classic setting lays in the heterogeneity of the indexing space. Indeed, the Gaussian literature mainly deals with random fields indexed by $\R^N$, or more generally manifolds, which have an homogeneous structure. On the contrary, the multifractal geometry of  stable trees described in \cite{Balanca-2015b} shows that the former are particularly non-homogeneous indexing spaces, with a local dimension varying largely from
one vertex to another. A consequence of this feature is that if one tries to apply directly Gaussian techniques to the process $\Wi$, it will only provide the worse case scenario in terms of regularity. If that strategy is good enough to retrieve some results such that the optimal Hölder regularity of the density presented in \cite{Fleischmann.Mytnik.ea-2010}, it is not sufficient to obtain a full picture of the geometry of the stable super-Brownian motion, and in particular, characterise its multifractal structure. Consequently, in order to circumvent this issue, we will investigate the properties of the process $(W_\sigma)_{\sigma\in\Ti}$ when indexed by subsets of the continuous tree $\Ti$, focusing in particular on those sufficiently homogeneous. More formally, we introduce %in the following definitions
now  several  important classes of such subsets.

\begin{definition}  \label{def:collections_Th1}
  For any $\delta>0$, interval $H\eqdef\ivff{h_0,h_1}\in\R$ and $\kappa\geq 1$, we denote by $\Ti\brc{\delta,H,\kappa}$ the collection:
  \begin{align*}
    \Ti\brc{\delta,H,\kappa} \eqdef \brcb{\sigma\in\Ti : \ell^{a(\sigma)}\pthb{\Bi(\sigma,2\delta)} \in\ivofb{\kappa\delta^{h_1},\kappa\delta^{h_0}} },
  \end{align*}
  where
  % \footnote{LM I do not like $h(\sigma)$ since we use $h$ all the way for H\"older index\\PB: Could use $a(\sigma)$ for the level of a vertex? Or $t(\sigma)$? \\ LM2 $a(\sigma)$ sounds fine}
  $a(\sigma)\eqdef d(\rho,\sigma)$.
  We also define $\Ti\brc{t,\delta,H,\kappa}\eqdef \Ti\brc{\delta,H,\kappa}\cap\Ti(t)$ for every $t>0$. In the cases $H=\ivfo{h,\infty}$ or $H=\ivof{-\infty,h}$, we respectively use the notations:
  \begin{align*}
    \Ti\brc{\delta,h_\geq,\kappa} = \Ti\brcb{\delta,\ivfo{h,\infty},\kappa} \quad\text{and}\quad \Ti\brcb{\delta,h_<,\kappa} = \Ti\brc{\delta,\ivof{-\infty,h},\kappa},
  \end{align*}
  setting as well $\Ti\brc{t,\delta,h_\geq,\kappa}=\Ti\brc{\delta,h_\geq,\kappa}\cap\Ti(t)$ and $\Ti\brc{t,\delta,h_<,\kappa}=\Ti\brc{\delta,h_<,\kappa}\cap\Ti(t)$.

  Finally, we will also make use of the following classes:
  \begin{align*}
    \Ti\brc{\delta_\geq,H,\kappa} \eqdef \bigcap_{r\geq\delta} \Ti\brcb{r,H,\kappa} = \brcb{\sigma\in\Ti : \forall r\geq\delta,\ \ell^{a(\sigma)}\pthb{\Bi(\sigma,2r)} \leq \ivofb{\kappa r^{h_1},\kappa r^{h_0}} },
  \end{align*}
  using as well the analogue notations $\Ti\brc{\delta_\geq,h_\geq,\kappa}$ and $\Ti\brc{\delta_\geq,h_<,\kappa}$ to refer to the cases $H=\ivfo{h,\infty}$ or $H=\ivof{-\infty,h}$.
\end{definition}

In addition to the introduction of the previous collections and subsets of $\Ti$ and $\Ti(t)$, we also require to define some analogue classes of on families of subtrees.
\begin{definition}  \label{def:collections_Th2}
  For every $t>0$, any $\delta>0$, any interval $H\eqdef\ivff{h_0,h_1}\in\R$ and any $\kappa\geq 1$, we define:
  \begin{align*}
    \Tbb\pth{t,\delta,H,\kappa} \eqdef \brcbb{ \Ti_\sigma\in\Tbb(t,\delta) : \sup_{u\in\ivff{\delta,2\delta}} \bk{\ell^u}(\Ti_{\sigma}) \in\ivofb{\kappa\delta^{h_1},\kappa\delta^{h_0}} }.
  \end{align*}
  We will also make use of the following notations: $\Tbb\pth{t,\delta,h_\geq,\kappa} \eqdef \Tbb\pth{t,\delta,\ivfo{h,\infty},\kappa}$ and $\Tbb\pth{t,\delta,h_<,\kappa} \eqdef \Tbb\pth{t,\delta,\ivof{-\infty,h},\kappa}$
\end{definition}
Note that in the previous definition, $\bk{\ell^u}(\Ti_{\sigma})$ refers to the local time at level $u$ in the sub-tree $\Ti_\sigma$.
As pointed out above, Definitions~\ref{def:collections_Th1} and \ref{def:collections_Th2} introduce sub-classes sufficiently homogeneous to study the restriction of process $(W_\sigma)_{\sigma\in\Ti}$ using known Gaussian techniques. Moreover, if one defines the following collections of intervals $\Hbb_n\eqdef\pth{H_p}_{1\leq p\leq n}$:
\begin{align}  \label{eq:intervals_Hp}
  \forall p\in\brc{1,\dotsc,n-1};\quad H_p \eqdef \ivfob{h_{p-1},h_{p}} \quad\text{and}\quad  H_{n} \eqdef \ivfob{h_{n-1},\infty}\text{ where } h_p\eqdef\tfrac{p}{n(\gamma-1)},
\end{align}
we remark that the former sub-classes form a partition of the tree:
\begin{align}  \label{eq:tree_partition}
  \Ti = \bigcup_{p=1}^ n\Ti\brc{\delta,H_p,\kappa}\quad\text{and}\quad \Tbb(t,\delta) = \bigcup_{p=1}^ n\Tbb\pth{t,\delta,H_p,\kappa} ,\quad\text{whenever } \kappa\geq \sup_{u\geq 0}\bk{\ell^u}.
\end{align}

% Note that in the previous definition, since for any $\Ti_\sigma\in\Tbb(t-\delta,\delta)$, $\bk{\ell^\delta}(\Ti_{\sigma}) = \ell^t(\Ti_\sigma)$, we get $\Ti\brc{t,\delta,H,\kappa} = \cup_{\Ti_\sigma\in\Tbb\pth{t,\delta,H,\kappa}} \Ti_\sigma(\delta)$.

As we will be investigating in the rest of the article the behaviour of the process $\Wi$ restricted to the subsets $\Ti\brc{\delta_\geq,H,\kappa}$ and $\Ti\brc{\delta,H,\kappa}$, we may first present a few properties of the latter. Note that the next two lemmas extend features of stable trees described in \cite[Lemma 4.4]{Balanca-2015b}.

Before stating these results, let us define a few additional notations. For any $x\in\ivoo{0,1}$, we set the function: $g(x) \eqdef \pthb{\log x^{-1} }^{-1}$. Moreover, for any $n\in\N$, $\delta_n\eqdef 2^{-n}$ and $\Di_n$ denotes the collection of standard closed dyadic intervals: $\Di_n \eqdef \brcb{ \ivff{k\delta_n,(k+1)\delta_n} : k\in\Z}$. Finally, in the rest of the article, $b>0$ will denote a deterministic fixed level.

\begin{lemma}  \label{lemma:unif_bound_Th}
  Let us use the notations \eqref{eq:intervals_Hp} introduced above. $\Nb$-a.e. there exists $N\geq 1$ such that for every $n\geq N$, all $t\in\ivoo{0,b}$, every $H_p\in\Hbb_n$, any $\sigma\in\Ti(t)$, $r\geq2\delta_n$ and $\kappa\geq 1$:
  \begin{align*}
    \#\pthb{ \Bi(\sigma,2r) \cap \Tbb\brc{t-\delta_n,\delta_n,H_p,\kappa} } \leq c_0\, \delta_n^{1-\gamma h_p} \pthB{\ell^t\pthb{ \Bi(\sigma,8r)} + g(r)^{-2}r^{\tfrac{1}{\gamma-1}}} + c_0\,\kappa n^2.
  \end{align*}
  where the positive constant $c_0$ only depends on $\gamma$.

  The same result holds on the collection of index intervals $\widehat H_p \eqdef \ivof{-\infty,h_{p}}$, $p\in\brc{1,\dotsc,n-1}$, and $\widehat H_n=\R$.
\end{lemma}
Before proving Lemma~\ref{lemma:unif_bound_Th}, let us note that the informal notation $\Bi(\sigma,2r) \cap \Tbb\brc{t-\delta_n,\delta_n,H_p,\kappa}$ is licit since for any $\Ti_\sigma\in \Tbb\brc{t-\delta_n,\delta_n,H_p,\kappa}$, either $\Ti_\sigma(\delta)\subset \Bi(\sigma,2r)$ or $\Ti_\sigma(\delta)\cap \Bi(\sigma,2r)=\vset$.

Let us also describe in words the meaning of the above lemma. Roughly speaking, for any $n$ suffuciently large the lemma gives a uniform bound on a number of subtrees (in a neighborhood of any node $\sigma$)  "born" at level $t-\delta_n$ that
possess a local time of order $\delta_n^h$. If $h_p<1/\gamma$ then the numder of such trees is finite!
% \footnote{LM
% 1.  it is not in the statement yet, but shoud be there? If not then delete the last statement.\\
% 2. I am also wondering now whether the number of such trees should be finite also in some additional cases when
%  $h_p>1/\gamma$. E.g. if we bound $r\in (2\delta_n,\epsilon_n)$ for some $\epsilon_n$ "small", then whenever
% $\ell^t(\Bi(\sigma,8\epsilon_n))$ is small (like in the case we use it in the proof of Lemma~\ref{lemma:sbm_clusters_ub}  where it is of order $\delta_k^h$) then as long as
% $\delta_n^{1-\gamma h_p} \ell^t\pthb{ \Bi(\sigma,8\epsilon_n)} $  goes to zero sufficiently fast one can get the
% "finite" number of trees in certain neighborhood of $\sigma$. Does it make sense?

% We do not need to put it in the paper, but just keep in mind.\\
% PB: Sure, it could be refine and I guess the latter should be true in most of cases.
% \\
% LM2 OK}
\begin{proof}
  Let us set $n\in\N$, $\kappa\geq 1$, $H_p\in\Hbb_n$, $j\delta_n>0$. We define a slight modification of collections $\Tbb\brc{j\delta_n,\delta_n,H_p,\kappa}$:
  \begin{align*}
    \widehat\Tbb\pth{j\delta_n,\delta_n,h_p,\kappa,\Ti} \eqdef \brcbb{ \Ti_\sigma\in\Tbb(j\delta_n,\delta_n,\Ti) : \sup_{u\in\ivof{\delta_n,4\delta_n}}\bk{\ell^u}(\Ti_{\sigma}) \geq \kappa\delta_n^{h_p} }.
  \end{align*}
  Observe that for any $t\in\ivfo{(j+1)\delta_n, (j+2)\delta_n}$, $\Tbb\brc{t-\delta_n,\delta_n,H_p,\kappa} \subset \widehat\Tbb\pth{j\delta_n,\delta_n,h_p,\kappa}$, and as a consequence, it is sufficient to control uniformly the size of the former collection.
  For that purpose, let us now set $\delta_k\geq\delta_n$ and introduce two families of events:
  \begin{align*}
    A_1(\delta_k, \delta_n) \eqdef \brcB{ \Ti: \inf_{u\in\ivfo{\delta_k,\delta_k+2\delta_n}} 2\bk{\ell^u} \leq \bk{\ell^{\delta_k}} \text{ and } \bk{\ell^{\delta_k}} \geq k^2\delta_k^{\tfrac{1}{\gamma-1}} }
  \end{align*}
  and $A_2(\delta_k,\delta_n,h_p) \eqdef \cup_{\kappa=1}^\infty A_2(\delta_k,\delta_n,h_p,\kappa)$ where for any $\kappa\geq 1$:
  \begin{align*}
    A_2(\delta_k,\delta_n,h_p,\kappa) \eqdef \brcB{ \Ti: \#\widehat\Tbb\pth{\delta_k,\delta_n,h_p,\kappa} \geq \delta_n^{1-\gamma h_p}\bk{\ell^{\delta_k}} + \kappa n^2}.
  \end{align*}
  We may then define a second collection of subtrees for any $a>0$:
  \begin{align*}
    \Tbb_A\pth{a,\delta_k,\delta_n,h_p} \eqdef \brcB{ \Ti_\sigma\in\Tbb(a,\delta_k) : \Ti_\sigma\in A_1(\delta_k, \delta_n)\text{ or } \Ti_\sigma\in A_2(\delta_k,\delta_n,h_p) }.
  \end{align*}
  Note that according the previous definitions, if $\Tbb_A\pth{a,\delta_k,\delta_n,h_p}=\vset$, then for any $t\in\ivfo{a+\delta_k+\delta_n,a+\delta_k+2\delta_n}$ and any $\sigma\in\Ti(t)$:
  \begin{align*}
    \#\pthb{ \Bi(\sigma,2\delta_k) \cap \Tbb\brc{t-\delta_n,\delta_n,H_p,\kappa} } \leq  c\,\delta_n^{1-\gamma h_p} \pthB{\ell^t\pthb{ \Bi(\sigma,4\delta_k)} + g(\delta_k)^{-2}\delta_k^{\tfrac{1}{\gamma-1}}} + c\,\kappa n^2.
  \end{align*}
  As a consequence, we will obtain the desired property if we prove that the collection $\Tbb_A\pth{a,\delta_k,\delta_n,h_p}$ is uniformly empty for any $\delta_n$ large enough.

  For that purpose, recall that given $\Gi_{a}$, the branching property of Lévy trees endows that $\#\Tbb_A\pth{a,\delta_k,\delta_n,h_p}$ is a Poisson random variable parametrised by $\bk{\ell^{a}}\Nb\pthb{A_1(\delta_k,\delta_n)\cup A_2(\delta_k,\delta_n,h_p)}$.
  We may begin by estimating the measure of the first event of $A_1(\delta_k, \delta_n)$. The Ray--Knight theorem proved by \citet[Th 1.4.1]{Duquesne.LeGall-2002} states that given $\Gi_{\delta_k}$, under $\Nb_{\delta_k}$, the process $X_u\eqdef\bk{\ell^{u-\delta_k}}$ is a stable CSBP starting from $\bk{\ell^{\delta_k}}$. Then, Lemma~3.4 presented in \citet{Balanca-2015b} entails that if $X$ is a stable CSBP, for any $x\geq k^2\delta_k^{1/(\gamma-1)}$
  \begin{align*}
    \prB[_x]{ \inf_{u\leq2\delta_n} X_u \leq x/2  } \leq \exp\pthb{-c_0\,x\delta_n^{-1/(\gamma-1)}} \leq \exp\pthb{-c_0\,k^2(\delta_k/\delta_n)^{1/(\gamma-1)}} \leq \exp\pth{-c_1\,n^{3/2}},
  \end{align*}
  where the constants $c_0,c_1>0$ only depends on $\gamma$. Hence, $\Nb\pthb{ A_1(\delta_k, \delta_n) } \leq v(\delta_k)\exp\pth{-c_1\,n^{3/2} }$.

  Let us now investigate the second event $A_2(\delta_k,\delta_n,h_p)$. Still using the branching property of stable trees, we know that given $\Gi_{\delta_k}$, $\#\widehat\Tbb\pth{\delta_k,\delta_n,h_p,\kappa}$ is a Poisson random variable parametrized by
  \begin{align*}
    \bk{\ell^{\delta_k}}\Nb\pthB{ \sup_{u\in\ivof{\delta_n,4\delta_n}}\bk{\ell^u} \geq \kappa\delta_n^{h_p} } \leq c_2\,\bk{\ell^{\delta_k}} \,\delta_n^{1-\gamma h_p}
  \end{align*}
  where the bound is a consequence of Lemma~3.5 in \cite{Balanca-2015b}. Therefore, using a classic Chernoff inequality on Poisson distributions, we get:
%   \footnote{LM Should it be in the display below $\geq 4c_2\kappa^{-\gamma}\delta_n^{1-\gamma h_p}\bk{\ell^{\delta_k}}\ldots  $?\\PB: Fixed it.\\
% LM2 I am still wondering if $c_2\ldots$ is enough? Last time I got that one needs more e.g. $4c_2\ldots$. I did not check again.\\PB2: I added it, just to be sure we're in the standard setting of Chernoff inequality.}
  \begin{align*}
    \Nb_{\delta_k}\pthcbb{ \bigcup_{\kappa=1}^\infty \#\widehat\Tbb\pth{\delta_k,\delta_n,h_p,\kappa} \geq 4c_2\,\delta_n^{1-\gamma h_p}\bk{\ell^{\delta_k}} + \kappa n^2 }{\Gi_{\delta_k}}
    &\leq \sum_{\kappa=1}^\infty \exp\pthb{-c_3\,\kappa n^2} \\
    &\leq c_4\,\exp\pthb{-c_1\,n^2}.
  \end{align*}
  up to a modification of $c_1$. Combining the two previous estimates, we have obtained: $\Nb\pthb{A_1(\delta_k, \delta_n)\cup A_2(\delta_k,\delta_n,h_p)} \leq c_4\,v(\delta_k)\exp\pthb{-c_1\,n^{3/2}}$. Hence, Markov's inequality entails:
  \begin{align*}
    \Nb_{a}\pthcb{\#\Tbb_A\pth{a,\delta_k,\delta_n,h_p} \geq 1}{\Gi_{a}} \leq c_4\,\bk{\ell^{a}} v(\delta_n)\exp\pthb{-c_1\,n^{3/2}},
  \end{align*}
  and thus, $\Nb\pthb{ \#\Tbb_A\pth{a,\delta_k,\delta_n,h_p} \geq 1 } \leq c_4\, v(\delta_n)\exp\pthb{-c_1\,n^{3/2}}$.

  As a consequence, we get:
  \begin{align*}
    \Nb\pthbb{ \bigcup_{n\in\N} \bigcup_{p=1}^n \bigcup_{\delta_k\in\ivff{\delta_n,2b}}  \bigcup_{j\delta_n\in\ivoo{\delta_k,2b}}   \brcb{\Tbb_A\pth{j\delta_n,\delta_k,\delta_n,h_p} \geq 1} } \leq c_5\sum_{n\in\N} n^22^{n} v(\delta_n)\exp\pthb{-c_1\,n^{3/2}} < \infty,
  \end{align*}
  Borel--Cantelli lemma then entails the desired result.
\end{proof}

As a direct consequence of Lemma~\ref{lemma:unif_bound_Th}, we are also able to bound the local contribution of $\Ti\brc{\delta_n,H_p,\kappa}$ to the local time:
\begin{align}  \label{eq:unif_bound_localtime_Th}
  \ell^t\pthb{ \Bi(\sigma,2r) \cap \Ti\brc{\delta_n,H_p,\kappa} } \leq \min\brcB{\ell^t\pthb{ \Bi(\sigma,2r)}, c_0\,&\kappa \delta_n^{1-(\gamma-1) h_p} \pthB{\ell^t\pthb{ \Bi(\sigma,8r)} + g(r)^{-2}r^{\tfrac{1}{\gamma-1}}} \nonumber \\
  +c_0\,&\kappa^{2}n^2\delta_n^{h_{p}} }.
\end{align}

As another corollary of this result which will be extensively used in the rest of this work, we obtain a uniform bound on the collections $\Tbb\pth{t,\delta,h_<,\kappa}$.
\begin{lemma}  \label{lemma:unif_bound_Th_cor}
  $\Nb$-a.e. there exist two positive constants $\delta_0$ and $c_0$ only depending on $\gamma$ such that for every $\delta\in\ivoo{0,\delta_0}$, all $t\in\ivoo{0,b}$, every $h\in\ivff{0,\infty}$ and any $\kappa\geq 1$
  \begin{align*}
    \#\Tbb\pth{t,\delta,h_<,\kappa} \leq c_0\,\kappa \brcB{\delta^{-(\gamma h-1)\wedge\tfrac{1}{\gamma-1}} \pthb{\bk{\ell^t}+1} +  g(\delta)^{-2} }.
  \end{align*}
\end{lemma}
\begin{proof}
  Set $t\in\ivoo{0,b}$, $\kappa>1$ and $\delta>0$ small enough. Let $h\in\ivff{0,\tfrac{1}{\gamma-1}}$, $n\in\N$ such that $\delta\in\ivfo{\delta_{n-1},\delta_n}$ and  $p_n\in\brc{1,\dotsc,n}$ such that $h\in\ivfo{h_{p_n},h_{p_n+1}}$. Then, observe that $\Tbb\pth{t,\delta,h_<,\kappa}\subset \Tbb\pth{t,\delta_n,\widehat H_{p_n},2^{-h}\kappa}$, and thus according to Lemma~\ref{lemma:unif_bound_Th}
  \begin{align*}
    \#\Tbb\pth{t,\delta,h_\leq,\kappa}
    \leq c_0\,\kappa \brcb{\delta_n^{1-\gamma h_{p_n}} \pthb{\bk{\ell^t}+1} +  g(\delta_n)^{-2} }
    \leq c_1\,\kappa \brcb{\delta^{1-\gamma h} \pthb{\bk{\ell^t}+1} +  g(\delta)^{-2} },
  \end{align*}
  since $\delta_n^{1/n}$ is constant (note that Lemma~\ref{lemma:unif_bound_Th} remains valid if ones replace the condition $\kappa\geq 1$ by $\kappa\geq\kappa_0$, the latter being fixed). Finally, the case $h>\tfrac{1}{\gamma-1}$ is a consequence on the bound on $\Tbb\pth{t,\delta_n,\R,2^{-1/(\gamma-1)}\kappa}$
\end{proof}

We will also require in the following sections a control on the infimum of the local time, as presented in the next lemma.
\begin{lemma}  \label{lemma:inf_local_time}
  $\Nb$-a.e. there exists $N\in\N$ such that for every $n\geq N$, any $\delta_k\in\ivff{\delta_n,b}$ and any $j\delta_n\in\ivoo{0,b}$:
  \begin{align*}
    \forall \Ti_\sigma\in\Tbb(j\delta_n,\delta_k);\quad \inf_{u\in\ivff{\delta_k,\delta_k+4\delta_n}} \bk{\ell^u}(\Ti_\sigma)  + g(\delta_k)^{-2} \delta_k^{\tfrac{1}{\gamma-1}} \geq \tfrac{1}{2}\bk{\ell^{\delta_k}}(\Ti_\sigma).
  \end{align*}
\end{lemma}
\begin{proof}
  Let us set $n\in\N$ and $k\leq n$. As recall in the proof of Lemma~\ref{lemma:unif_bound_Th}, given $\Gi_{\delta_k}$, the process $X:u\mapsto\bk{\ell^{u-\delta_k}}$ is a stable CSBP starting from $\bk{\ell^{\delta_k}}$. Hence, still according to Lemma~3.4 presented in \citet{Balanca-2015b}, for any $x\geq 2k^2\delta_k^{1/(\gamma-1)}$
  \begin{align*}
    \prB[_x]{ \inf_{u\leq4\delta_n} X_u \leq \tfrac{x}{2}-k^2\delta_k^{1/(\gamma-1)}  } \leq \exp\pthb{-c_0\,x\delta_n^{-1/(\gamma-1)}} \leq \exp\pth{-c_1\,n^{3/2}}.
  \end{align*}
  As a consequence of the Ray--Knight theorem,
  \begin{align*}
    \Nb\pthB{ \inf_{u\in\ivff{\delta_k,\delta_k+4\delta_n}} \bk{\ell^u}(\Ti_\sigma)  + g(\delta_k)^{-2} \delta_k^{\tfrac{1}{\gamma-1}} < \tfrac{1}{2}\bk{\ell^{\delta_k}} } \leq v(\delta_k)\exp\pth{-c_0\,n^{3/2}}.
  \end{align*}
  The rest of the proof then follows the structure previously presented in Lemma~\ref{lemma:unif_bound_Th}: summing over $k\leq n$ and $n\in\N$, the branching property of stable trees and Borel--Cantelli lemma entail the desired result.
\end{proof}

Note that as a direct consequence of Lemma~\ref{lemma:inf_local_time}, we also obtain a property on the collections $\Ti\brc{t,\delta_\geq,h_\geq,\kappa}$. Indeed, suppose $n\in\N$, $j\delta_n\in\ivoo{0,b}$, $h\in\ivffb{0,\tfrac{1}{\gamma-1}}$ and $t\in\ivff{j\delta_n,(j+1)\delta_n}$. Then,
\begin{align}  \label{eq:cor_property_Th}
  \sigma\in\Ti\brc{t,\delta_{n\geq},h_\geq,\kappa} \Longrightarrow \sigma'\in\Ti\brc{j\delta_n,\delta_{n\geq},h_\geq,\kappa'}
\end{align}
where $\kappa'=2\max(4^{1/(\gamma-1)}\kappa,n^2)$ and $\sigma'=\iivff{\rho,\delta}\cap\Ti(j\delta_n)$.

%!TEX root = article.tex
% mainfile: article.tex
%%%%%%%%%%%%%%%%%%%%%%%%%%%%%%%%%%%%%%%%%%%%%%%%%%%%%%%%%%%%%%%%%%%%%%%%
%% Stable super-Brownian motion in high dimension
%%%%%%%%%%%%%%%%%%%%%%%%%%%%%%%%%%%%%%%%%%%%%%%%%%%%%%%%%%%%%%%%%%%%%%%%

% \section{Proof of Proposition~\ref{prop:spectrum_excursions}:spectrum of excursion measures} \label{sec:stable_sbm_high_dim}
\section{Proof of Proposition~\ref{prop:spectrum_excursions}: multifractal spectrum  under measure \texorpdfstring{$\N_x$}{Nx}} \label{sec:stable_sbm_high_dim}

In this section, we aim to lift the multifractal structure of random stable trees to multifractal spectrum of measures $\Xii_t$ (recall \eqref{eq:sbm_excursion_measures} for its definition).
As previously outlined, the use of the Lévy snake approach pushes naturally towards the adaptation of techniques existing in the multiparameter Gaussian literature.

\subsection{Upper bound estimates}  \label{ssec:sbm_spectrum_ub}

To start with, we may present a simple connection between the Hölder regularity of the local time on stable trees and the measure-valued excursions of stable super-Brownian motion.
\begin{lemma}  \label{lemma:sbm_bound_holder}
  Suppose $x\in\R^d$. Then, $\N_x$-a.e. for every $t>0$,
  \begin{align}
    \forall\sigma\in\Ti(t);\quad \alpha_{\Xii_t}(W_\sigma) \leq 2 \alpha_\ell(\sigma,\Ti),
  \end{align}
  where $\alpha_{\Xii_t}(W_\sigma)$ and $\alpha_\ell(\sigma,\Ti)$ respectively denote to the pointwise exponents of the measure $\Xii_t(\dt x)$ and $\ell^t(\dt\sigma)$ at $W_\sigma$ and $\sigma$ (note that $\alpha_\ell(\sigma,\Ti)$ can be directly deduced from \eqref{eq:def_pointwise_exponent}).
\end{lemma}
\begin{proof}
  Set $t>0$ and $\sigma\in\Ti(t)$. We recall that for every $\eps>0$, $W$ is $\tfrac{1-\eps}{2}$ Hölder continuous. Hence, for any $r>0$ sufficiently small, $W\pthb{ \Bi(\sigma,r^{2/(1-\eps)}) } \subset B(W_\sigma,r)$.
  As a consequence, the characterisation~\eqref{eq:sbm_excursion_measures} of the excursion measure $\Xii_t$ yields
  \begin{align*}
    \Xii_t\pthb{ B(W_\sigma,r) } = \ell^t\pthb{\brc{\sigma' :\abs{ W_{\sigma'}-W_\sigma}<r}} \geq \ell^t\pthb{ \Bi(\sigma,r^{2/(1-\eps)}) }.
  \end{align*}
  Therefore, $\alpha_{\Xii_t}(W_\sigma) \leq \tfrac{2}{1-\eps}\,\alpha_\ell(\sigma,\Ti)$, proving the lemma as $\eps\rightarrow 0$.
\end{proof}
Introducing the following sets $F(h,\ell^t)$ and $F(h,\Xii_t)$ derived from the iso-Hölder sets \eqref{eq:def_spectrum}:
\begin{align}  \label{eq:sets_F}
  F(h,\ell^t) \eqdef \brcb{\sigma\in\Ti(t) : \alpha_\ell(\sigma,\Ti)\leq h}\quad\text{and}\quad F(h,\Xii_t) \eqdef \brcb{x\in\R^d : \alpha_{\Xii_t}(x)\leq h}
\end{align}
we note that Lemma~\ref{lemma:sbm_bound_holder} entails $\N_x$-a.e.
\begin{align}
  \forall t>0,\ \forall h\geq 0;\quad W(F(h,\ell^t)) \subseteq F(2h,\Xii_t).
\end{align}

If the local mass of the stable SBM can be simply lower bounded by the local time in the tree, it remains more complicated and subtle to obtain an equivalent upper bound. Namely, the key ingredient is to understand the structure of the set $W^{-1}(B(W_\sigma,r))$ and, at least, provide a cover sufficiently optimal of the latter, which will then provide an upper bound of the mass $\Xii_t\pthb{ B(W_\sigma,r) }$.\vsp

For that purpose, let us set in the rest of this section $\epsilon>0$ and for every $n\in\N$, define $r_n\eqdef \delta_n^{(1-\epsilon)/2}$. As previously outlined in the Section~\ref{sec:notations}, the tree-indexed process $(W_\sigma)_{\sigma\in\Ti}$ is known to be $\eta$ continuous on the tree $\Ti$ for any $\eta\in(0,1/2)$. Thus, in what follows, we may and will assume that $n$ is
sufficiently large (greater or equal than some $N_0=N_0(\omega)$) such that
$$   |W_{\sigma}-W_{\sigma'}|\leq r_n,\;\;\;\forall \sigma,\sigma': |\sigma-\sigma'|\leq \delta_n$$

Now, to start with, for every $n\in\N$, $\kappa\geq 1$, $h\in\ivffb{0,\tfrac{1}{\gamma-1}}$ and $\sigma_0\in\Ti(t)$, we will look more closely at the local mass
%\footnote{LM  I did niot understand the expression $\Xii_t\pthb{ B(W_{\sigma_0},r_n)\cap \Ti\brc{t,2\delta_{n\geq},h_\geq,\kappa} }$ and changed it. Please check\\
%PB: Indeed, better like that.}
$\ell^t\pthb{ W^{-1}(B(W_{\sigma_0},r_n))\cap \Ti\brc{t,2\delta_{n\geq},h_\geq,\kappa} }$.
% $\Xii_t\pthb{ B(W_{\sigma_0},r_n)\cap \Ti\brc{t,2\delta_{n\geq},h_\geq,\kappa} }$.
More precisely, setting $j\geq 3$ and $t\in\ivfo{j\delta_n,(j+1)\delta_n}$, we observe that:
\begin{align}  \label{eq:sbm_sum_local_time0}
  \ell^t\pthb{ \brc{ \sigma : \abs{W_{\sigma_0} - W_\sigma} < r_n }\cap \Ti\brc{t,2\delta_{n\geq},h_\geq,\kappa} }
  &\leq \sum_{ \substack{\Ti_{\sigma'}\in\Tbb((j-1)\delta_n,\delta_n)\cap \Ti\brc{t,2\delta_{n\geq},h_\geq,\kappa} \\ \abs{W_{\sigma_0}-W_{\sigma'}}\leq 4r_n }} \ell^t\pthb{ \Ti_{\sigma'} }.
\end{align}
We note that $\abs{W_{\sigma}-W_{\sigma'}}\leq 4r_n$ if $\Ti_{\sigma'}\in\Tbb((j-1)\delta_n,\delta_n)$ and $\sigma\in\Ti_{\sigma'}\cap\Ti(t)$ and recall that the informal notation $\Tbb((j-1)\delta_n,\delta_n)\cap \Ti\brc{t,2\delta_{n\geq},h_\geq,\kappa}$ is licit as for any $\Ti_{\sigma'}\in\Tbb((j-1)\delta_n,\delta_n)$, either $\Ti_{\sigma'}\cap\Ti\brc{t,2\delta_{n\geq},h_\geq,\kappa}=\vset$ or $\Ti_{\sigma'}\cap\Ti(t)\subset\Ti\brc{t,2\delta_{n\geq},h_\geq,\kappa}$.

Then, using the partition of $\Tbb((j-1)\delta_n,\delta_n)$ (and the notations) presented in Definition~\ref{def:collections_Th2} and Lemma~\ref{lemma:unif_bound_Th}, the right hand term in \eqref{eq:sbm_sum_local_time0} is bounded by:
% \footnote{LM can we just bound it by
%  $$c_0\,\kappa \sum_{ \substack{\Ti_{\sigma'}\in\Tbb((j-1)\delta_n,\delta_n,(h_{p-1})_{\geq},\kappa)\cap \Ti\brc{t,2\delta_{n\geq},h_\geq,\kappa} \\ \abs{W_{\sigma_0}-W_{\sigma'}}\leq 5r_n }} \delta_n^{h_p}$$
% for $p: h\in [h_{p-1},h_p)\in \Hbb_n$,
% since others $H_p$ in the sum are irrelevant?\\
% PB: $\Ti\brc{t,2\delta_{n\geq},h_\geq,\kappa}$ contains masses of higher order than just $h$, so I think we need to keep to the sum to consider all the cases appearing in this set.
% }
\begin{align}  \label{eq:sbm_sum_local_time1}
  c_0\,\kappa \sum_{H_p\in\Hbb_n} \ \sum_{ \substack{\Ti_{\sigma'}\in\Tbb((j-1)\delta_n,\delta_n,H_p,\kappa)\cap \Ti\brc{t,2\delta_{n\geq},h_\geq,\kappa} \\ \abs{W_{\sigma_0}-W_{\sigma'}}\leq 5r_n }} \delta_n^{h_p},
\end{align}
since $\ell^t\pthb{ \Ti_{\sigma'} }\leq c_0\kappa\delta_n^{h_p}$ (recall that by definition  $\delta_n^{h_{p-1}}=c\delta_n^{h_{p}}$)
% \footnote{LM I got $c_0\kappa\delta_n^{h_{p-1}}$ here. Please check.\\
% PB: Both are equivalent up to a constant. Want to change?\\
% LM2 Ooops, sorry. Maybe wd  shoud put soemthing like "recall that by definition  $\delta_n^{h_{p-1}}=c\delta_n^{h_{p}}$"\\
% PB2: Ok, just added :)
% }
for any $\Ti_{\sigma'}\in\Tbb((j-1)\delta_n,\delta_n,H_p,\kappa)$.
Moreover, according to Lemma~\ref{lemma:inf_local_time} and Equation~\eqref{eq:cor_property_Th}, for any $\sigma\in\Ti\brc{t,2\delta_{n\geq},h_\geq,\kappa}$ and $\sigma'\eqdef\iivff{\rho,\sigma}\cap\Ti(j\delta_n)$, then $\sigma'\in\Ti\brc{j\delta_n,\delta_{n\geq},h_\geq,\kappa'}$ where $\kappa'=2\max(4^{1/(\gamma-1)}\kappa,n^2)$. As a consequence,
\begin{align*}
  \Tbb((j-1)\delta_n,\delta_n,H_p,\kappa)\cap \Ti\brc{t,2\delta_{n\geq},h_\geq,\kappa}  \subset \Tbb((j-1)\delta_n,\delta_n,H_p,\kappa)\cap \Ti\brc{j\delta_n,\delta_{n\geq},h_\geq,\kappa'},
\end{align*}
and based on \eqref{eq:sbm_sum_local_time1}, our problem is simplified into estimating the size of subsets $\Vbb\subset \Tbb((j-1)\delta_n,\delta_n,H_p,\kappa)\cap \Ti\brc{j\delta_n,2\delta_{n\geq},h_\geq,\kappa'}$ such that for all $\Ti_{\sigma_i},\Ti_{\sigma_l}\in\Vbb$, $\abs{W_{\sigma_i} - W_{\sigma_l}} \leq 8 r_n$. The following lemma aims to obtain a uniform bound on this quantity.
\begin{lemma}  \label{lemma:sbm_clusters_ub}
  Suppose $h\in\ivffb{0,\tfrac{1}{\gamma-1}}$.
% \footnote{LM I guess here and throughout the  proof + in many other places below one should
% change $Q_\Ti$-a.s. by  $\N_x$-a.e.\\
% PB: The tree is set in the proof, only working on the motion, that why I kept this notation. Should maybe detail it?
% \\
% LM2 In the proof of Lemma 3.2, you use Lemma 2.2 (see 1/2 page before the end of the proof where we need the tight bound on certain elements of the tree). The result  in  Lemma 2.2 holds $\mathbf{N}$-a..e. on trees?
% So I do not see how we set the tree in advance unless we say somewhere explicitly that we choose
% the tree on the event from Lemma 2.2 (and maybe we need some other events?)\\
% PB2: I guess you're right. I detail in the proof that $\Ti$ follows the law of stable trees and Lemma 2.2
% }
%$Q_\Ti$-a.s.
$\N_x$-a.e. for all integers $n\in\N$ large enough, any $j\delta_n\in\ivoo{b^{-1},b}$, $t\in\ivfo{j\delta_n,(j+1)\delta_n}$, $H_p\in\Hbb_n$ and any subcollection
% \footnote{LM please check again whether we have $\Tbb(j\delta_n,\delta_n,H_p,\kappa_n)\cap\Ti\brc{(j+1)\delta_n\ldots}$ or $\Tbb((j-1)\delta_n,\delta_n,H_p,\kappa_n)\cap\Ti\brc{j\delta_n\ldots}$. Given our definitions and argument tt seems to me that the latter is right. Please check. Also throughout the proof.}
 $\Vbb\subset \Tbb((j-1)\delta_n,\delta_n,H_p,\kappa_n)\cap\Ti\brc{j\delta_n,\delta_{n\geq},h_\geq,\kappa_n'}$ such that
  \begin{align*}
    \forall \Ti_{\sigma_i},\Ti_{\sigma_l}\in\Vbb;\quad \abs{W_{\sigma_i} - W_{\sigma_l}} \leq 8 r_n,
  \end{align*}
  then the cardinal of the subset $\Vbb$ satisfies
  \begin{align*}
    \#\Vbb \leq \pthb{ \delta_n^{1-\gamma h_p + (d/2\wedge h)} + 1 } \,\delta_n^{-2\epsilon d},
  \end{align*}
  where $\tfrac{d}{2}\wedge h\eqdef\min(\tfrac{d}{2},h)$, $\kappa_n=n^2$ and $\kappa'_n=4^{\gamma/(\gamma-1)}n^2$.
\end{lemma}
Before presenting the proof, let us remark that Lemma~\ref{lemma:sbm_clusters_ub} is mainly interesting in the case $h_p > \tfrac{1}{\gamma}$ since if
% \footnote{LM  I changed to  $h \leq \tfrac{1}{\gamma}$ here?\\
% LM2 You agree?\\ PB2: Yes, it should be fine. Note, in fact, the remark is on $h_p$, not $h$. My bad for the initial mistake.}
$h_p \leq \tfrac{1}{\gamma}$, Lemma~\ref{lemma:unif_bound_Th} already provides a tight bound:
\[
  \#\Vbb \leq \# \Tbb((j-1)\delta_n,\delta_n,H_p,\kappa_n) \leq c\,\kappa_n\log(1/\delta_n)^2.
\]
% \footnote{LM We should give that bound here and there!\\PB2: Is it fine for you like that? I just think it's easier if there is a single case in the lemma. And the bound of Lemma 2.2 does not anyway help to improve the rest.}
\begin{proof}
  $\Ti$ is supposed to follow the law of stable trees and satisfy Lemma 2.2.
  In addition, we set $h\in\ivffb{0,\tfrac{1}{\gamma-1}}$, $n\in\N$, $j\delta_n\in\ivoo{b^{-1},b}$, $H_p\in\Hbb_n$ and for the sake of readability:
%   \footnote{LM I guess we need $\Ti\brc{(j+1)\delta_n,2\delta_{n\geq},h_\geq,\kappa_n'}.$
% (with factor of $2$) below\\
% PB: Fix the equation above Lemma 3.2 to avoid the factor $2$\\
% LM2 I guess you are right}
  \begin{align*}
    \Tbb(j,n,p)\eqdef\Tbb((j-1)\delta_n,\delta_n,H_p,\kappa_n)\cap\Ti\brc{j\delta_n,\delta_{n\geq},h_\geq,\kappa_n'}.
  \end{align*}
  In the proof, we will study a slightly different collection $\Vbb'\subset\Tbb(j,n,p)$. Namely, let $A(j,n,H_p)$ be the event there exists a subcollection $\Vbb'\subset\Tbb(j,n,p)$ such that
  \begin{align*}
    \#\Vbb' > z_n\quad \text{where } z_n \eqdef \pthb{ \delta_n^{1-\gamma h_p + (d/2\wedge h)} + 1 } \,\delta_n^{-2\epsilon d},
  \end{align*}
  and
  \begin{align*}
    \forall\Ti_{\sigma_i},\Ti_{\sigma_l}\in\Vbb';\quad \abs{W_{\varsigma_i} - W_{\varsigma_l}} \leq 12r_n,\quad\text{where }\varsigma_i\eqdef\Ti\pthb{j\delta_n}\cap\iivff{\rho,\sigma_\zeta(\Ti_{\sigma_i})},
  \end{align*}
  recalling that $\sigma_\zeta(\Ti_{\sigma_i})$ denotes the unique extinction vertex of the subtree $\Ti_{\sigma_i}$. We note that the desired bound on the subcollection $\Vbb$ can be easily deduced from an equivalent result on $\Vbb'$.

  Then, setting $m\geq 1$, we define the r.v. $\Ni(j,n,H_p)$ as following:
  \begin{align*}
    \Ni(j,n,H_p) = \underset{\Ti_{\sigma_1},\dotsc,\Ti_{\sigma_m} \text{ distinct}}{\sum\cdots\sum} \indi_{\brcb{ \max_{i,l\leq m}\abs{W_{\varsigma_i} - W_{\varsigma_l}} \leq 12 r_n}},
  \end{align*}
  where for the sake of readability, we omit to recall that the sum is over elements $\Ti_{\sigma_i}\in\Tbb(j,n,p)$. We observe that if such a previous collection $\Vbb'$ exists, we must have $\Ni(j,n,H_p) \geq \binom{\ceil{z_n}}{m}$. In other words, $A(j,n,H_p)\subset\brcb{\Ni(j,n,H_p) \geq \binom{\ceil{z_n}}{m}}$ and Markov inequality then entails
  \begin{align*}
    Q_\Ti\pthb{ \indi_{A(j,n,H_p)} } \leq \binom{\ceil{z_n}}{m}^{-1} Q_\Ti\pthb{ \Ni(j,n,H_p) }.
  \end{align*}
  We thus need to estimate
  \begin{align*}
    Q_\Ti\pthb{\Ni(j,n,H_p) } = \underset{\Ti_{\sigma_1},\dotsc,\Ti_{\sigma_m} \text{ distinct}}{\sum\cdots\sum} Q_\Ti\pthB{\brcB{\max_{1\leq i\neq l\leq m}\abs{W_{\varsigma_i} - W_{\varsigma_l}} \leq 12r_n}}.
  \end{align*}
  Let us fix $\Ti_{\sigma_1},\dotsc,\Ti_{\sigma_{m-1}}\in\Tbb(j,n,p)$ and investigate the sum
  \begin{align*}
    \sum_{\Ti_{\sigma_m}\in\Tbb(j,n,p)} Q_\Ti{\brcB{\max_{1\leq i\neq l\leq m}\abs{W_{\varsigma_i} - W_{\varsigma_l}} \leq 12r_n}}.
  \end{align*}
  Recall that $W$ is a $d$-dimensional Gaussian process whose $d$ components are independent. Furthermore, each of them satisfies the LND property described in Lemma~\ref{lemma:sbm_lnd}, i.e.
  \begin{align*}
    \varcb{W^0_{\varsigma_m}}{W^0_{\varsigma_0},W^0_{\varsigma_1},\cdots,W^0_{\varsigma_{m-1}}} \geq c_{1,m} \min_{0\leq i\leq m-1} d(\varsigma_i,\varsigma_m).
  \end{align*}
  where by convention $\varsigma_0=\rho$.
  Consequently, since the conditional variable is still Gaussian, using simple Gaussian estimates, the term $Q_\Ti{\brcb{\max_{1\leq i\neq l\leq m}\abs{W_{\varsigma_i} - W_{\varsigma_l}} \leq 12r_n}}$ is bounded, up to a constant, by
  \begin{align*}
    &Q_\Ti{\brcB{\max_{1\leq i\neq l\leq m-1}\abs{W_{\varsigma_i} - W_{\varsigma_l}} \leq 12r_n}} \cdot
    r_n^d \pthB{ \min_{1\leq k\leq m-1} d(\varsigma_k,\varsigma_m) }^{-d/2}.
  \end{align*}
  Note that we may omit the root $\rho=\varsigma_0$ as the minimum distance with the other elements at level $j\delta_n$ will always be at most of order $d(\rho,\varsigma_m)$.

  We observe as well that $\brcb{\min_{1\leq k\leq m-1} d(\varsigma_k,\varsigma_m) }^{-d/2} \leq \sum_{k=1}^m  d(\varsigma_k,\varsigma_m)^{-d/2}$, and as a consequence, our problem is reduced to the study of the following type of sum, for any $k\in\brc{1,\dots,m-1}$:
  \begin{align*}
    \underset{\Ti_{\sigma_1},\dotsc,\Ti_{\sigma_m} \text{ distinct}}{\sum\cdots\sum} Q_\Ti{\brcB{\max_{1\leq i\neq l\leq m-1}\abs{W_{\varsigma_i} - W_{\varsigma_l}} \leq 12r_n}} \cdot\,r_n^d \, d(\varsigma_k,\varsigma_m)^{-d/2}.
  \end{align*}
  Since the latter quantity is invariant under permutations on the index set $\brc{1,\dots,m-1}$, we may assume without any loss of generality that $k=m-1$ and then proceed by induction. At the end, our study is therefore simplified into the investigation of the following quantity:
  \begin{align*}
    r_n^{d(m-1)} \underset{\Ti_{\sigma_1},\dotsc,\Ti_{\sigma_m} \text{ distinct}}{\sum\cdots\sum} \prod_{i=1}^{m-1} d(\varsigma_i,\varsigma_{i+1})^{-d/2}.
  \end{align*}
  Note that the use of the collection $\Vbb'$ (instead of $\Vbb$) and the tree configuration ensure that for every $i\in\brc{1,\dotsc,m-1}$, $d(\varsigma_i,\varsigma_{i+1})\geq 2\delta_n$.
  Hence, setting once more $\Ti_{\sigma_1},\dotsc,\Ti_{\sigma_{m-1}}\in\Tbb(j,n,p)$, we observe
  \begin{align*}
    \sum_{\Ti_{\sigma_{m}}\neq\Ti_{\sigma_{m-1}}} d(\varsigma_{m-1},\varsigma_{m})^{-d/2}
    &\leq 2\sum_{\delta_k\in\ivff{\delta_n,b}}  \sum_{\Ti_{\sigma_m}\in \Bi(\varsigma_{m-1},2\delta_k)\cap\Tbb(j,n,p)} \delta_k^{-d/2},
    % &\leq \sum_{k=-K}^n 2^{k(d/2-h')} \delta_n^{1-\gamma h'' -\epsilon}
  \end{align*}
  Based on Definition~\ref{def:collections_Th1} of $\Ti\brc{j\delta_n,\delta_{n\geq},h_\geq,\kappa_n'}$, Lemma~\ref{lemma:unif_bound_Th} provides a tight bound on the number of elements in $\Bi(\varsigma_{m-1},2\delta_k)\cap\Tbb((j-1)\delta_n,\delta_n,H_p,\kappa_n)\cap\Ti\brc{j\delta_n,\delta_{n\geq},h_\geq,\kappa_n'}$, entailing
% \footnote{LM we have $n^4$, below; why not $n^2$? although it
% does not change anything\\
% PB: One $n^2$ comes from Lemma 2.2 and the other one from $\kappa_n'$\\
% LM2 Fine }
% \footnote{LM As we discussed and is mentioned immediately after the formulation of the lemma, the lemma is interesting only for $h>1/\gamma$. However as I mentioned in the footnote after Lemma 2.2, one can try
% to optimize it even for $h>1/\gamma$ (but less  than $1/(\gamma-1)$. Here one can try to use it by first giving the flat bound on a number of subtrees for $\delta_k$ small enough (like smaller than $\delta_n^\eta$ for some
% $\eta\in (0,1)$) all these without using LND technique; and then for larger $\delta_k$ use LND
% but here we may gain since now $\delta_k^{-d/2}$ is not that large as  $\delta_n^{-d/2}$!
% Thus, my guess is that we could get finite number of trees for $d\geq 2/(\gamma-1)$ and $h<1/(\gamma-1)$. Or maybe even for more (or less :)) parameters. \\
% Again, no need to put it in the paper, but let us just keep it in mind.\\
% PB: Thanks! But I guess it would not change the lemma result, since at scales "larger" than $\delta_n$, the LND seems to me as optimal.\\
% LM2 We will see. Anyway it is not relevant now --- we can disciss it some other time. It was just a side remark.
%  }
  \begin{align*}
    \sum_{\Ti_{\sigma_{m}}\neq\Ti_{\sigma_{m-1}}} d(\varsigma_{m-1},\varsigma_{m})^{-d/2}
    &\leq c_0\, n^4 \sum_{\delta_k\in\ivff{\delta_n,b}} \delta_k^{-d/2} \pthb{ \delta_n^{1-\gamma h_p } \delta_k^h + 1 } \\
    &\leq c_1\, n^4 \brcB{ \delta_n^{1-\gamma h_p + (h-d/2)\wedge 0 } + \delta_n^{-d/2} },
  \end{align*}
  recalling that whenever $\Bi(\varsigma_{m-1},2\delta_k)\cap\Ti\brc{j\delta_n,\delta_{n\geq},h_\geq,\kappa_n'}\neq\vset$, $\ell^{j\delta_n}\pthb{\Bi(\varsigma_{m-1},2\delta_k)} \leq \kappa_n' \delta_k^h$.
  As a consequence, since $r_n=\delta_n^{(1-\epsilon)/2}$, we obtain by induction
  \begin{align*}
    Q_\Ti\pthb{\Ni(j,n,H_p) }
    &\leq c_1\,c_{1,m}^m m! \, n^{4m} r_n^{d(m-1)}\delta_n^{-(m-1)d/2} \brcB{ \delta_n^{1-\gamma h_p + (d/2\wedge h)} + 1 }^{m-1}\delta_n^{-1/(\gamma-1)} \\
    &\leq c_2\,c_{1,m}^m m!\, \brcB{ \delta_n^{1-\gamma h_p + (d/2\wedge h)} + 1 }^{m} \delta_n^{-md\epsilon-1/(\gamma-1)},
  \end{align*}
  where the term $n^{2}\delta_n^{-1/(\gamma-1)}$ comes from the last stage of the induction, summing over all $\Ti_{\sigma_{1}}\in\Tbb((j-1)\delta_n,\delta_n,H_p,\kappa_n)\cap\Ti\brc{j\delta_n,\delta_{n\geq},h_\geq,\kappa_n'}$.   Recalling $z_n =\pthb{ \delta_n^{1-\gamma h_p + (d/2\wedge h)} + 1 } \,\delta_n^{-2d\epsilon}$ and observing that $\binom{\ceil{z_n}}{m} \geq z_n^m \,m^{-m}$, we get
  \begin{align*}
    Q_\Ti\pthb{ \indi_{A(j,n,H_p)} }
    &\leq c_2\,c_{1,m}^m m^{2m} \, \delta_n^{md\epsilon-1/(\gamma-1)}.
  \end{align*}
  The parameter $m\geq 1$ can then be chosen sufficiently large to obtain $Q_\Ti\pthb{ \indi_{A(j,n,H_p)} } \leq c_2\,c_{1,m}^m m^{2m}\, \delta_n^{2}$. Finally, as we aim to obtain a uniform bound on $j\delta_n\in\ivoo{b^{-1},b}$ and $H_p\in\Hbb_n$,
  \begin{align*}
    \sum_{n\in\N} Q_\Ti\pthbb{ \bigcup_{j\delta_n\in\ivoo{b^{-1},b}} \bigcup_{H_p\in\Hbb_n} A(j,n,H_p) } \leq b c_2\,c_{1,m}^m m^{2m} \sum_{n\in\N} n \delta_n < \infty.
  \end{align*}
  Borel--Cantelli lemma then concludes the proof.
\end{proof}
Note that the previous lemma is inspired by the literature on multiparameter Gaussian processes, and in particular the seminal works of \citet{Kaufman-1969} and \citet{Khoshnevisan.Wu.ea-2006} on the uniform fractal geometry of Brownian sheet.

Based on the previous lemma, we may estimate the contribution to the local time $\ell^t\pthb{ \brc{ \sigma : \abs{W_{\sigma_0} - W_\sigma} < r_n }\cap \Ti\brc{t,2\delta_{n\geq},h_\geq,n^2} }$.
\begin{lemma}  \label{lemma:sbm_cluster_mass}
  Suppose $h\in\ivffb{0,\tfrac{1}{\gamma-1}}$. $\N_x$-a.e.
  % \footnote{LM again $\N_x$-a.e.?}
  for all integers $n$ large enough, any $j\delta_n\in\ivoo{b^{-1},b}$ and every $t\in\ivfo{j\delta_n,(j+1)\delta_n}$:
  \begin{align*}
    \forall \sigma_0\in\Ti(t);\quad\ell^t\pthb{ \brc{ \sigma : \abs{W_{\sigma_0} - W_\sigma} < r_n }\cap \Ti\brc{t,2\delta_{n\geq},h_\geq,n^2} } \leq c_0\,\delta_n^{(d/2\wedge h)-3\epsilon d}
  \end{align*}
  for some positive constant $c_0$ independent of $n$ and $j$.
\end{lemma}
\begin{proof}
  Set $h\in\ivffb{0,\tfrac{1}{\gamma-1}}$. For any integer $n$ large enough and $j\delta_n\in\ivoo{b^{-1},b}$, observe first that
  \begin{align*}
    \Tbb((j-1)\delta_n,\delta_n,H_p,n^2)\cap \Ti\brc{j\delta_n,2\delta_{n\geq},h_\geq,\kappa'_n} = \vset \quad\text{whenever}\quad \delta_n^{h_{p}} > 4^{\gamma/(\gamma-1)} \delta_n^h,
  \end{align*}
  where we recall $\kappa'_n=4^{\gamma/(\gamma-1)}n^2$.
  Hence, according to the bound presented in Lemma~\ref{lemma:sbm_clusters_ub} and bounds \eqref{eq:sbm_sum_local_time0} and \eqref{eq:sbm_sum_local_time1}, for any $t\in\ivff{j\delta_n,(j+1)\delta_n}$ and any $\sigma_0\in\Ti(t)$:
  \begin{align*}
    \ell^t\pthb{ \brc{ \sigma : \abs{W_{\sigma_0} - W_\sigma} < r_n }\cap \Ti\brc{t,2\delta_{n\geq},h_\geq,n^2} }
    &\leq c_0\,n^2 \sum_{H_p\in\widehat\Hbb_n} \delta_n^{h_p} \pthb{ \delta_n^{1-\gamma h_p + (d/2\wedge h)} + 1 } \,\delta_n^{-2\epsilon d},
  \end{align*}
  where  $\widehat\Hbb_n=\brcb{H_p\in\Hbb_n : \delta_n^{h_{p}} \leq 4^{\gamma/(\gamma-1)} \delta_n^h}$. As a consequence, since $1-(\gamma-1) h_p\geq 0$,
  \begin{align*}
    \ell^t\pthb{ \brc{ \sigma : \abs{W_{\sigma_0} - W_\sigma} < r_n }\cap \Ti\brc{t,2\delta_{n\geq},h_\geq,n^2} }
    &\leq c_1\pthb{ \delta_n^{d/2\wedge h} + \delta_n^{h} } \,\delta_n^{-3\epsilon d} \leq 2c_1\,\delta_n^{d/2\wedge h-3\epsilon d}.
  \end{align*}
\end{proof}

From the previous estimates, we may finally deduce  the upper bounds on dimensions that appear in
%first part of the proof of
Proposition~\ref{prop:spectrum_excursions}.
\begin{lemma}  \label{lemma:sbm_spectrum_ub}
  $\N_x$-a.e. for any Borel set $G\subset\ivoo{0,\infty}$:
%\footnote{LM2 Should we say (starting from Theorem~2?) something
%about
%the case when the right hand side is negative --- that then dim=0. This is in many places. Just now paid attention}
  \begin{align*}
    \forall h\in\ivffb{\tfrac{1}{\gamma},\tfrac{1}{\gamma-1}}\cap\ivfob{0,\tfrac{d}{2}};\quad \dimH \bigcup_{s\in G} F(2h,\Xii_s) \leq 2\gamma h - 2 + 2\dimP G,
  \end{align*}
  and in particular, $\dimH \,F(2h,\Xii_t) \leq 2\gamma h - 2$. Moreover,
  \begin{align*}
    \forall h\in\ivfob{0,\tfrac{1}{\gamma}}\cap\ivfob{0,\tfrac{d}{2}};\quad \dimH \brcb{t>0 : F(2h,\Xii_t) \neq \vset } \leq \gamma h.
  \end{align*}
  Finally, assuming $d\geq 2$, for any Borel set $G\subset\ivoo{0,\infty}$, $\N_x$-a.e.
  \begin{align*}
    \inf_{x\in\R^d} \inf_{s\in G} \alpha_{\Xii_s}(x) = \frac{2-2\dimP G\cap\ivoo{0,h(\Ti)}}{\gamma}.
  \end{align*}
\end{lemma}
\begin{proof}
  Let us first note that by a standard regularisation argument, it is sufficient to prove the upper bound replacing the packing dimension by the upper box dimension of $G$ (we refer for instance to \cite[Lemma 4.5]{Balanca-2015b} for the details of this argument).

  The proof of the upper bound of the multifractal spectrum is divided in two main steps: we first start by constructing a proper cover of $F(2h,\Xii_t)$ using the previous lemmas, and then we deduce a bound on the Hausdorff dimension. For that purpose, let us set $\epsilon>0$, $h\in\ivffb{4\epsilon d,\tfrac{1}{\gamma-1}}\cap\ivfo{0,\tfrac{d}{2}}$ and $t\in\ivoo{b^{-1},b}$.
%   \footnote{LM2 Why $h\in (8\epsilon d,\cdot]$?
% $4\epsilon d$ should be enough. Also why $h>4\epsilon d$ and not $h\geq 4\epsilon d$? It might be important in the 2nd part of the proof.\\PB2: Indeed, it's $4\epsilon d$. I took the strict inequality not to be annoyed by the limit case (but it should be fine after checking).
% The monotonicity of the collection $h\rightarrow F(h)$ was anyway enough to get the bound for $h=0$.}
  For every $n\in\N$, we define:
  \begin{align}  \label{eq:def_cover_Gth}
    V(t,\delta_n,h) \eqdef \bigcup_{\delta_k\in\ivff{\delta_n,2b}} \ \ \bigcup_{\sigma\in\Ti(t):\ell^t(\Bi(\sigma,2\delta_k) ) > n^2\delta_k^h } B(W_\sigma,r_k),
  \end{align}
  where $r_k=\delta_k^{(1-\epsilon)/2}$. We also introduce the limit set $V(t,h) = \limsup_{n\rightarrow\infty} V(t,\delta_n,h)$. Suppose $\sigma_0\in\Ti(t)$ is such that $W_{\sigma_0}\notin V(t,h)$, implying the existence of $N\in\N$ such that for all integers $n\geq N$:
  \begin{align*}
    \forall\delta_k\in\ivff{\delta_n,b},\ \forall \sigma\in\Ti(t) : W_\sigma\in B(W_{\sigma_0},r_k);\quad \ell^t(\Bi(\sigma,2\delta_k) ) \leq n^2\delta_k^h.
  \end{align*}
  As a consequence, for every integer $n\geq N$ and any $\sigma\in\Ti(t)$ such that $W_\sigma\in B(W_{\sigma_0},r_n)$, we get $\sigma\in\Ti\brc{t,2\delta_{n\geq},h_\geq,n^2}$.
  Therefore, using in addition the estimate obtained in Lemma~\ref{lemma:sbm_cluster_mass}, we have for all $n\geq N$:
  \begin{align*}
     \Xii_t\pthb{ B(W_{\sigma_0},r_n) } = \ell^t\pthb{\brc{\sigma :\abs{ W_{\sigma_0}-W_\sigma}<r_n}\cap \Ti\brc{t,2\delta_{n\geq},h_\geq,n^2}}
     \leq c_0\,\delta_n^{h-3\epsilon d}
     \leq c_0\, r_n^{2h-6\epsilon d},
  \end{align*}
  implying
 %  \footnote{LM What is going on if $h=0$?
 % $ F(-8\epsilon d,\Xii_t)$ is not defined in this case\ldots?\\
 % PB: True! Added above that $h>8\epsilon d$}
 that $W_{\sigma_0}\notin F(2h-8\epsilon d,\Xii_t)$, and thus the desired covering property: $F(2h-8\epsilon d,\Xii_t) \subset V(t,h)$. \vsp

  Hence, we may now focus on obtaining a proper bound on the Hausdorff dimension of $V(t,h)$. For that purpose, let us now restrict to the situation $h\in\ivffb{\tfrac{1}{\gamma},\tfrac{1}{\gamma-1}}\cap\ivfo{0,\tfrac{d}{2}}$. For any $\sigma\in V(t,h)$, $n$-infinitely often, there exists $\delta_k\in\ivff{\delta_n,2b}$ such that $\ell^t(\Bi(\sigma,2\delta_k) ) > n^2\delta_k^h$. Nevertheless, since $\bk{\ell^t}<\infty$, one must have $\delta_k\leq \pthb{\bk{\ell^t}n^{-2}}^{1/h}\rightarrow_{n\rightarrow\infty} 0$. Hence, according to definition~\eqref{eq:def_cover_Gth}, for every $n\in\N$,
  \begin{align*}
    V(t,\delta_n,h) \subset \bigcup_{\delta_k \leq \pth{\bk{\ell^t}n^{-2}}^{1/h}} \ \ \bigcup_{\sigma\in\Ti(t):\ell^t(\Bi(\sigma,2\delta_k) ) > \delta_k^h } B(W_\sigma,r_k).
  \end{align*}
  Consequently, noting that the second union does not depend any more on $n\in\N$ and using in addition Definition~\ref{def:collections_Th2}, we obtain:
  \begin{align}  \label{eq:cover_Gth}
    V(t,h)
    &\subset\varlimsup_{n\rightarrow\infty} \bigcup_{\sigma\in\Ti(t):\ell^t(\Bi(\sigma,2\delta_n) ) > \delta_n^h } B(W_\sigma,r_n)
    &\subset \varlimsup_{n\rightarrow\infty} \bigcup_{ \Ti_{\sigma'}\in\Tbb\pth{(j-1)\delta_n,\delta_n,h_<,1} } B(W_{\sigma'},3r_n),
  \end{align}
  where $j\geq 1$ is such that $t\in\ivfo{j\delta_n,(j+1)\delta_n}$ and observing that for any $\sigma\in\Ti_{\sigma'}\cap\Ti(t)$, $\abs{W_{\sigma}-W_{\sigma'}}\leq 2r_n$. Then, recall that Lemma~\ref{lemma:unif_bound_Th_cor} provides a tight bound on the size of $\Tbb\pth{(j-1)\delta_n,\delta_n,h_<,1}$, uniformly in $j\delta_n\in\ivoo{0,b}$: for all $n$ large enough and every $j\delta_n\in\ivoo{0,b}$:
  \begin{align}  \label{eq:bound_cover}
    \#\Tbb\pth{(j-1)\delta_n,\delta_n,h_<,1} \leq c_0\,\delta_n^{1-\gamma h} \brcB{\sup_{u\geq 0} \,\bk{\ell^u}+1} \leq c_0\,r_n^{2-2\gamma h} \brcB{\sup_{u\geq 0} \,\bk{\ell^u}+1}.
  \end{align}

  Assuming $G\subset\ivoo{b^{-1},b}$ is a Borel set, we denote by $\Ji(n)$ the collection $n$-dyadic intervals necessary to cover the former. If $s=\dimBu G$, we know that for every $n\in\N$ sufficiently large, $\#\Ji(n)\leq 2^{n(s+\epsilon)}$. As a consequence, using the previous covering \eqref{eq:cover_Gth} of $V(t,h)$ and the bound \eqref{eq:bound_cover}, one easily deduces as well a proper covering of the set $\cup_{s\in G} V(s,h)$, therefore entailing the desired bound:
  \begin{align*}
    \dimH \bigcup_{s\in G} F(2h-8\epsilon d,\Xii_s) \leq 2\gamma h - 2 + 2\dimBu G + 2\epsilon.
  \end{align*}
  The limit $\epsilon\rightarrow 0$ and $b\rightarrow\infty$ then concludes the first part of the proof.\vsp

  Fix again arbitrary $\epsilon>0$ and $h\in\ivff{4\epsilon d,\tfrac{d}{2}\wedge\tfrac{1}{\gamma}}$.
  % Let us now investigate the  % second
% \footnote{LMt2 I took out "second"}
% case $h\in\ivfob{0,\tfrac{d}{2}\wedge\tfrac{1}{\gamma}}$.
% \footnote{LM2: I would say: "Fix again arbitrary $\epsilon>0$, $h\in[4\epsilon d,\tfrac{d}{2}\wedge\tfrac{1}{\gamma}]$." --- I guess, we are still interested in $h\geq 4\epsilon d$? Then by taking  $h=4\epsilon d$ we can handle $F(0,\cdot)$ case? --- This is why I meant in the previous remark that $h\geq 4\epsilon d$ (vis $h>4\epsilon d$) is important.\\ PB2: Ok, changed it to your suggestion.}
% For that purpose, for
For  every $n\in\N$, let us introduce the following r.v.
  \begin{align*}
    N(n) \eqdef \#\brcb{j\delta_n\in\ivoo{0,b} : \Tbb\pth{(j-1)\delta_n,\delta_n,h_<,1} \neq\vset }.
  \end{align*}
  Then, according to the proof of \cite[Lemma 4.6]{Balanca-2015b}, for any $\epsilon>0$ and every $n$ sufficiently large: $N(n) \leq \delta_n^{-\gamma h - 3\epsilon}$. In particular, based on the covering property~\eqref{eq:cover_Gth} of $V(t,h)$, we get:
  \begin{align*}
    \dimH \brcb{t>0 : F(2h-8\epsilon d,\Xii_t) \neq \vset } \leq \gamma h-3\epsilon,
  \end{align*}
  which entails the desired bound as $\epsilon\rightarrow 0$.\vsp

  Finally, we may consider the last equality for a fixed Borel set $G\subset\ivoo{b^{-1},b}$. The lower bound inequality
  \begin{align*}
    \inf_{x\in\R^d} \inf_{s\in G} \alpha_{\Xii_s}(x) \leq \frac{2-2\dimP G\cap\ivoo{0,h(\Ti)}}{\gamma}.
  \end{align*}
  is trivial using Lemma~\ref{lemma:sbm_bound_holder} and the same result \cite[Th. 5]{Balanca-2015b} on stable trees. To obtain the other side inequality, we note that the assumption $d\geq 2$ ensures that $\inf_{x\in\R^d} \inf_{s\in G} \alpha_{\Xii_s}(x) < d$. Then, the former is a consequence of the cover \eqref{eq:cover_Gth}  and \cite[Lemma 4.6]{Balanca-2015b} which proves that $F(h,\ell)\cap\Ti(G)$ is empty whenever $h<\tfrac{1-\dimP G\cap\ivoo{0,h(\Ti)}}{\gamma}$.
\end{proof}

\subsection{Lower bound estimates}  \label{ssec:sbm_spectrum_lb}

The lower bound on the multifractal spectrum of super-Brownian motion will also make use of the previous work \cite{Balanca-2015b} investigating the multifractal structure of stable trees. Consequently, we will briefly recall a few elements and notations introduced in \cite{Balanca-2015b}. For that purpose, let us fix in this section a closed interval $\Hi\subset\ivof{\tfrac{1}{\gamma},\tfrac{1}{\gamma-1}}$ and $\epsilon>0$. Then, according to \cite[Lemma 4.21]{Balanca-2015b}, $\Nb(\dt\Ti)$-a.e. for every $h\in\Hi$ and any $t\in\ivoo{\epsilon,h(\Ti)-\epsilon}$, there exist $G(t,h)\subset \Ti(t)$ and a probability measure $\mu_{t,h}(\dt\sigma)$ supported by $G(t,h)$ such that for any $\sigma\in G(t,h)$, $\alpha_\ell(\sigma,\Ti) \leq h$, and
\begin{align}  \label{eq:mass_pple_spectrum1}
  \forall r>0;\quad \mu_{t,h}\pthb{ \Bi(\sigma,r) } \leq r^{\gamma h-1-\eps(r)} \log(1/r)^\eta
\end{align}
where $\eta>0$ is independent of $t$ and $h$, and $\eps(\cdot)$ is a positive non-decreasing function satisfying $\lim_{\eps\rightarrow 0} \eps(r) = 0$.

As a natural way to obtain a lower bound on the multifractal spectrum of stable super-Brownian motion, we will prove that the pushforward measure $W_\star\pth{ \mu_{t,h} }$ satisfies as well a proper mass distribution principle. More specifically, we define for every $h\in\Hi$, the measure $\nu_{t,2h}(\dt x)$:
\begin{align}  \label{eq:pushforward_measure}
  \forall V\in\Bi(\R^d);\quad \nu_{t,2h}(V) \eqdef \mu_{t,h}(W^{-1}(V)).
\end{align}
Owing to Lemma~\ref{lemma:sbm_bound_holder}, we already know that $\supp\nu_{t,2h} \subset W(F(h,\ell^t)) \subset F(2h,\Xii_t)$. Consequently, using as well the upper bound on the spectrum obtained in Lemma~\ref{lemma:sbm_spectrum_ub}, it only remains to prove that $\nu_{t,2h}$ satisfies a proper mass distribution principle. We will adopt a strategy similar to Lemma~\ref{lemma:sbm_clusters_ub} and, using the local nondeterminism property (Lemma \ref{lemma:sbm_lnd}), estimate properly the accumulation phenomena that might appear on $\nu_{t,2h}(\dt x)$.\vsp

For that purpose, we need to recall a few more elements concerning the construction of the collection of measures $\mu_{t,h}(\dt\sigma)$. To begin with, let $(\rho_n)_{n\in\N}$ be a fast decreasing sequence to zero such that $\rho_n = 2^{-\rho_{n-1}^{-1}}$ and $(\Hi_n)_{n\in\N}$ denote the dyadic-like approximations of elements in $\Hi$. Then, as presented in \cite{Balanca-2015b}, for every $n\in\N$, $j\rho_n\in\ivoo{\epsilon,h(\Ti)-\epsilon}$ and $h\in\Hi_n$, there exists a non-empty collection of subtrees $\Vbb_{n}(j,h)\subset\Tbb((j-1)\rho_{n},\rho_{n})$ such that
\begin{align*}
  \forall \Ti_\sigma\in\Vbb_{n}(j,h);\quad \inf_{u\in\ivff{\rho_{n}/2,2\rho_{n}}}\bk{\ell^u}(\Ti_\sigma) \geq g(\rho_{n})^{-\alpha}\rho_{n}^h,
\end{align*}
where $\alpha>1$ is a fixed positive real. In addition, we also set $\Vbb_{n}(h) \eqdef \cup_{j\rho_{n}\in\ivoo{\epsilon,h(\Ti)-\epsilon}} \Vbb_{n}(j,h)$. The construction presented in \cite{Balanca-2015b} then ensures that the collections $(\Vbb_{n}(h))_{n\in\N,h\in\Hi_n}$ are nested, allowing to define $G(t,h)$ as following:
\begin{align*}
  G(t,h) \eqdef \bigcap_{n\in\N} G(t,h,n)\quad \text{where }\  G(t,h,n) \eqdef \bigcup_{\Ti_\sigma\in \Vbb_n(k_n,h_n)} \Ti_\sigma\cap \Ti(t),
\end{align*}
and the sequences $(k_n)_{n\in\N}$ and $(h_n)_{n\in\N}$ are such that $t\in\ivfo{k_n\rho_n,(k_n+1)\rho_n}$ and $h_n\rightarrow h$. The nested structure of the collections $(\Vbb_{n}(h))_{n\in\N,h\in\Hi_n}$ allows to construct $\mu_{t,h}(\dt\sigma)$ in a Cantor-like fashion: starting with $\mu_{t,h,0} = \ell^t$, ones define $\mu_{t,h,n+1}$ by spreading the mass of $\mu_{t,h,n}$ ``uniformly'' on the set $G(t,h,n+1)$ (we refer to \cite[Lemma 4.21]{Balanca-2015b} for the precise description of the construction). The Cantor-like structure ensures the convergence to a finite measure $\mu_{t,h}$ supported by $G(t,h)$ and Lemma~4.21 in \cite{Balanca-2015b} then proves a proper mass distribution principle \eqref{eq:mass_pple_spectrum1} on the former.

In order to prove an equivalent property on the measures $\nu_{t,2h}$, we need a more precise description of the properties the collection $(\Vbb_{n}(h))_{n\in\N,h\in\Hi_n}$. We provide for that purpose in the next lemma a construction of a dyadic-like collection of nodes related to the former.
\begin{lemma}  \label{lemma:trees_suitable_dyadics}
  Suppose $\Hi\subset\ivofb{\tfrac{1}{\gamma},\tfrac{1}{\gamma-1}}$ is a fixed closed interval and $\epsilon>0$. $\Nb(\dt \Ti)$-a.e., for every $n$ large enough, any $\delta_k\in\ivfo{\rho_{n},\rho_{n-1}}$, $u\in\Di_k=\brcb{ \ivff{j\delta_k,(j+1)\delta_k} : j\in\Z}$ and $h\in\Hi_n$, there exists a collection of nodes $\Vi_n(u,\delta_k,h)\subset\Ti(u)$ satisfying the following properties:
  \begin{enumerate}[(i)]
    \item for every $j\rho_n\in\ivoo{\epsilon,h(\Ti)-\epsilon}$, there exists $u\in\Di_k\cap\ivof{0,j\rho_n}$ such that
    \begin{align*}
      \forall \Ti_\sigma\in\Vbb_n(j,h), \ \exists \sigma'\in\Vi_n(u,\delta_k,h);\quad d(\sigma,\sigma') \leq 4\delta_k;
    \end{align*}
    \item for every $u\in\Di_k$, $h\in\Hi_n$ and any $\sigma\in\Vi_n(u,\delta_k,h)$,
    \begin{align*}
      \#\pthb{ \Bi(\sigma,2r) \cap \Vi_n(u,\delta_k,h) } \leq c_0
      \begin{cases}
        1 + \pthb{r\vartheta_n^{-1}}^{\tfrac{1}{\gamma-1}} g(r)^{-\beta}  &\text{if } \delta_k\in\ivfo{\rho_n,\vartheta_n}
        \\ &\text{and }r\geq\delta_k; \\
        \pthb{r\delta_k^{-1}}^{\tfrac{1}{\gamma-1}} g(\delta_k)^{-\beta} &\text{if } \delta_k\in\ivfo{\vartheta_{n},\rho_{n-1}} \\
        & \text{and } r\in\ivfo{\delta_k,\rho_{n-1}}; \\
        \pthb{\rho_{n-1}\delta_k^{-1}}^{\tfrac{1}{\gamma-1}} g(\delta_k)^{-\beta}\pthB{1+\pthb{r\vartheta_{n-1}^{-1}}^{\tfrac{1}{\gamma-1}} } & \text{if } \delta_k\in\ivfo{\vartheta_{n},\rho_{n-1}} \\
         &\text{and } r\geq\rho_{n-1}, \\
      \end{cases}
    \end{align*}
    where the constants $\beta$ and $c_0$ are independent of the parameters $n$, $k$ and $h$, and $\vartheta_n \eqdef \rho_n^{(\gamma-1)(\gamma h -1 )} g(\rho_n)^{-\alpha\gamma(\gamma-1)} \geq \rho_n$.
  \end{enumerate}
  Finally, we will simply denote by $\Vi_n(\delta_k,h)$ the full collection $\cup_{u\in\Di_k} \Vi_n(u,\delta_k,h)$.
\end{lemma}
The technical proof of this lemma is mostly the continuation of the construction presented in \cite{Balanca-2015b} and does not rely on any new estimates on continuous stable trees. As a consequence, for the sake of readability, we only presented the former in Appendix~\ref{sec:appendix}.

Similarly to Lemma~\ref{lemma:sbm_clusters_ub}, we may now investigate the accumulation behaviour on the collections $\Vi_n(u,\delta_k,h)$, $u\in\Di_k$.
\begin{lemma}  \label{lemma:sbm_clusters_lb1}
  Suppose $x\in\R^d$. $\N_x$-a.e. for every $n$ large enough, any $h\in\Hi_n$, $\delta_k\in\ivfo{\rho_n,\rho_{n-1}}$, $u\in\Di_k$ and any subcollection $\Vi\subset\Vi_n(u,\delta_k,h)$ satisfying
  \begin{align*}
    \forall \sigma_i,\sigma_j\in\Vi;\quad \abs{W_{\sigma_i} - W_{\sigma_j}} \leq 8r_k\eqdef8\delta_k^{(1-\epsilon)/2},
  \end{align*}
  the cardinal of the subset $\Vi$ satisfies:
  \begin{align}  \label{eq:sbm_clusters_lb1}
    \#\Vi \leq
    \begin{cases}
      r_k^{d-\epsilon} \vartheta_n^{-\pthb{\tfrac{1}{\gamma-1}\vee\tfrac{d}{2}}} & \text{if } \delta_k\in\ivfo{\rho_n,\vartheta_n}; \\
      r_k^{d-\epsilon} \pthbb{ \delta_k^{-\pthb{\tfrac{1}{\gamma-1}\vee\tfrac{d}{2}}} \rho_{n-1}^{\pthb{\tfrac{1}{\gamma-1}-\tfrac{d}{2}}\vee 0}+ \pthb{\rho_{n-1}\delta_k^{-1}}^{\tfrac{1}{\gamma-1}} \vartheta_{n-1}^{-\pthb{\tfrac{1}{\gamma-1}\vee\tfrac{d}{2}}} } &\text{if } \delta_k\in\ivfo{\vartheta_n,\rho_{n-1}} . \\
    \end{cases}
  \end{align}
  where we recall the notation $\vartheta_n \eqdef \rho_n^{(\gamma-1)(\gamma h -1 )} g(\rho_n)^{-\alpha\gamma(\gamma-1)} \geq \rho_n$.
\end{lemma}
\begin{proof}
  The structure of the proof is clearly similar to Lemma~\ref{lemma:sbm_clusters_ub} and we will therefore omit technical details that remain the same. Let us set $n\in\N$, $\delta_k\in\ivfo{\rho_n,\rho_{n-1}}$, $h\in\Hi_n$ and $u\in\Di_k$ and $A(u,k,n,h)$ be the event there exist a subcollection $\Vi\subset\Vi_n(u,\delta_k,h)$ such that $\forall \sigma_i,\sigma_j\in\Vi$, $\abs{W_{\sigma_i} - W_{\sigma_j}} \leq 8\delta_k^{(1-\epsilon)/2}$ and $\#\Vi > z_k$, the latter denoting the upper bound presented in Equation~\eqref{eq:sbm_clusters_lb1}. Then, set $p\geq 1$ and define as well the random variable $\Ni(u,k,n,h)$:
  \begin{align*}
    \Ni(u,k,n,h) = \underset{\sigma_1,\dotsc,\sigma_p \text{ distinct} }{\sum\cdots\sum} \indi_{\brc{\max_{i,j}\abs{W_{\sigma_i} - W_{\sigma_j}} \leq 8r_k}},
  \end{align*}
  where we omit to recall that the sum is over distinct elements $\sigma_i\in\Vi_n(u,\delta_k,h)$. The same counting argument and the Markov inequality still entail $ Q_\Ti\pthb{ \indi_{A(u,k,n,h)} } \leq \binom{z_k}{p}^{-1} Q_\Ti\pthb{ \Ni(u,k,n,h) }$.
  To bound $Q_\Ti\pthb{ \Ni(u,k,n,h) }$, we adopt a strategy similar to Lemma~\ref{lemma:sbm_clusters_ub}. The induction presented in the latter still holds and entails:
  \begin{align*}
    Q_\Ti\pthb{ \Ni(u,k,n,h) } \leq c_{0,p}\,r_k^{d(p-1)} \underset{\sigma_1,\dotsc,\sigma_p \text{ distinct}}{\sum\cdots\sum} \prod_{j=1}^{p-1}d(\sigma_j,\sigma_{j+1})^{-d/2}.
  \end{align*}
  the constant $c_{0,p}$ depending on $p$.
  We may now distinguish two different cases, depending on the value of $\delta_k$.
  \begin{enumerate}[(i)]
    \item To start with, let us consider the case $\delta_k\in\ivfo{\rho_n,\vartheta_n}$. Set $\sigma_1,\dotsc,\sigma_{p-1}\in\Vi_n(u,\delta_k,h)$. According to Lemma~\ref{lemma:trees_suitable_dyadics}, and the construction presented in \cite{Balanca-2015b}, there is no $\sigma_p\neq\sigma_{p-1}$ in the ball $\Bi(\sigma_{p-1},2r)$, when $r < \vartheta_n$. Furthermore, according to the estimate presented in Lemma~\ref{lemma:trees_suitable_dyadics},
    \begin{align*}
      \forall r\geq\vartheta_n;\quad \#\pthb{ \Bi(\sigma_{p-1},2r) \cap \Vi_n(u,\delta_k,h) } \leq c_0 \, \pthb{r \vartheta_n^{-1}}^{\tfrac{1}{\gamma-1}} g(r)^{-\beta}.
    \end{align*}
    recalling that $\vartheta_n=\rho_n^{(\gamma h-1)(\gamma-1)} g(\rho_n)^{-\alpha\gamma(\gamma-1)}$.
    Hence,
    \begin{align*}
      \sum_{\sigma_p\neq\sigma_{p-1}} d(\sigma_{p-1},\sigma_{p})^{-d/2}
      &\leq \sum_{\delta_m\geq \vartheta_n}  \sum_{\sigma_p\in \Bi(\sigma_{p-1},2\delta_m) \cap \Vi_n(u,\delta_k,h)} \delta_m^{-d/2} \\
      &\leq c_1\sum_{\delta_m\geq \vartheta_n} \vartheta_n^{-\tfrac{1}{\gamma-1}} \, \delta_m^{\tfrac{1}{\gamma-1}-\tfrac{d}{2}} g(\delta_m)^{-\beta} \\
      &\leq c_2 \, \vartheta_n^{-\pthb{\tfrac{1}{\gamma-1}\vee\tfrac{d}{2}}} g(\delta_k)^{-\beta-1}
      \leq c_3 \,r_k^{-d+\epsilon/2} z_k,
    \end{align*}
    using the notation $z_k$ previously introduced.\vsp

    \item Let us now look at the case $\delta_k\in\ivfo{\vartheta_n,\rho_{n-1}}$. Owing the estimates presented in Lemma~\ref{lemma:trees_suitable_dyadics}, we need to split the sum $\sum_{\sigma_p\neq\sigma_{p-1}} d(\sigma_{p-1},\sigma_{p})^{-d/2}$ into three different components corresponding to the intervals $\ivfo{\delta_k,\rho_{n-1}}$, $\ivfo{\rho_{n-1},\vartheta_{n-1}}$ and \ivfo{\vartheta_{n-1},\infty}. To begin with,
    \begin{align*}
      \sum_{\delta_m\in\ivfo{\delta_k,\rho_{n-1}}} \sum_{\sigma_p\in \Bi(\sigma_{p-1},2\delta_m) \cap \Vi_n(u,\delta_k,h)} \delta_m^{-d/2}
      &\leq c_1\sum_{\delta_m\in\ivfo{\delta_k,\rho_{n-1}}} \delta_m^{\tfrac{1}{\gamma-1}-\tfrac{d}{2}} \delta_k^{-\tfrac{1}{\gamma-1}} g(\delta_k)^{-\beta} \\
      &\leq c_2\, \delta_k^{-\pthb{\tfrac{1}{\gamma-1}\vee\tfrac{d}{2}}} \rho_{n-1}^{\pthb{\tfrac{1}{\gamma-1}-\tfrac{d}{2}}\vee 0} g(\delta_k)^{-\beta-1}.
    \end{align*}
    In addition,
    \begin{align*}
      \sum_{\delta_m\in\ivfo{\rho_{n-1},\vartheta_{n-1}}} \sum_{\sigma_p\in \Bi(\sigma_{p-1},2\delta_m) \cap \Vi_n(u,\delta_k,h)} \delta_m^{-d/2}
      &\leq c_1\sum_{\delta_m\in\ivfo{\rho_{n-1},\vartheta_{n-1}}} \delta_m^{-\tfrac{d}{2}} \pthb{\delta_k\rho_{n-1}^{-1}}^{-\tfrac{1}{\gamma-1}} g(\delta_k)^{-\beta} \\
      &\leq c_2\, \delta_k^{-\tfrac{1}{\gamma-1}} \rho_{n-1}^{{\tfrac{1}{\gamma-1}-\tfrac{d}{2}}} g(\delta_k)^{-\beta} \\
      &\leq c_2\, \delta_k^{-\pthb{\tfrac{1}{\gamma-1}\vee\tfrac{d}{2}}} \rho_{n-1}^{\pthb{\tfrac{1}{\gamma-1}-\tfrac{d}{2}}\vee 0} g(\delta_k)^{-\beta-1}.
    \end{align*}
    since $\delta_k\leq\rho_{n-1}$. Finally, the last part is such that
    \begin{align*}
      \sum_{\delta_m\geq\vartheta_{n-1}} \sum_{\sigma_p\in \Bi(\sigma_{p-1},2\delta_m) \cap \Vi_n(u,\delta_k,h)} \delta_m^{-d/2}
      &\leq c_1\sum_{\delta_m\geq\vartheta_{n-1}} \pthb{\rho_{n-1}\delta_k^{-1}}^{\tfrac{1}{\gamma-1}}  \delta_m^{-d/2} \pthb{\delta_m\,\vartheta_{n-1}^{-1}}^{\tfrac{1}{\gamma-1}} g(\delta_k)^{-\beta} \\
      &\leq c_2 \pthb{\rho_{n-1}\delta_k^{-1}}^{\tfrac{1}{\gamma-1}} \vartheta_{n-1}^{-\pthb{\tfrac{1}{\gamma-1}\vee\tfrac{d}{2}}} g(\delta_k)^{-\beta-1}.
    \end{align*}
    Combining the three previous bounds, we get as well $\sum_{\sigma_p\neq\sigma_{p-1}} d(\sigma_{p-1},\sigma_{p})^{-d/2}
      \leq c_3 \,r_k^{-d+\epsilon/2} z_k$.
  \end{enumerate}
  Consequently, we obtain by induction in two previous cases:
  \begin{align*}
    Q_\Ti\pthb{ \Ni(u,k,n,h) }
    &\leq c_{1,p}\, r_k^{\epsilon(p-1)/2} z_k^{p-1}\cdot \delta_k^{-1/(\gamma-1)-\epsilon}.
  \end{align*}
  where the last term in the bound stands for the last step in the induction and the cardinal of $\Vi_n(u,\delta_k,h)$. Therefore, $Q_\Ti\pthb{ \indi_{A(u,k,n,h)} } \leq c_{2,p} \, r_k^{\epsilon(p-1)/2} \cdot z_k^{-1} \delta_k^{-1/(\gamma-1)-\epsilon}$ and
  \begin{align*}
    Q_\Ti\pthbb{ \bigcup_{u\in\Di_k\cap\ivoo{0,h(\Ti)}} \indi_{A(u,k,n,h)} } \leq c_{2,p} \, r_k^{\epsilon(p-1)/2} \cdot z_k^{-1} \delta_k^{-1/(\gamma-1)-\epsilon-\eta} \leq c_{3,p} \, \delta_k,
  \end{align*}
  if the parameter $p$ is chosen sufficiently large. Summing over $\delta_k\in\ivfo{\rho_n,\rho_{n-1}}$ and $h\in\Hi_n$, we obtain
  \begin{align*}
    Q_\Ti\pthbb{ \bigcup_{h\in\Hi_n}  \bigcup_{\delta_k\in\ivfo{\rho_n,\rho_{n-1}}} \bigcup_{u\in\Di_k\cap\ivoo{0,h(\Ti)}} \indi_{A(u,k,n,h)} }  \leq c_{4,p} \, 2^n\rho_{n-1},
  \end{align*}
  The sum over $n\in\N$ of the last quantity clearly converges, hence concluding the proof with the help of Borel--Cantelli lemma.
\end{proof}
Let us also present a similar accumulation lemma on the complete collections $\Vi_n(\delta_k,h)$, necessary to the proof of Theorem~\ref{th:sbm_spectrum2}.
\begin{lemma}  \label{lemma:sbm_clusters_lb2}
  Suppose $x\in\R^d$ and $d\geq\tfrac{2\gamma}{\gamma-1}$. $\N_x$-a.e. for every $n$ large enough, any $h\in\Hi_n$, $\delta_k\in\ivfo{\rho_n,\rho_{n-1}}$ and any subcollection $\Vi\subset\Vi_n(\delta_k,h)$ satisfying
  \begin{align*}
    \forall \sigma_i,\sigma_j\in\Vi;\quad \abs{W_{\sigma_i} - W_{\sigma_j}} \leq 8r_k\eqdef8\delta_k^{(1-\epsilon)/2},
  \end{align*}
  the cardinal of the subset $\Vi$ satisfies:
  \begin{align}  \label{eq:sbm_clusters_lb2}
    \#\Vi \leq
    \begin{cases}
      r_k^{d-\epsilon}\delta_k^{-1} \vartheta_n^{1-\tfrac{d}{2}} & \text{if } \delta_k\in\ivfo{\rho_n,\vartheta_n}; \\
      r_k^{d-\epsilon} \pthbb{ \delta_k^{-\tfrac{d}{2}} + \rho_{n-1}^{\tfrac{1}{\gamma-1}}\delta_k^{-\tfrac{\gamma}{\gamma-1}} \vartheta_{n-1}^{1-\tfrac{d}{2}}} &\text{if } \delta_k\in\ivfo{\vartheta_n,\rho_{n-1}} . \\
    \end{cases}
  \end{align}
  where we recall the notation $\vartheta_n \eqdef \rho_n^{(\gamma-1)(\gamma h -1 )} g(\rho_n)^{-\alpha\gamma(\gamma-1)}$.
\end{lemma}
\begin{proof}
  The proof being very close to the one of Lemma~\ref{lemma:sbm_clusters_lb1}, let us only focus on the estimates which differ. To begin with, we define similarly the random variable $\Ni(k,n,h)$:
  \begin{align*}
    \Ni(k,n,h) = \underset{\sigma_1,\dotsc,\sigma_p \text{ distinct} }{\sum\cdots\sum} \indi_{\brc{\max_{i,j}\abs{W_{\sigma_i} - W_{\sigma_j}} \leq 8r_k}},
  \end{align*}
  where the sum is over elements $\sigma_i\in\Vi_n(\delta_k,h)$. Similarly, we get:
  \begin{align*}
    Q_\Ti\pthb{ \Ni(k,n,h) } \leq c_{0,p}\,r_k^{d(p-1)} \underset{\sigma_1,\dotsc,\sigma_p \text{ distinct}}{\sum\cdots\sum} \prod_{j=1}^{p-1}d(\sigma_j,\sigma_{j+1})^{-d/2}.
  \end{align*}
  \begin{enumerate}[(i)]
    \item In the case $\delta_k\in\ivfo{\rho_n,\vartheta_n}$, set $\sigma_1,\dotsc,\sigma_{p-1}\in\Vi_n(\delta_k,h)$. According to Lemma~\ref{lemma:trees_suitable_dyadics}, and the construction presented in \cite{Balanca-2015b}, there is no $\sigma_p\neq\sigma_{p-1}$ in the ball $\Bi(\sigma_{p-1},2r)$, when $r < \vartheta_n$. Furthermore, as a direct consequence of the estimates presented in Lemma~\ref{lemma:trees_suitable_dyadics}:
    \begin{align*}
      \forall r\geq\vartheta_n;\quad \#\pthb{ \Bi(\sigma_{p-1},2r) \cap \Vi_n(\delta_k,h) } \leq c_0 \, r\delta_k^{-1} \cdot \pthb{r\,\vartheta_n^{-1}}^{\tfrac{1}{\gamma-1}} g(r)^{-\beta}.
    \end{align*}
    Hence,
    \begin{align*}
      \sum_{\sigma_p\neq\sigma_{p-1}} d(\sigma_{p-1},\sigma_{p})^{-d/2}
      &\leq \sum_{\delta_m\geq \vartheta_n}  \sum_{\sigma_p\in \Bi(\sigma_{p-1},2\delta_m) \cap \Vi_n(\delta_k,h)} \delta_m^{-d/2} \\
      &\leq c_1\sum_{\delta_m\geq \vartheta_n} \delta_k^{-1}\vartheta_n^{-\tfrac{1}{\gamma-1}} \, \delta_m^{\tfrac{\gamma}{\gamma-1}-\tfrac{d}{2}} g(\delta_m)^{-\beta} \\
      &\leq c_2 \, \delta_k^{-1}\vartheta_n^{1-\tfrac{d}{2}} g(\delta_k)^{-\beta-1}
      \leq c_3 \,r_k^{-d+\epsilon/2} z_k,
    \end{align*}
    where $z_k$ refers to right hand term in \eqref{eq:sbm_clusters_lb2}.\vsp

    \item Let us now look at the case $\delta_k\in\ivfo{\vartheta_n,\rho_{n-1}}$. Owing the estimates presented in Lemma~\ref{lemma:trees_suitable_dyadics}, we also need to split the sum $\sum_{\sigma_p\neq\sigma_{p-1}} d(\sigma_{p-1},\sigma_{p})^{-d/2}$ into three different components corresponding to the intervals $\ivfo{\delta_k,\rho_{n-1}}$, $\ivfo{\rho_{n-1},\vartheta_{n-1}}$ and $\ivfo{\vartheta_{n-1},\infty}$. To begin with,
%\footnote{LM is this the place whwre we use $d\geq \frac{2\gamma}{\gamma-1}$?\\
   % PB: Yes, otherwise the bound is of order $\delta_k^{-\gamma/(\gamma-1)}$}
    \begin{align*}
      \sum_{\delta_m\in\ivfo{\delta_k,\rho_{n-1}}} \sum_{\sigma_p\in \Bi(\sigma_{p-1},2\delta_m) \cap \Vi_n(\delta_k,h)} \delta_m^{-d/2}
      &\leq c_1\sum_{\delta_m\in\ivfo{\delta_k,\rho_{n-1}}} \delta_m^{\tfrac{\gamma}{\gamma-1}-\tfrac{d}{2}} \delta_k^{-\tfrac{\gamma}{\gamma-1}} g(\delta_k)^{-\beta-1} \\
      &\leq c_2\, \delta_k^{-\tfrac{d}{2}} g(\delta_k)^{-\beta-1}.
    \end{align*}
    According to the construction recalled at the beginning of the section, for any $\delta_m\in\ivfo{\rho_{n-1},\vartheta_{n-1}}$, $\Bi(\sigma_{p-1},2\delta_m) \cap \Vi_n(\delta_k,h) = \Bi(\sigma_{p-1},2\rho_{n-1}) \cap \Vi_n(\delta_k,h)$. Hence, still using Lemma~\ref{lemma:trees_suitable_dyadics}:
    \begin{align*}
      \sum_{\delta_m\in\ivfo{\rho_{n-1},\vartheta_{n-1}}} \sum_{\sigma_p\in \Bi(\sigma_{p-1},2\delta_m) \cap \Vi_n(\delta_k,h)} \delta_m^{-d/2}
      &\leq c_1\sum_{\delta_m\in\ivfo{\rho_{n-1},\vartheta_{n-1}}} \delta_m^{-\tfrac{d}{2}} \pthb{\delta_k\rho_{n-1}^{-1}}^{-\tfrac{\gamma}{\gamma-1}} g(\delta_k)^{-\beta} \\
      &\leq c_2\, \delta_k^{-\tfrac{\gamma}{\gamma-1}} \rho_{n-1}^{{\tfrac{\gamma}{\gamma-1}-\tfrac{d}{2}}} g(\delta_k)^{-\beta-1}
      \leq c_2\, \delta_k^{-\tfrac{d}{2}} g(\delta_k)^{-\beta-1}.
    \end{align*}
    since $\delta_k\leq\rho_{n-1}$. Finally, the last part is such that
    \begin{align*}
      \sum_{\delta_m\geq\vartheta_{n-1}} \sum_{\sigma_p\in \Bi(\sigma_{p-1},2\delta_m) \cap \Vi_n(\delta_k,h)} \delta_m^{-d/2}
      &\leq c_1\sum_{\delta_m\geq\vartheta_{n-1}}  \pthb{\rho_{n-1}\vartheta_{n-1}^{-1}}^{\tfrac{1}{\gamma-1}}  \pthb{\delta_m\delta_k^{-1}}^{\tfrac{\gamma}{\gamma-1}} g(\delta_k)^{-\beta-1} \\
      &\leq c_2 \rho_{n-1}^{\tfrac{1}{\gamma-1}}\delta_k^{-\tfrac{\gamma}{\gamma-1}} \vartheta_{n-1}^{1-\tfrac{d}{2}} g(\delta_k)^{-\beta-1}.
    \end{align*}
    Combining the three previous bounds, we get $\sum_{\sigma_p\neq\sigma_{p-1}} d(\sigma_{p-1},\sigma_{p})^{-d/2} \leq c_3 \,r_k^{-d+\epsilon/2} z_k$.
  \end{enumerate}
  We omit the rest of the proof which remains exactly the same.
\end{proof}

The estimate obtained in Lemma~\ref{lemma:sbm_clusters_lb1} is sufficient to deduce a mass distribution principle on the collection of measure $(\nu_{t,2h})_{h\in\Hi}$, recalling that $\Hi$ stands for a closed interval $\Hi\subset\ivof{\tfrac{1}{\gamma},\tfrac{1}{\gamma-1}}$.
\begin{lemma}  \label{lemma:sbm_mass_principle_lb}
  Suppose $x\in\R^d$. $\N_x$-a.e. for every $h\in\Hi$ and any time $t\in\ivoo{\epsilon,h(\Ti)-\epsilon}$,
  \begin{align}  \label{eq:sbm_mass_pple_spectrum1}
    \forall z\in\R^d,\ \forall r\in\ivoo{0,1};\quad \nu_{t,2h}\pthb{B(z,r)} \leq c_0\,r^{d\wedge(2\gamma h-2)-\eta\eps},
  \end{align}
  where $c_0>0$ only depends on $\epsilon$ and $\Hi$, and $\eta>0$ on $\Hi$. In addition, there also exists a decreasing sequence $\varrho_n\rightarrow 0$ such that
  \begin{align}  \label{eq:sbm_mass_pple_spectrum2}
    \forall z\in\R^d,\ \forall n\in\N;\quad \nu_{t,2h}\pthb{B(z,\varrho_n)} \leq c_0\,\varrho_n^{d\wedge\tfrac{2}{\gamma-1}-\eta\eps}.
  \end{align}
\end{lemma}
\begin{proof}
  Let us set $h\in\Hi$, $t\in\ivoo{\epsilon,h(\Ti)-\epsilon}$, $z\in\R^d$ and $r\in\ivoo{0,1}$. Without any loss of generality, we may assume that $z=W_{\sigma_0}$, for some $\sigma_0\in G(t,h)$ (we refer to the introduction of the section for the definition of the latter) and $r=r_k\eqdef\delta_k^{(1-\epsilon)/2}$, $k\in\N$. We aim in this proof to bound the local mass:
  \begin{align*}
    \nu_{t,2h}(B(z,r_k)) = \mu_{t,h}\pthb{ \brcb{\sigma\in G(t,h) : \abs{W_{\sigma_0}-W_{\sigma}} \leq r_k } }.
  \end{align*}
  We know there exists $n\in\N$ and $h_n\in\Hi_n$ such that $\delta_k\in\ivfo{\rho_n,\rho_{n-1}}$ and $\sigma_0\in G(t,h_n,n)$. In addition, according to the properties of $\Vi_n(\delta_k,h_n)$ presented in Lemma~\ref{lemma:trees_suitable_dyadics}, there is $u\leq t\in\Di_k$ such that
  \begin{align*}
    \forall \sigma\in G(t,h),\ \exists \sigma'\in\Vi_n(u,\delta_k,h_n);\quad d(\sigma,\sigma') \leq 6\delta_k.
  \end{align*}
  Consequently, since $W$ is $\tfrac{1-\epsilon}{2}$-Hölder continuous,
  \begin{align*}
    \brcb{\sigma\in G(t,h) : \abs{W_{\sigma_0}-W_{\sigma}} \leq r_k }
    \subset \bigcup_{\substack{\sigma\in\Vi_n(u,\delta_k,h_n):\abs{W_\sigma-W_{\sigma_0}}\leq 4r_k\\ \sigma\in\Ti_\varsigma\in\Tbb(t-6\delta_k,6\delta_k) } } G(t,h)\cap\Ti_\varsigma
  \end{align*}
  For any $\sigma,\sigma'\in\Vi_n(u,\delta_k,h_n)$ such that $\abs{W_\sigma-W_{\sigma_0}}\leq 4r_k$ and $\abs{W_{\sigma'}-W_{\sigma_0}}\leq 4r_k$, one gets $\abs{W_\sigma-W_{\sigma'}}\leq 8r_k$. Consequently, we may use the bound presented in Lemma~\ref{lemma:sbm_clusters_lb1} to obtain an estimate of $\nu_{t,2h}(B(z,r_k))$. Similarly to the latter, we need to distinguish two different cases.
  \begin{enumerate}[(i)]
    \item Consider first the case $\delta_k\in\ivfo{\rho_n,\vartheta_n}$. According to \cite[Lemma 4.21]{Balanca-2015b}, there exist two positive constants $c_0$ and $\eta_0$ such that
    \begin{align}  \label{eq:mass_pple_mu0}
      \forall \sigma\in G(t,h);\quad \mu_{t,h}(\Bi(\sigma,12\delta_k)) \leq c_0\,g(\delta_k)^{-\eta_0}\,\vartheta_n^{\tfrac{1}{\gamma-1}}.
    \end{align}
    Hence, using the bound presented in Equation~\eqref{eq:sbm_clusters_lb1}, we get:
    \begin{align*}
      \nu_{t,2h}(B(z,r_k))
      &\leq c_0\,\vartheta_n^{\tfrac{1}{\gamma-1}} g(\delta_k)^{-\eta_0} \cdot r_k^{d-\epsilon} \vartheta_n^{-\pthb{\tfrac{1}{\gamma-1}\vee\tfrac{d}{2}}}
      \leq r_k^{d-2\epsilon} \vartheta_n^{\pthb{\tfrac{1}{\gamma-1}-\tfrac{d}{2}}\wedge 0}.
    \end{align*}
    If $\tfrac{d}{2}\leq \gamma h-1\leq \tfrac{1}{\gamma-1}$, the previous bound readily implies $\nu_{t,2h}(B(z,r_k)) \leq r_k^{d-2\epsilon}$. On the other hand, if $\tfrac{d}{2} > \gamma h-1$,
    \begin{align*}
      \nu_{t,2h}(B(z,r_k))
      &\leq r_k^{2\gamma h -2-\eta\epsilon} \cdot \frac{\delta_k^{d/2-(\gamma h-1)}}{ \vartheta_n^{\pth{d/2-1/(\gamma-1)}\vee 0} }.
    \end{align*}
    for some $\eta>0$. The second part of the right term is then bounded by a constant, since $\delta_k\leq \vartheta_n$ and $d/2-(\gamma h-1) \geq \pth{d/2-1/(\gamma-1)}\vee 0$, therefore providing the desired estimate.\vsp

    \item Let us now investigate the case $\delta_k\in\ivfo{\vartheta_n,\rho_{n-1}}$. Similarly, \cite[Lemma 4.21]{Balanca-2015b} entails:
    \begin{align}  \label{eq:mass_pple_mu1}
      \forall \sigma\in G(t,h);\quad \mu_{t,h}(\Bi(\sigma,12\delta_k)) \leq c_0\,g(\delta_k)^{-\eta_0}\, \pthb{ \delta_k\vartheta_{n-1}\rho_{n-1}^{-1} }^{\tfrac{1}{\gamma-1}}.
    \end{align}
    Therefore, using the second bound obtained in Lemma~\ref{lemma:sbm_clusters_lb1}, we get:
    \begin{align*}
      \nu_{t,2h}(B(z,r_k)) \leq c_0\, g(\delta_k)^{-\eta_0}\, \pthb{ \delta_k\vartheta_{n-1}\rho_{n-1}^{-1} }^{\tfrac{1}{\gamma-1}} \cdot r_k^{d-\epsilon} \pthbb{
      &\delta_k^{-\pthb{\tfrac{1}{\gamma-1}\vee\tfrac{d}{2}}} \rho_{n-1}^{\pthb{\tfrac{1}{\gamma-1}-\tfrac{d}{2}}\vee 0} + \\
      & \pthb{\rho_{n-1}\delta_k^{-1}}^{\tfrac{1}{\gamma-1}} \vartheta_{n-1}^{-\pthb{\tfrac{1}{\gamma-1}\vee\tfrac{d}{2}}} }.
    \end{align*}
    Let us first suppose that $\tfrac{d}{2}\leq \gamma h-1\leq \tfrac{1}{\gamma-1}$. Simplifying the previous expression, we obtain
    \begin{align*}
      \nu_{t,2h}(B(z,r_k))
      \leq r_k^{d-2\epsilon} \cdot \pthB{ \vartheta_{n-1}^{\tfrac{1}{\gamma-1}}\rho_{n-1}^{-\tfrac{d}{2}}  + 1 }
      \leq r_k^{d-\eta\epsilon} \pthb{\rho_{n-1}^{\gamma h -1 -d/2}+1} \leq  r_k^{d-\eta\epsilon},
    \end{align*}
    recalling that $\vartheta_{n-1} = \rho_{n-1}^{(\gamma h_{n-1} - 1)(\gamma-1)} g(\rho_{n-1})^{-\alpha\gamma(\gamma-1)}$.

    Let us now assume that $\tfrac{d}{2}>\gamma h-1$. Then,
    \begin{align*}
      \nu_{t,2h}(B(z,r_k))
      &\leq r_k^{-\eta_0\epsilon} \pthbb{
      \vartheta_{n-1}^{\tfrac{1}{\gamma-1}} \delta_k^{\tfrac{1}{\gamma-1}\wedge\tfrac{d}{2}} \rho_{n-1}^{-\pthb{\tfrac{1}{\gamma-1}\wedge\tfrac{d}{2}}} + \delta_k^{\tfrac{d}{2}} \vartheta_{n-1}^{\pthb{\tfrac{1}{\gamma-1}-\tfrac{d}{2}}\wedge 0} } \\
      &\leq r_k^{2\gamma h - 2 - \eta_1\epsilon} \pthbb{
      \pthb{\delta_k \rho_{n-1}^{-1} } ^{\tfrac{1}{\gamma-1}\wedge\tfrac{d}{2}-(\gamma h-1)} + \delta_k^{\tfrac{d}{2}-(\gamma h - 1)} \vartheta_{n-1}^{\pthb{\tfrac{1}{\gamma-1}-\tfrac{d}{2}}\wedge 0} } \\
      &\leq 2 r_k^{2\gamma h - 2 - \eta_1\epsilon},
    \end{align*}
    since as previously $\delta_k\leq \rho_{n-1}\leq\vartheta_{n-1}$ and $d/2-(\gamma h-1) \geq \pth{d/2-1/(\gamma-1)}\vee 0$.
  \end{enumerate}
  This last inequality concludes the first part of the proof, as we have obtained the expected upper bound in both of the two cases.
  The second inequality is simply a consequence of the previous bounds in the particular case $\delta_k=\vartheta_n$.
\end{proof}

In high dimension $d\geq\tfrac{2\gamma}{\gamma-1}$, we are able to obtain a uniform mass distribution principle. More specifically, suppose $F\subset\ivoo{\epsilon,\infty}$ is a Borel set satisfying the \emph{strong Frostman's lemma}: for every $\eps>0$,
\begin{align*}
  \exists r_0>0,\quad  \forall x\in F,\ \forall r\in\ivoo{r,r_0};\quad \mu_F\pthb{B(x,r)} \leq r^{\dimH F - \eps}.
\end{align*}
where $\mu_F$ is probability measure on $F$. Following the definition of $\mu_{F,h}$ in \cite{Balanca-2015b}, we then introduce the natural pushforward measure $\nu_{F,2h}(\dt x)$:
\begin{align}  \label{eq:pushforward_measure2}
  \forall V\in\Bi(\R^d);\quad \nu_{F,2h}(V) \eqdef \mu_{F,h}\pthb{ W^{-1}(V) } = \int_{\ivoo{0,\infty}} \mu_{t,h}\pthb{ W^{-1}(V) } \,\mu_F(\dt t),
\end{align}
% \begin{align}  \label{eq:pushforward_measure2}
%   \forall V\in\Bi(\R^d);\quad \nu_{F,2h}(V) \eqdef \int_{\ivoo{0,\infty}} \nu_{t,2h}(V) \,\mu_F(\dt t) = \int_{\ivoo{0,\infty}} \mu_{t,h}\pthb{ W^{-1}(V) } \,\mu_F(\dt t),
% \end{align}
and the set $G(F,h)=\cup_{t\in F} G(t,h)$.

Similarly to the previous lemma, Lemma~\ref{lemma:sbm_clusters_lb2} leads to  a proper mass distribution principle on the collection of measures $(\nu_{F,2h})_{h\in\Hi}$.
\begin{lemma}  \label{lemma:sbm_mass_principle_lb2}
  Suppose $d\geq\tfrac{2\gamma}{\gamma-1}$ and $x\in\R^d$. $\N_x$-a.e. for all $h\in\Hi$ and every Borel set $F$ satisfying \eqref{eq:strong_frostman}, we get:
  \begin{align*}
    \forall z\in\R^d,\ \forall r\in\ivoo{0,1};\quad \nu_{F,2h}\pthb{B(z,r)} \leq c_0\,r^{2\gamma h-2+2\dimH F-\eta\eps},
  \end{align*}
  where $c_0>0$ only depends on $\epsilon$ and $\Hi$, and $\eta>0$ depends on $\Hi$.
\end{lemma}
\begin{proof}
  The structure of the proof follows the one of Lemma~\ref{lemma:sbm_mass_principle_lb}. Let us set $h\in\Hi$, $F\in\ivoo{\epsilon,\infty}$ satisfying \eqref{eq:strong_frostman}, $s=\dimH F$, $z\in\R^d$ and $r\in\ivoo{0,1}$. Without any loss of generality, we may assume that $z=W_{\sigma_0}$, for some $\sigma_0\in G(F,h)$ and $r=r_k\eqdef\delta_k^{(1-\epsilon)/2}$, $k\in\N$. Similarly, we observe that:
  \begin{align*}
    \forall n\in\N;\quad \brcb{\sigma\in G(F,h) : \abs{W_{\sigma_0}-W_{\sigma}} \leq r_k }
    \subset \bigcup_{\substack{\sigma\in\Vi_n(\delta_k,h_n)\\\abs{W_\sigma-W_{\sigma_0}}\leq 4r_k} } G(F,h)\cap \Bi(\sigma,12\delta_k).
  \end{align*}
  \begin{enumerate}[(i)]
    \item Consider first the case $\delta_k\in\ivfo{\rho_n,\vartheta_n}$. According to \cite[Lemma 4.21]{Balanca-2015b}, the uniform bound \eqref{eq:mass_pple_mu0} and the property \eqref{eq:strong_frostman}, for any $\sigma\in\Vi_n(\delta_k,h_n)$:
    \begin{align*}
      \mu_{F,h}\pthb{ G(F,h)\cap\Bi(\sigma,12\delta_k) }
      \leq \sup_{s\in I(\sigma,k)} \mu_{s,h}\pthb{ \Bi(\sigma,12\delta_k) } \cdot \mu_F\pthb{I(\sigma,k)}
      \leq c_1\,\delta_k^{s-\epsilon}\cdot g(\delta_k)^{-\eta_0}\,\vartheta_n^{\tfrac{1}{\gamma-1}}.
    \end{align*}
    where $I(\sigma,k)\eqdef \ivoob{h(\sigma)-12\delta_k,h(\sigma)+12\delta_k}$. Hence, using the bound presented in Equation~\eqref{eq:sbm_clusters_lb2}, we get:
    \begin{align*}
      \nu_{F,2h}(B(z,r_k))
      &\leq c_1\,\delta_k^{s-\epsilon} g(\delta_k)^{-\eta_0}\,\vartheta_n^{\tfrac{1}{\gamma-1}} \cdot r_k^{d-\epsilon}\delta_k^{-1} \vartheta_n^{1-\tfrac{d}{2}}
      \leq c_1\, r_k^{2\gamma h-2+2s-\eta\epsilon} \cdot \frac{\delta_k^{d/2-\gamma h}}{ \vartheta_n^{d/2-\gamma/(\gamma-1)} },
    \end{align*}
    which entails the desired bound as $\delta_k\leq \vartheta_n$ and $d/2-\gamma h \geq d/2-\gamma/(\gamma-1)$.

    \item Let us now investigate the case $\delta_k\in\ivfo{\vartheta_n,\rho_{n-1}}$. Similarly, still according to the estimates \eqref{eq:mass_pple_mu1} and \eqref{eq:strong_frostman}, for any $\sigma\in\Vi_n(\delta_k,h_n)$:
    \begin{align*}
      \mu_{F,h}\pthb{ G(F,h)\cap\Bi(\sigma,12\delta_k) } \leq c_1\,\delta_k^{s-\epsilon}\cdot g(\delta_k)^{-\eta_0}\,\pthb{ \delta_k\vartheta_{n-1}\rho_{n-1}^{-1} }^{\tfrac{1}{\gamma-1}}.
    \end{align*}
    Therefore, using the second bound in Lemma~\ref{lemma:sbm_clusters_lb2}, we get:
    \begin{align*}
      \nu_{t,2h}(B(z,r_k))
      &\leq c_1\,\delta_k^{s-\epsilon}\cdot g(\delta_k)^{-\eta_0}\, \pthb{ \delta_k\vartheta_{n-1}\rho_{n-1}^{-1} }^{\tfrac{1}{\gamma-1}} \cdot r_k^{d-\epsilon} \pthbb{ \delta_k^{-\tfrac{d}{2}} + \rho_{n-1}^{\tfrac{1}{\gamma-1}}\delta_k^{-\tfrac{\gamma}{\gamma-1}} \vartheta_{n-1}^{1-\tfrac{d}{2}} } \\
      &\leq c_1\, r_k^{2\gamma h-2+2s-\eta\epsilon} \pthbb{ \pthb{\delta_k\rho_{n-1}^{-1}}^{\tfrac{1}{\gamma-1}-(\gamma h -1)} + \delta_k^{\tfrac{d}{2}-\gamma h} \vartheta_{n-1}^{-\tfrac{d}{2}+\tfrac{\gamma}{\gamma-1}} } \\
      &\leq c_1\, r_k^{2\gamma h-2+2s-\eta\epsilon},
    \end{align*}
    since as previously $\delta_k\leq \rho_{n-1}\leq\vartheta_{n-1}$ and $d/2-\gamma h \geq d/2-\gamma/(\gamma-1)$.
  \end{enumerate}
  This last inequality concludes the proof.
\end{proof}

We may now present the second part of the proof of Proposition~\ref{prop:spectrum_excursions}.
\begin{lemma}  \label{lemma:sbm_spectrum_lb}
  Suppose $x\in\R^d$. Then, the following statements  hold $\N_x$-a.e.
  \begin{enumerate}[(a)]
    \item Assuming $d\geq 2$, the spectrum of singularities of the excursion measure $\Xii_t(\dt x)$ is equal to:
    \begin{align*}
      \forall h\in\ivffb{\tfrac{1}{\gamma},\tfrac{1}{\gamma-1}}\cap\ivfob{0,\tfrac{d}{2}};\quad \dimH \,E(2h,\Xii_t)\cap V \geq 2\gamma h - 2,
    \end{align*}
    for any $t>0$ and open set $V\subset\R^d$ such that $\Xii_t(V)>0$.
    \item Supposing $d\geq \tfrac{2\gamma}{\gamma-1}$, for any Borel set $F\subset\ivoo{0,h(\Ti)}$ satisfying the \emph{strong Frostman's lemma} \eqref{eq:strong_frostman} and such that $\dimH F=\dimP F$, we have:
    \begin{align*}
      \forall h\in\ivffb{\tfrac{1}{\gamma}, \tfrac{1}{\gamma-1}};\quad \dimH \bigcup_{s\in F} E(2h,\Xii_s) \geq 2\gamma h - 2 + 2\dimH F.
    \end{align*}
     \item For any dimension $d\geq 1$,
    \begin{align*}
      \forall h\in\ivfob{0,\tfrac{1}{\gamma}}\cap\ivfob{0,\tfrac{d}{2}};\quad \dimH \brcb{t>0 : E(2h,\Xii_t) \neq \vset } \geq \gamma h,
    \end{align*}
    Moreover, for every $t>0$ and any $h\in\ivfob{0,\tfrac{1}{\gamma}}\cap\ivfob{0,\tfrac{d}{2}}$, $E(2h,\Xii_t)$ is either empty or has zero Hausdorff dimension.
  \end{enumerate}
\end{lemma}
\begin{proof}
  Suppose $\Hi\subset\ivofb{\tfrac{1}{\gamma},\tfrac{1}{\gamma-1}}$ and $h\in\Hi\cap\ivfob{0,\tfrac{d}{2}}$ (whenever $d\geq 2$). According to the result of Lemma~\ref{lemma:sbm_mass_principle_lb} and the upper bound in Lemma \ref{lemma:sbm_spectrum_ub}, one has $\N_x$-a.e.
  \begin{align*}
    \forall t\in\ivoo{0,h(\Ti)},\ \forall h'\in\ivfo{0,h};\quad \nu_{t,2h}\pthb{ E(2h',\Xii_t) } = 0.
  \end{align*}
  In addition, defining the collection $G_W(t,2h) = W(G(t,h))$, we observe according to Lemma~\ref{lemma:sbm_bound_holder} that for every $z\in G_W(t,2h)$, $\alpha(z,\Xii_t) \leq 2h$. Therefore, $G_W(t,2h)\subset F(2h,\Xii_t)$ and
  \begin{align*}
    \nu_{t,2h}\pthb{ E(2h,\Xii_t) } \geq  \nu_{t,2h}\pthbb{ G_W(t,h)\setminus\bigcup_{h'<h} E(2h',\Xii_t) } > 0
  \end{align*}
  Lemma~\ref{lemma:sbm_mass_principle_lb} combined with the celebrated mass distribution principle (see for instance \cite[Th. 4.2]{Falconer-2003}) then gives the desired lower bound: $\dimH E(2h,\Xii_t) \geq 2\gamma h-2$. In the specific case $h=\tfrac{1}{\gamma}$, we make use of the specific construction of a measure $\mu_{t,h}$ described in \cite[Lemma 4.23]{Balanca-2015b} and the observation in the former that for any $h'<\tfrac{1}{\gamma}$, $\mu_{t,\tfrac{1}{\gamma}}(E(h',\ell^t)) = 0$. The properties of the set $V(t,h)$ presented in Lemma \ref{lemma:sbm_spectrum_ub} are then sufficient to conclude that $E(\tfrac{2}{\gamma},\Xii_t)$ is non-empty for every $t\in\ivoo{0,h(\Ti)}$. Finally, the self-similarity of stable trees immediately provides the local version for any open set $V$, concluding the proof of the first part.\vsp

  The proof of the second statement is rather similar. Suppose $F\subset\ivoo{0,h(\Ti)}$ satisfies the \emph{strong Frostman's lemma} \eqref{eq:strong_frostman} and is such that $\dimH F=\dimP F$. According to the upper bound Lemma \ref{lemma:sbm_spectrum_ub}, one obtains as well:
  \begin{align*}
    \forall h'\in\ivfo{0,h};\quad \nu_{F,2h}\pthbb{ \bigcup_{s\in F} E(2h',\Xii_s) } = 0.
  \end{align*}
  Note that the strong Frostman assumption is the key element to obtain the previous equality, for any $h'\in\ivfo{0,h}$.
  Lemma~\ref{lemma:sbm_mass_principle_lb2} and the mass distribution principle then entail the desired result. The case $h=\tfrac{1}{\gamma}$ is also treated similarly, still relying on the construction presented in \cite[Lemma 4.23]{Balanca-2015b}.

  Finally, the third part is direct consequence of Lemma~\ref{lemma:sbm_bound_holder} proving that $\alpha_{\Xii_t}(W_\sigma) \leq 2 \alpha_\ell(\sigma,\Ti)$ and the construction in \cite[Lemma 4.29]{Balanca-2015b} of proper probability measures $\mu_h(\dt t)$ carried by the sets $\brcb{t>0 : F(h,\ell^t)\neq\vset}$. Using the same measures, the upper bound previously obtained and the mass distribution principle give the desired estimate.
\end{proof}

%%%%%%%%%%%%%%%%%%%%%%%%%%%%%%%%%%%%%%%%%%%%%%%%%%%%%%%%%%%%%%%%%%%%%%%%
%%
%%%%%%%%%%%%%%%%%%%%%%%%%%%%%%%%%%%%%%%%%%%%%%%%%%%%%%%%%%%%%%%%%%%%%%%%
\section{Proof of Theorems \ref{th:sbm_spectrum1}, \ref{th:sbm_spectrum2} and  \ref{th:support_dimension}: spectrum of stable super-Brownian motion} \label{sec:spectrum_sbm}

Finally, let us present of the proofs of the main Theorems \ref{th:sbm_spectrum1}, \ref{th:sbm_spectrum2} and \ref{th:support_dimension} on stable super-Brownian motion.
\begin{proof}[Proof of Theorems \ref{th:sbm_spectrum1} and \ref{th:sbm_spectrum2}]
  Following the construction recalled in Proposition~\ref{prop:sbm_construction}, we can assumed that the stable SBM has the following representation: $X_t(\dt x) = \sum_{i\in\Ii} \Xii_t(\Ti^i,\Wi^i)(\dt x)$.
  For any $\eps>0$, there is only a finite number of trees $\Ti^i$ such that $h(\Ti^i)>\eps$. As a consequence, $\Pr_\mu$-a.s.
  \begin{align*}
    \forall t>0,\ \forall x\in\R^d;\quad \alpha_{X_t}(x) = \inf_{i\in\Ii} \alpha_{\Xii^i_t}(x),
  \end{align*}
  noting that the infimum is in fact a minimum over a finite collection.
  Hence, $\Pr_\mu$-a.s. for every $t>0$ and any $h\geq 0$, $F(2h,X_t) = \cup_{i\in\Ii} F(2h,\Xii^i_t)$, and Proposition~\ref{prop:spectrum_excursions} on the spectrum of excursion measures entails our results on stable super-Brownian motion.
\end{proof}

The proof of Theorem~\ref{th:support_dimension} is a slight variation of the arguments presented in Section~\ref{sec:stable_sbm_high_dim}.
\begin{proof}[Proof of Theorem \ref{th:support_dimension}]
  The first result on the Hausdorff dimension is a direct consequence of the uniform upper bound (Lemma \ref{lemma:sbm_spectrum_ub}) and the mass distribution principle (Lemma \ref{lemma:sbm_mass_principle_lb}). An equivalent statement on the packing dimension is straightforward using the image of covers described in \cite[Lemma 4.4]{Balanca-2015b} and the $\tfrac{1}{2}$-Hölder continuity of the process $(W_\sigma)_{\sigma\in\Ti}$.

  To obtain the second uniform statement, we also rely on \cite[Lemma 4.4]{Balanca-2015b} to construct a proper cover and obtain the optimal upper bound. Lemma~\ref{lemma:sbm_mass_principle_lb2} can be adapted to this particular setting, using for that purpose the classic Frostman's lemma: for any Borel set $F$ and every $\alpha<\dimH F$, there exists a compact set $F_\star\subset F$ such that $\Hi^\alpha(F_\star)\in\ivoo{0,\infty}$ and
  \begin{align*}
    \forall r>0,\ \forall t\in\R;\quad \Hi^\alpha\pthb{F_\star\cap B(t,r)} \leq c_F\,r^\alpha,
  \end{align*}
  for some positive constant $c_F$. The desired lower bound is then obtained by adapting Lemma~\ref{lemma:sbm_mass_principle_lb2} and replacing the measure $\mu_F(\dt t)$ by $\Hi^\alpha\pthb{F_\star\cap\dt t}$.
\end{proof}

%!TEX root = article.tex
% mainfile: article.tex
%%%%%%%%%%%%%%%%%%%%%%%%%%%%%%%%%%%%%%%%%%%%%%%%%%%%%%%%%%%%%%%%%%%%%%%%
%% Appendix: Technical results on stable trees
%%%%%%%%%%%%%%%%%%%%%%%%%%%%%%%%%%%%%%%%%%%%%%%%%%%%%%%%%%%%%%%%%%%%%%%%
\appendix

\section{Proof of Lemma~\ref{lemma:trees_suitable_dyadics} on stable trees}  \label{sec:appendix}

In this section, we present the proof of Lemma~\ref{lemma:trees_suitable_dyadics}. The latter relies heavily on the specific construction of the measures $\mu_{a,h}(\dt\sigma)$ presented in \cite{Balanca-2015b}, and even though quite technical, is only a consequence of the properties and estimates presented in this previous work.
\begin{lemma*}[\ref{lemma:trees_suitable_dyadics}]
  Suppose $\Hi\subset\ivofb{\tfrac{1}{\gamma},\tfrac{1}{\gamma-1}}$ is a closed interval and $\epsilon>0$. $\Nb(\dt \Ti)$-a.e., for every $n$ large enough, any $\delta_k\in\ivfo{\rho_{n},\rho_{n-1}}$, $u\in\Di_k$ and $h\in\Hi_n$, there exists a collection of nodes $\Vi_n(u,\delta_k,h)\subset\Ti(u)$ satisfying the following properties:
  \begin{enumerate}[(i)]
    \item for every $j\rho_n\in\ivoo{\epsilon,h(\Ti)-\epsilon}$, there exists $u\in\Di_k\cap\ivof{0,j\rho_n}$ such that
    \begin{align*}
      \forall \Ti_\sigma\in\Vbb_n(j,h), \ \exists \sigma'\in\Vi_n(u,\delta_k,h);\quad d(\sigma,\sigma') \leq 4\delta_k;
    \end{align*}
    \item for every $u\in\Di_k$, $h\in\Hi_n$ and any $\sigma\in\Vi_n(u,\delta_k,h)$,
    \begin{align*}
      \#\pthb{ \Bi(\sigma,2r) \cap \Vi_n(u,\delta_k,h) } \leq c_0
      \begin{cases}
        1 + \pthb{r\vartheta_n^{-1}}^{\tfrac{1}{\gamma-1}} g(r)^{-\beta}  &\text{if } \delta_k\in\ivfo{\rho_n,\vartheta_n}
        \\ &\text{and }r\geq\delta_k; \\
        \pthb{r\delta_k^{-1}}^{\tfrac{1}{\gamma-1}} g(\delta_k)^{-\beta} &\text{if } \delta_k\in\ivfo{\vartheta_{n},\rho_{n-1}} \\
        & \text{and } r\in\ivfo{\delta_k,\rho_{n-1}}; \\
        \pthb{\rho_{n-1}\delta_k^{-1}}^{\tfrac{1}{\gamma-1}} g(\delta_k)^{-\beta}\pthB{1+\pthb{r\vartheta_{n-1}^{-1}}^{\tfrac{1}{\gamma-1}} } & \text{if } \delta_k\in\ivfo{\vartheta_{n},\rho_{n-1}} \\
         &\text{and } r\geq\rho_{n-1}, \\
      \end{cases}
    \end{align*}
    where the constants $\beta$ and $c_0$ are independent of the parameters $n$, $k$ and $h$, and $\vartheta_n \eqdef \rho_n^{(\gamma-1)(\gamma h -1 )} g(\rho_n)^{-\alpha\gamma(\gamma-1)}$.

    % In addition, $\#\Ei_n(\delta_k) \leq h(\Ti)\,\delta_k^{-\eta}$, where $\eta>0$ only depends on $\Hi$.
  \end{enumerate}
  Finally, we will simply denote by $\Vi(\delta_k,h)$ the full collection $\cup_{u\in\Di_k} \Vi(u,\delta_k,h)$.
\end{lemma*}
\begin{proof}
  Let us set $\delta_k\in\ivof{\rho_n,\rho_{n-1}}$, $u=m\delta_k\in\Di_k$ and $h\in\Hi_n$. We start by defining the following collection:
  \begin{align*}
    \Vi^\star_n(m\delta_k,\delta_k,h) \eqdef \bigcup_{j\rho_n\in\ivfo{(m+1)\delta_k,(m+2)\delta_k}} \bigcup_{\Ti_{\sigma}\in\Vbb_n(j,h)} \iivff{\rho(\Ti),\sigma}\cap\Ti(m\delta_k) \subset \Ti(m\delta_k).
  \end{align*}
  The collection $\Vi_n(m\delta_k,\delta_k,h)$ is then defined as the following equivalent class:
  \begin{align*}
    \Vi_n(m\delta_k,\delta_k,h) = \Vi^\star_n(m\delta_k,\delta_k,h) \,/ \sim_{\Ti(m\delta_k)}
  \end{align*}
  where $\sigma\sim_{\Ti(m\delta_k)}\sigma'$ if and only if there exists $\Ti_\varsigma\in\Tbb((m-1)\delta_k,\delta_k)$ such that $\sigma,\sigma'\in\Ti_\varsigma$. According to the previous definition, for any $\sigma,\sigma'$ belonging to the same equivalence class, we get $d(\sigma,\sigma')\leq 2\delta_k$. Consequently, for any $\sigma\neq\sigma'\in\Vi_n(m\delta_k,\delta_k,h)$, $d(\sigma,\sigma')\geq 2\delta_k$. The previous construction also clearly entails property \emph{(i)}.

  In order to verify the second point, we may naturally distinguish two different cases in our construction.
  \begin{enumerate}[(i)]
    \item Suppose first $\delta_k\in\ivfo{\rho_n,\vartheta_{n}}$. In order to verify $(ii)$, we make use of the estimates presented in \cite[Lemmas 4.16-4.21]{Balanca-2015b}. To start with, if $r\in\ivff{\delta_k,\vartheta_n}$, the construction described in \cite[Lemma 4.16]{Balanca-2015b} ensures that for any $\Ti_{\sigma}\in\Vbb_n(j,h)$, $\#\pthb{ \Bi(\sigma,2r) \cap \Vbb_n(j,h) } = 1$, and thus $\#\pthb{ \Bi(\sigma,2r) \cap \Vi_n(u,\delta_k,h) } = 1$ as well.
    If $r\in\ivoo{\vartheta_n,\rho_{n-1}}$, we also remark that according to the construction presented in \cite{Balanca-2015b}, for any $\sigma\in\Vi_n(u,\delta_k,h)$:
    \begin{align*}
      \forall r\in\ivff{\vartheta_n,\rho_{n-1}};\quad \ell^u(\Bi(\sigma,2r))\in\ivff{\underline{r}(r),\overline{r}(r)},
    \end{align*}
    where $\underline{r}(r) \eqdef \pthb{r \,g(r)^{1+\epsilon}}^{1/(\gamma-1)}$ and $\overline{r}(r) \eqdef \pthb{r / g(r)^{1+4\epsilon}}^{1/(\gamma-1)}$. The appropriate bound $\#\pthb{ \Bi(\sigma,2r) \cap \Vi_n(u,\delta_k,h) }$ is then a direct consequence of the previous estimates and the property $\ell^u(\Bi(\sigma,2\vartheta_n)) \geq c_0\,\vartheta_n^{1/(\gamma-1)} g(\rho_{n-1})^\epsilon$ (see \cite[Lemma 4.16]{Balanca-2015b}).

    If $r\geq \rho_{n-1}$, we use equivalent properties satisfied by the collections $\Vbb_{n}(k,h)$. Indeed, we know that the local time is similarly properly controlled at every scale, i.e. for any $\sigma\in\Vi_n(u,\delta_k,h)$
    \begin{align*}
      \forall l\in\brc{1,\dotsc,k}\ \forall r\in\ivff{\vartheta_l,\rho_{l-1}};\quad \ell^u(\Bi(\sigma,2r))\in\ivff{\underline{r}(r),\overline{r}(r)}.
    \end{align*}
    Hence, combining these precise estimates with the construction \cite[Lemma 4.16]{Balanca-2015b}, we get:
    \begin{align*}
      \forall l\in\brc{1,\dotsc,k-1}\ \forall r\in\ivff{\vartheta_l,\rho_{l-1}};\quad \#\pthb{ \Bi(\sigma,2r) \cap \Vi_n(u,\delta_k,h) } \leq \pthb{r\vartheta_n^{-1}}^{\tfrac{1}{\gamma-1}} g(r)^{-\beta},
    \end{align*}
    for some $\beta>0$ independent of $r$ and $l$. Finally, the extension to any $r\in\ivfo{\rho_l,\vartheta_l}$ is a consequence of the isolation of nodes constructed at every scale, entailing the equality:
    \begin{align*}
      \forall l\geq k-1\ \forall r\in\ivff{\rho_l,\vartheta_l};\quad \#\pthb{ \Bi(\sigma,2r) \cap \Vi_n(u,\delta_k,h) } = \#\pthb{ \Bi(\sigma,2\rho_l) \cap \Vi_n(u,\delta_k,h) }.
    \end{align*}

    \item Let now consider the second case $\delta_k\in\ivof{\vartheta_n,\rho_{n-1}}$. In order to check the second bound, we still rely on the construction properties presented in \cite{Balanca-2015b}. In particular, for any $\sigma'\in\Vi_n(m\delta_k,\delta_k,h)$ and $\sigma'\in\Ti((m+1)\delta_k)$ such that $\sigma=\iivff{\rho,\sigma'}\cap\Ti(m\delta_k)$, we know that:
    \begin{align*}
      \forall r\in\ivff{\vartheta_n,\rho_{n-1}};\quad \ell^u(\Bi(\sigma',2r)) \leq \overline{r}(r) \quad\text{and}\quad \ell^u(\Bi(\sigma',2\delta_k)) \geq \pthb{\delta_k\,g(\delta_k)^{3}}^{1/(\gamma-1)}.
    \end{align*}
    Note that the latter lower bound is more precisely a consequence of \cite[Lemma 4.16]{Balanca-2015b} and the exponential tail \cite[Lemma 3.4]{Balanca-2015b}. These tight estimates of the local time provide the first bound whenever $r\in\ivff{\delta_k,\rho_{n-1}}$.

    On the other hand, if $r\geq\rho_{n-1}$, we then simply combine the bound obtained in the first case with the second one: every element in $\Vbb_{n-1}(h)$ contains at most $\pthb{\rho_{n-1}\delta_k^{-1}}^{1/(\gamma-1)} g(\delta_k)^{-\beta}$ nodes in $\Vi_n(m\delta_k,\delta_k,h)$. In addition, we also know from the previous case that every ball $\Bi(\sigma,2r)$ contains at most $1 + \pthb{r\vartheta_{n-1}^{-1}}^{1/(\gamma-1)} g(r)^{-\beta}$ elements in $\Vbb_{n-1}(h)$ rooted at the same level. The combination of the two previous estimates then leads to the expected bound, up to a modification of the constant $\beta>0$.
  \end{enumerate}
  % \textbf{Still to be a bit improved...}
\end{proof}

\section{Proof of statement \eqref{eq:weak_asymp}}  \label{sec:appendix2}

Rather surprisingly, we could not find a presentation of statement \eqref{eq:weak_asymp} in the literature on superprocesses. Nevertheless, the latter follows easily from a set of well-known results. Hence, we quickly present in the following lemma the proof of this equality.
\begin{lemma}  \label{lemma:literature_results}
  For any fixed $t>0$,
  \begin{align*}
    \lim_{r\rightarrow 0} \frac{\log X_t(B(x,r))}{\log r} = \frac{2}{\gamma-1}\quad X_t(\dt x)\text{-a.e.} \quad \Pr_\mu\text{-a.s.}
  \end{align*}
\end{lemma}
\begin{proof}
  Let us start by presenting the proof of the lower bound. For any fixed $t>0$, according to the proof of Theorem 6.3 in \cite{Duquesne.LeGall-2005} and the well-known Frostman lemma (\cite[Th. 4.13]{Falconer-2003}), for any $s<\tfrac{2}{\gamma-1}$:
  \begin{align*}
    \limsup_{r\rightarrow 0} X_t\pthb{B(x,r)} \, r^{-s} = 0 \quad X_t(\dt x)\text{-a.e.} \quad \Pr_\mu\text{-a.s.}
  \end{align*}
  hence proving that $\liminf_{r\rightarrow 0} \tfrac{\log X_t(B(x,r))}{\log r} \geq \tfrac{2}{\gamma-1}$ $X_t(\dt x)$-a.e.

  In order to obtain the other side inequality, let us recall a property on the packing dimension \cite[Prop. 2.3]{Falconer-1997}: for a Borel set $F$ and a finite measure $\mu$,
  \begin{align*}
    \text{if }\limsup_{r\rightarrow 0} \frac{\log \mu(B(x,r))}{\log r} \geq s \text{ for all }x\in F\text{ and } \mu(F) > 0 \text{, then }\dimP F \geq s.
  \end{align*}
  Let us suppose now there exists an event $\Omega_0$ of positive probability such that for every $\omega\in\Omega_0$, there is a set $F\subset\R^d$ satisfying
  \begin{align*}
    \limsup_{r\rightarrow 0} \frac{\log X_t(B(x,r))}{\log r} \geq s \text{ for all }x\in F\text{ and } X_t(F) > 0.
  \end{align*}
  for some $s>\tfrac{2}{\gamma-1}$. Using the previous property of the packing dimension, it would clearly contradict Theorem~\ref{th:support_dimension}, and more precisely that $\dimP \supp X_t = \tfrac{2}{\gamma-1}$ on the event $X_t(R^d) > 0$, hence proving the upper bound.
\end{proof}

%%%%%%%%%%%%%%%%%%%%%%%%%%%%%%%%%%%%%%%%%%%%%%%%%%%%%%%%%%%%%%%%%%%
%% Bibliography                                                  %%
%%%%%%%%%%%%%%%%%%%%%%%%%%%%%%%%%%%%%%%%%%%%%%%%%%%%%%%%%%%%%%%%%%%
\bibliographystyle{abbrvnat-simple}

\begin{thebibliography}{44}
\providecommand{\natexlab}[1]{#1}
\providecommand{\url}[1]{\texttt{#1}}
\expandafter\ifx\csname urlstyle\endcsname\relax
  \providecommand{\doi}[1]{doi: #1}\else
  \providecommand{\doi}{doi: \begingroup \urlstyle{rm}\Url}\fi

\bibitem[Ayache and Xiao(2005)]{Ayache.Xiao-2005}
A.~Ayache and Y.~Xiao.
\newblock Asymptotic properties and {H}ausdorff dimensions of fractional
  {B}rownian sheets.
\newblock \emph{J. Fourier Anal. Appl.}, 11\penalty0 (4):\penalty0 407--439,
  2005.

\bibitem[Balan\c{c}a(2014)]{Balanca-2014}
P.~Balan\c{c}a.
\newblock Fine regularity of {L}\'evy processes and linear (multi)fractional
  stable motion.
\newblock \emph{Electron. J. Probab.}, 19:\penalty0 no. 101, 37, 2014.

\bibitem[Balan\c{c}a(2015)]{Balanca-2015b}
P.~Balan\c{c}a.
\newblock Uniform multifractal structure of stable trees.
\newblock \emph{Submitted}, pages 1--50, 2015.

\bibitem[Berestycki(2003)]{Berestycki-2003}
J.~Berestycki.
\newblock Multifractal spectra of fragmentation processes.
\newblock \emph{J. Statist. Phys.}, 113\penalty0 (3-4):\penalty0 411--430,
  2003.

\bibitem[Berestycki et~al.(2007)Berestycki, Berestycki, and
  Schweinsberg]{Berestycki.Berestycki.ea-2007}
J.~Berestycki, N.~Berestycki, and J.~Schweinsberg.
\newblock Beta-coalescents and continuous stable random trees.
\newblock \emph{Ann. Probab.}, 35\penalty0 (5):\penalty0 1835--1887, 2007.

\bibitem[Berman(1970)]{Berman-1970}
S.~M. Berman.
\newblock Gaussian processes with stationary increments: {L}ocal times and
  sample function properties.
\newblock \emph{Ann. Math. Statist.}, 41:\penalty0 1260--1272, 1970.

\bibitem[Dawson and Hochberg(1979)]{Dawson.Hochberg-1979}
D.~A. Dawson and K.~J. Hochberg.
\newblock The carrying dimension of a stochastic measure diffusion.
\newblock \emph{Ann. Probab.}, 7\penalty0 (4):\penalty0 693--703, 1979.

\bibitem[Dawson and Perkins(1991)]{Dawson.Perkins-1991}
D.~A. Dawson and E.~A. Perkins.
\newblock Historical processes.
\newblock \emph{Mem. Amer. Math. Soc.}, 93\penalty0 (454):\penalty0 iv+179,
  1991.

\bibitem[Delmas(1999)]{Delmas-1999}
J.-F. Delmas.
\newblock Path properties of superprocesses with a general branching mechanism.
\newblock \emph{Ann. Probab.}, 27\penalty0 (3):\penalty0 1099--1134, 1999.

\bibitem[Dembo et~al.(2000)Dembo, Peres, Rosen, and
  Zeitouni]{Dembo.Peres.ea-2000}
A.~Dembo, Y.~Peres, J.~Rosen, and O.~Zeitouni.
\newblock Thick points for spatial {B}rownian motion: multifractal analysis of
  occupation measure.
\newblock \emph{Ann. Probab.}, 28\penalty0 (1):\penalty0 1--35, 2000.

\bibitem[Dress et~al.(1996)Dress, Moulton, and Terhalle]{Dress.Moulton.ea-1996}
A.~Dress, V.~Moulton, and W.~Terhalle.
\newblock {$T$}-theory: an overview.
\newblock \emph{European J. Combin.}, 17\penalty0 (2-3):\penalty0 161--175,
  1996.
\newblock Discrete metric spaces (Bielefeld, 1994).

\bibitem[Duquesne(2009)]{Duquesne-2009}
T.~Duquesne.
\newblock The packing measure of the range of super-{B}rownian motion.
\newblock \emph{Ann. Probab.}, 37\penalty0 (6):\penalty0 2431--2458, 2009.

\bibitem[Duquesne and Duhalde(2014)]{Duquesne.Duhalde-2014}
T.~Duquesne and X.~Duhalde.
\newblock Exact packing measure of the range of $\psi$-super brownian motions.
\newblock 2014.

\bibitem[Duquesne and Le~Gall(2002)]{Duquesne.LeGall-2002}
T.~Duquesne and J.-F. Le~Gall.
\newblock Random trees, {L}\'evy processes and spatial branching processes.
\newblock \emph{Ast\'erisque}, \penalty0 (281):\penalty0 vi+147, 2002.

\bibitem[Duquesne and Le~Gall(2005)]{Duquesne.LeGall-2005}
T.~Duquesne and J.-F. Le~Gall.
\newblock Probabilistic and fractal aspects of {L}\'evy trees.
\newblock \emph{Probab. Theory Related Fields}, 131\penalty0 (4):\penalty0
  553--603, 2005.

\bibitem[Durand(2009)]{Durand-2009}
A.~Durand.
\newblock Singularity sets of {L}\'evy processes.
\newblock \emph{Probab. Theory Related Fields}, 143\penalty0 (3-4):\penalty0
  517--544, 2009.

\bibitem[Evans(2010)]{Evans-2010}
L.~C. Evans.
\newblock \emph{Partial differential equations}, volume~19 of \emph{Graduate
  Studies in Mathematics}.
\newblock American Mathematical Society, Providence, RI, second edition, 2010.

\bibitem[Evans et~al.(2006)Evans, Pitman, and Winter]{Evans.Pitman.ea-2006}
S.~N. Evans, J.~Pitman, and A.~Winter.
\newblock Rayleigh processes, real trees, and root growth with re-grafting.
\newblock \emph{Probab. Theory Related Fields}, 134\penalty0 (1):\penalty0
  81--126, 2006.

\bibitem[Falconer(1997)]{Falconer-1997}
K.~Falconer.
\newblock \emph{Techniques in fractal geometry}.
\newblock John Wiley \& Sons, Ltd., Chichester, 1997.

\bibitem[Falconer(2003)]{Falconer-2003}
K.~J. Falconer.
\newblock \emph{Fractal geometry}.
\newblock John Wiley \& Sons Inc., Hoboken, NJ, second edition, 2003.
\newblock Mathematical foundations and applications.

\bibitem[Fleischmann(1988)]{Fleischmann-1988}
K.~Fleischmann.
\newblock Critical behavior of some measure-valued processes.
\newblock \emph{Math. Nachr.}, 135:\penalty0 131--147, 1988.

\bibitem[Fleischmann et~al.(2010)Fleischmann, Mytnik, and
  Wachtel]{Fleischmann.Mytnik.ea-2010}
K.~Fleischmann, L.~Mytnik, and V.~Wachtel.
\newblock Optimal local {H}\"older index for density states of superprocesses
  with {$(1+\beta)$}-branching mechanism.
\newblock \emph{Ann. Probab.}, 38\penalty0 (3):\penalty0 1180--1220, 2010.

\bibitem[Frisch and Parisi(1985)]{Frisch.Parisi-1985}
U.~Frisch and G.~Parisi.
\newblock Fully developed turbulence and intermittency.
\newblock \emph{Turbulence and predictability in geophysical fluid dynamics and
  climate dynamics}, 88:\penalty0 71--88, 1985.

\bibitem[Jaffard(1999)]{Jaffard-1999}
S.~Jaffard.
\newblock The multifractal nature of {L}\'evy processes.
\newblock \emph{Probab. Theory Related Fields}, 114\penalty0 (2):\penalty0
  207--227, 1999.

\bibitem[Kaufman(1969)]{Kaufman-1969}
R.~Kaufman.
\newblock Une propri\'et\'e m\'etrique du mouvement brownien.
\newblock \emph{C. R. Acad. Sci. Paris S\'er. A-B}, 268:\penalty0 A727--A728,
  1969.

\bibitem[Khoshnevisan et~al.(2006)Khoshnevisan, Wu, and
  Xiao]{Khoshnevisan.Wu.ea-2006}
D.~Khoshnevisan, D.~Wu, and Y.~Xiao.
\newblock Sectorial local non-determinism and the geometry of the {B}rownian
  sheet.
\newblock \emph{Electron. J. Probab.}, 11:\penalty0 no. 32, 817--843, 2006.

\bibitem[Le~Gall(1999)]{LeGall-1999}
J.-F. Le~Gall.
\newblock \emph{Spatial branching processes, random snakes and partial
  differential equations}.
\newblock Lectures in Mathematics ETH Z\"urich. Birkh\"auser Verlag, Basel,
  1999.

\bibitem[Le~Gall and Le~Jan(1998)]{LeGall.LeJan-1998a}
J.-F. Le~Gall and Y.~Le~Jan.
\newblock Branching processes in {L}\'evy processes: {L}aplace functionals of
  snakes and superprocesses.
\newblock \emph{Ann. Probab.}, 26\penalty0 (4):\penalty0 1407--1432, 1998.

\bibitem[Le~Gall et~al.(1995)Le~Gall, Perkins, and
  Taylor]{LeGall.Perkins.ea-1995}
J.-F. Le~Gall, E.~A. Perkins, and S.~J. Taylor.
\newblock The packing measure of the support of super-{B}rownian motion.
\newblock \emph{Stochastic Process. Appl.}, 59\penalty0 (1):\penalty0 1--20,
  1995.

\bibitem[M{\"o}rters(2001)]{Moerters-2001}
P.~M{\"o}rters.
\newblock How fast are the particles of super-{B}rownian motion?
\newblock \emph{Probab. Theory Related Fields}, 121\penalty0 (2):\penalty0
  171--197, 2001.

\bibitem[M{\"o}rters and Shieh(2002)]{Moerters.Shieh-2002}
P.~M{\"o}rters and N.-R. Shieh.
\newblock Thin and thick points for branching measure on a {G}alton-{W}atson
  tree.
\newblock \emph{Statist. Probab. Lett.}, 58\penalty0 (1):\penalty0 13--22,
  2002.

\bibitem[M{\"o}rters and Shieh(2008)]{Moerters.Shieh-2008}
P.~M{\"o}rters and N.-R. Shieh.
\newblock Multifractal analysis of branching measure on a {G}alton-{W}atson
  tree.
\newblock In \emph{Third {I}nternational {C}ongress of {C}hinese
  {M}athematicians. {P}art 1, 2}, volume~2 of \emph{AMS/IP Stud. Adv. Math.,
  42, pt. 1}, pages 655--662. Amer. Math. Soc., Providence, RI, 2008.

\bibitem[Mytnik and Perkins(2003)]{Mytnik.Perkins-2003}
L.~Mytnik and E.~Perkins.
\newblock Regularity and irregularity of {$(1+\beta)$}-stable super-{B}rownian
  motion.
\newblock \emph{Ann. Probab.}, 31\penalty0 (3):\penalty0 1413--1440, 2003.

\bibitem[Mytnik and Wachtel(2015)]{Mytnik.Wachtel-2015}
L.~Mytnik and V.~Wachtel.
\newblock Multifractal analysis of superprocesses with stable branching in
  dimension one.
\newblock \emph{Ann. Probab.}, 43\penalty0 (5):\penalty0 2763--2809, 2015.

\bibitem[Perkins(1989)]{Perkins-1989}
E.~Perkins.
\newblock The {H}ausdorff measure of the closed support of super-{B}rownian
  motion.
\newblock \emph{Ann. Inst. H. Poincar\'e Probab. Statist.}, 25\penalty0
  (2):\penalty0 205--224, 1989.

\bibitem[Perkins and Taylor(1998)]{Perkins.Taylor-1998}
E.~A. Perkins and S.~J. Taylor.
\newblock The multifractal structure of super-{B}rownian motion.
\newblock \emph{Ann. Inst. H. Poincar\'e Probab. Statist.}, 34\penalty0
  (1):\penalty0 97--138, 1998.

\bibitem[Serlet(1995)]{Serlet-1995}
L.~Serlet.
\newblock Some dimension results for super-{B}rownian motion.
\newblock \emph{Probab. Theory Related Fields}, 101\penalty0 (3):\penalty0
  371--391, 1995.

\bibitem[Shieh and Taylor(1998)]{Shieh.Taylor-1998}
N.-R. Shieh and S.~J. Taylor.
\newblock Logarithmic multifractal spectrum of stable occupation measure.
\newblock \emph{Stochastic Process. Appl.}, 75\penalty0 (2):\penalty0 249--261,
  1998.

\bibitem[Tribe(1989)]{Tribe-1989}
R.~Tribe.
\newblock \emph{Path properties of superprocesses}.
\newblock PhD thesis, University of British Colombia, 1989.

\bibitem[Weill(2007)]{Weill-2007}
M.~Weill.
\newblock Regenerative real trees.
\newblock \emph{Ann. Probab.}, 35\penalty0 (6):\penalty0 2091--2121, 2007.

\bibitem[Xiao(2006)]{Xiao-2006}
Y.~Xiao.
\newblock Properties of local-nondeterminism of {G}aussian and stable random
  fields and their applications.
\newblock \emph{Ann. Fac. Sci. Toulouse Math. (6)}, 15\penalty0 (1):\penalty0
  157--193, 2006.

\bibitem[Xiao(2009)]{Xiao-2009a}
Y.~Xiao.
\newblock Sample path properties of anisotropic {G}aussian random fields.
\newblock In \emph{A minicourse on stochastic partial differential equations},
  volume 1962 of \emph{Lecture Notes in Math.}, pages 145--212. Springer,
  Berlin, 2009.

\bibitem[Xiao(2013)]{Xiao-2013}
Y.~Xiao.
\newblock Recent developments on fractal properties of gaussian random fields.
\newblock In \emph{Further Developments in Fractals and Related Fields}, pages
  255--288. Springer, New York, 2013.

\bibitem[Xiao and Zhang(2002)]{Xiao.Zhang-2002}
Y.~Xiao and T.~Zhang.
\newblock Local times of fractional {B}rownian sheets.
\newblock \emph{Probab. Theory Related Fields}, 124\penalty0 (2):\penalty0
  204--226, 2002.

\end{thebibliography}
\renewcommand{\bibfont}{\normalfont\small}
\setlength{\bibsep}{4pt}

\end{document}